%% file: A_numerical_measure_of_the_instability_of_Mapper.tex
\documentclass[12pt,reqno,english]{amsart}
\usepackage[margin=0.95in]{geometry}
\usepackage{abstract, amsmath, amssymb, amsthm, array,
babel,bm,
catchfilebetweentags, centernot, cite, color,
dsfont,
emptypage,
fancyhdr, float, fnpos,
graphicx,
hyperref,
ifthen,
latexsym,
mathtools,
pdfpages,
qtree,
subcaption, 	
verbatim,
wrapfig,
xcolor,
lipsum}

\usepackage{algorithm}
\usepackage[noend]{algpseudocode}

\definecolor{orcidlogocol}{HTML}{A6CE39}

\algblock{Input}{EndInput}
\algnotext{EndInput}
\algblock{Output}{EndOutput}
\algnotext{EndOutput}
\newcommand{\Desc}[2]{\State \makebox[15em][l]{#1}#2}

\usepackage[draft]{todonotes}

\usepackage{booktabs}

\usepackage[all]{xy}

\hypersetup
{
    colorlinks,	
    citecolor=red,
    filecolor=black,
    linkcolor=blue,
    urlcolor=blue
}

\makeatletter
\def\BState{\State\hskip-\ALG@thistlm}
\makeatother

\SelectTips{cm}{}

\DeclarePairedDelimiterX\setc[2]{\{}{\}}{\,#1 \;\delimsize\vert\; #2\,}

\DeclarePairedDelimiter\abs{\lvert}{\rvert}
\DeclarePairedDelimiter\norm{\lVert}{\rVert}

\makeatletter
\let\oldabs\abs
\def\abs{\@ifstar{\oldabs}{\oldabs*}}

\let\oldnorm\norm
\def\norm{\@ifstar{\oldnorm}{\oldnorm*}}
\makeatother

\newtheorem{theorem}{Theorem}[section]
\newtheorem{corollary}[theorem]{Corollary}
\newtheorem{lemma}[theorem]{Lemma}
\newtheorem{proposition}[theorem]{Proposition}

\theoremstyle{definition}
\newtheorem{definition}[theorem]{Definition}
\newtheorem{example}[theorem]{Example}
\newtheorem{remark}[theorem]{Remark}

\newtheoremstyle{theorem-w/o-number}
  {\topsep}   
  {\topsep}   
  {\itshape}  
  {0pt}       
  {\bfseries} 
  {}         
  {5pt plus 1pt minus 1pt} 
  {}         
\theoremstyle{theorem-w/o-number}
\newtheorem*{theorem*}{Theorem}
\newtheorem*{teorema*}{Teorema}
\newtheorem*{corollary*}{Corollary}
\newtheorem*{corolario*}{Corolario}
\newtheorem*{lemma*}{Lemma}
\newtheorem*{lema*}{Lema}
\newtheorem*{proposition*}{Proposition}
\newtheorem*{proposicion*}{Proposici\'on}

\newtheoremstyle{definition-w/o-number}
  {\topsep}  
  {\topsep}   
  {\normalfont}  
  {0pt}      
  {\bfseries} 
  {}         
  {5pt plus 1pt minus 1pt} 
  {}         
\theoremstyle{definition-w/o-number}
\newtheorem*{definition*}{Definition}
\newtheorem*{definicion*}{Definici\'on}
\newtheorem*{example*}{Example}
\newtheorem*{ejemplo*}{Ejemplo}\newtheorem*{remark*}{Remark}
\newtheorem*{observacion*}{Observaci\'on}

\numberwithin{equation}{section}

\let\tmp\oddsidemargin
\let\oddsidemargin\evensidemargin
\let\evensidemargin\tmp
\reversemarginpar

\newcolumntype{L}[1]{>{\raggedright\let\newline\\\arraybackslash\hspace{0pt}}m{#1}}
\newcolumntype{C}[1]{>{\centering\let\newline\\\arraybackslash\hspace{0pt}}m{#1}}
\newcolumntype{R}[1]{>{\raggedleft\let\newline\\\arraybackslash\hspace{0pt}}m{#1}}

\usepackage[font=footnotesize,labelfont=bf]{caption}

\newcommand{\coproduct}{\coprod}

\def\argmin{\qopname\relax m{argmin}}

\usepackage[latin1]{inputenc}

\title[]{A numerical measure of the instability of Mapper-type 
algorithms}
\author{Francisco Belch\'{\i}}
\email{frbegu@gmail.com, ORCID: \href{https://orcid.org/0000-0001-5863-3343}{0000-0001-5863-3343}}
\author{Jacek Brodzki}
\email{j.brodzki@soton.ac.uk}
\author{Matthew Burfitt}
\email{m.i.burfitt@soton.ac.uk, ORCID:
\href{https://orcid.org/0000-0001-7176-3476}{0000-0001-7176-3476}}
\author{Mahesan Niranjan} 
\email{mn@ecs.soton.ac.uk}
\begin{document}
\thanks{This research was supported by the EPSRC grant EP/N014189/1.}

\maketitle

\begin{abstract}
Mapper is an unsupervised machine learning algorithm generalising the notion of clustering to obtain a geometric description of a dataset.
The procedure splits the data into possibly overlapping bins which are then clustered.
The output of the algorithm is a graph where nodes represent clusters and edges represent the sharing of data points between two clusters.
However, several parameters must be selected before applying Mapper and the resulting graph may vary dramatically with the choice of parameters.

We define an intrinsic notion of Mapper instability that measures the variability of the output as a function of the choice of parameters required to construct a Mapper output. Our results and discussion are general and apply to all Mapper-type algorithms. We derive theoretical results that provide estimates for the instability and suggest practical ways to control it.
We provide also experiments to illustrate our results and in particular we demonstrate that a reliable candidate Mapper output can be identified as a local minimum of instability regarded as a function of Mapper input parameters. 
\end{abstract}

%%%%%%%%%%%%%%%%%%%%%%%%%%%%%%%%%%%%%
%%%%%%%%%%%%%%%%%%%%%%%%%%%%%%%%%%%%%
%%%%%%%%%%%%%%%%%%%%%%%%%%%%%%%%%%%%%

\section{Introduction}
The success of topological data analysis rests on the discovery, demonstrated in many groundbreaking results, that methods 
from algebraic topology can provide insight into the structure and 
meaning of complex, multidimensional data \cite{Carlsson09}. 
Mapper is a very important tool in any practical implementation of the central philosophy of topological data analysis and has been 
used with great success in many contexts. The 
list is very long and diverse, and includes breakthrough results
in medical applications such as cancer research \cite{Cecco2015, Monica2011, Nicolau2014}, the study of 
asthma \cite{Hinks2016, Schneider2016, Schofield2019, Timothy2015}, diabetes \cite{Sarikonda2014, Li2015}
and others \cite{Carlsson2017, Nielson2015, Rucco2015}.
Mapper was also applied to a variety of other disciplines, including genomic data analysis \cite{Camara2017, Rizvi2017, Chang2013, Chan2013, Bowman2008},
chemistry \cite{Duponchel2018a, LeeEtAl17}, the 
study of aqueous solubility \cite{Pirashvili2018}, remote sensing \cite{Duponchel2018b}, 
soil science \cite{Savir2017}, agriculture \cite{Kamruzzaman2017},  sport \cite{Alagappan2012} and voting pattern analysis \cite{Pek2013}.
 
Broadly speaking, the Mapper algorithm provides an approximate 
representation of the structure of the data, typically given as a point cloud, through a simplicial complex. This complex provides 
a synthesis of the main topological features of the data in the sense that 
similar data points are grouped into clusters, and clusters are connected forming loops, flares, etc.
An important step in any Mapper implementation is a choice 
of a clustering procedure that will implement the required 
notion of similarity of data points. Given 
 that all known clustering procedures display
various levels of instability \cite{vonLuxburg10_ClustStability_overview}, it is to be expected that Mapper will suffer from a similar problem, and indeed, Mapper 
instability has been well demonstrated \cite{Carriere_17}. 

Our main contribution in this paper is a numerical 
measure of the instability of Mapper as a function of its input
parameters. We demonstrate that our notion of instability 
can be used to select parameter ranges which make the corresponding 
Mapper output reliable. 

\begin{figure}[ht!]
\centering
\def\svgwidth{450pt}
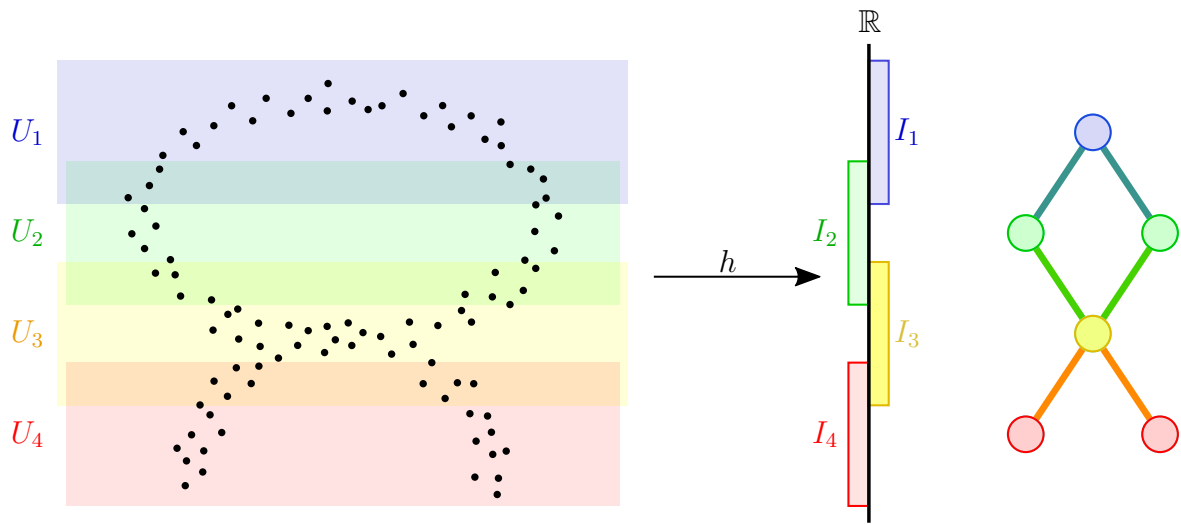
\caption{In the picture, $X$ is represented by the black dots on the left, $h$ assigns to each data point $(x, y)\in X$ its $y$ coordinate and the intervals $I^i$ are plotted as rectangles adjacent to the real line.
Mapper first clusters each group of data points $h^{-1}(I^i) \subseteq X$ and views each cluster as a node. If two clusters $c_i \subseteq h^{-1}(I^i)$, $c_j \subseteq h^{-1}(I^j)$ ($i \neq j$) share a point $x \in c_i \cap c_j$, the algorithm connects the nodes of $c_i$ and $c_j$ with an edge.
The resulting graph is the output of the classical Mapper algorithm.
Note that the resulting graph on the right looks like a simplified version of the point cloud $X$, exhibiting a hole on top and two flares at the bottom.}
\label{fig_MapperIllustration}
\end{figure}

To elucidate the problem, it is important to bear in mind that any 
practical use of Mapper on a dataset $X$ requires a number of choices. In the classical Mapper implementation, 
we need to choose a real valued function $h \colon X \longrightarrow \mathbb{R}$ (known as a \emph{filter} or a \emph{lens}) and a collection of intervals $\left\lbrace  I_i \right\rbrace _{i=1}^t$ covering $h(X)$, as can be seen in Figure \ref{fig_MapperIllustration}. The latter choice involves at least two further parameters, as we need to choose  both the length of the intervals and the amount of overlap between successive intervals. We also must choose a clustering method to apply on the bins $h^{-1}(I_i)$ to implement the required notion of similarity. 

Because of the choices involved, the creators of  Mapper remarked 
in their foundational paper \cite{Singh-Memoli-Carlsson07_foundationalMapperPaper} that  
the method is rather ad hoc, and posed the question of 
how to create a formal framework that would control the   necessary choices and would provide a measure of reliability of a particular
Mapper output. 
In this paper we provide an answer to this problem. 

%%%%%%%%%%%%%%%%%%%%%%%%%%%%%%%%%
%%%%%%%%%%%%%%%%%%%%%%%%%%%%%%%%%
%%%%%%%%%%%%%%%%%%%%%%%%%%%%%%%%%

\subsection{Contributions and related work} \label{sec:Related_work}
Following its many successful applications, several attempts have been made to reduce the number of choices required 
to create a Mapper output. 

Dey, M{\'{e}}moli and Wang \cite{Dey16,Dey17} study the structure and stability of a stable signature for what they called \emph{multiscale Mapper}, which uses a hierarchy of covers instead of a single one. However, it is not clear how to translate their findings to the context of the original Mapper.

Jeitziner, Carri\'{e}re, Rougemont, Oudot, Hess and Brisken \cite{Jeitziner17} develop a two-tier version of Mapper applied to clustering gene-expression data in order to identify subgroups.
Their version of Mapper  
is tailored specifically to the type of data for which it was 
intended and does not require any user choices. Within its intended regime, this version of Mapper is stable. 
It is not clear at this stage, however, how to extend it to 
other contexts. 

D{\l}otko \cite{Dlotko19} sets out a procedure to generate Mapper covers by balls centred around selected points in the data.
Once a cover is chosen a sequence of multiscale covers are obtained by expanding the ball sizes.

The work of Carri{\`e}re, Michel and Oudot \cite{Carriere_17} represent ideas most similar to the present paper. 
Carri{\`e}re and Oudot \cite{Carrire_Oudot16} provide bounds on the stability of Mapper in a deterministic setting on manifolds by comparing it to the Reeb graph.
This is achieved though a feature set obtained from an extended persistence diagram of the Mapper graph with respect to the filter function.
In particular, the features correspond to loops and flairs in Mapper graph.
Through further statistical analysis \cite{Carriere_17},
bounds are determined on the expectation of the bottleneck distance between the features of the Mapper and Reeb graphs, assuming points are sampled from and underlying manifold.  
This provides a way to obtain confidence regions for features on the persistence diagram that may be used to identify reliable Mapper outputs.

Our approach provides a more general setting than that of \cite{Carriere_17}.
Points are only assumed to be sampled from an underlying probability distribution rather than a distribution on a smooth manifold.
Furthermore the required covers may be chosen arbitrarily rather than being restricted to arising from an interval cover and filter function.

In particular, our approach will account for the size of features in terms of cluster size, not just their presence. This is an important improvement over methods relying on persistent homology, 
where cluster size is ignored. 
This new idea allows us to study the effects of the 
choice of a clustering algorithm, which can even be picked to be different on different
parts of the cover. This possible variability in the clustering procedure as well as any inherent instability of the chosen 
clustering procedure have not been investigated so far and we fill that gap here. 

Despite the ubiquity of clustering techniques within
 unsupervised learning, it has proved difficult to establish a good theoretical foundation for this methodology. A lot 
 of effort has been devoted to the study of quality and 
 stability of clustering. Highlights include the famous
 impossibility theorem of Kleinberg
 \cite{Kleinberg03}, who proved that there is no clustering procedure satisfying all of his natural axioms. This was 
 taken up by Carlsson and M\'{e}mmoli \cite{Carlsson10}, who proposed an axiomatic approach allowing them to provide an existence and uniqueness result for single-linkage clustering.
More recently, Strazzeri and S{\'{a}}nchez{-}Garc{\'{\i}}a \cite{Strazzeri18} provided a clustering procedure that satisfies Kleinberg's axioms after an alteration of the consistency axiom.

The work of Ackerman and Ben-David \cite{Ben-David09} studied clustering quality measures rather than the clustering functions, which provides a richer setting in which an alternative to Kleinberg's axioms can be consistently stated.

In a similar vein, instability provides 
a measure of reliability
of a particular output for the choice of 
input parameters. In particular, it will 
identify regions in the parameter space 
where the output is very sensitive to the changes of parameter values and so 
is typically  less reliable. 
Much effort has been invested in studying clustering stability and while the theoretical principles are agreed upon, at present there is no standard implementation to determine its value. For an overview see \cite{vonLuxburg10_ClustStability_overview}.
In particular, methods of data perturbation and resampling have been successful in practice, for instance in the biomedical setting \cite{Bittner2000, BenHur02, Levine2001}.
Resampling methods such as bagging \cite{Breiman1996,breiman1998} have also long been successfully applied within supervised leaning.
A procedure using resampling methods and statistics derived form the Mapper algorithm \cite{Riihimaki2019} has also been used to obtain very acurate classification results on tree species data.

The most comprehensive theoretical study of clustering stability 
by Ben-David and von Luxburg \cite{BenDavid_vonLuxburg08} defined a notion of clustering stability and related it to properties of the decision boundaries of the algorithm.
This is the starting point of the theoretical part of this work. We extend these notions to account for the considerably more complex Mapper construction.

This paper is organised as follows. 
In \S \ref{sec:Related_work}, we discuss some related work and its connections to the current paper.
In \S \ref{section:Clustering_on_each_term_of_Cover}, we give background on clustering stability required for the remainder of the paper. 
This allows us in \S \ref{section:Distances_of_Mapper_networks} to set out how the ideas of Ben-David and von Luxburg \cite{BenDavid_vonLuxburg08} can be generalised to the Mapper setting.
In particular, we introduce Mapper functions in Definition \ref{def:Mapper_functions}, which provide a new way of expressing Mapper outputs.
Crucially, this is used to define a similarity metric between Mapper functions, $D_M$ in Definition \ref{def:D_Mapper}.
The Distance $D_M$ captures the structure of the whole Mapper output and leads to the definition of our notion of instability of Mapper (Definition \ref{def:Mapper_instability}) with respect to a large class of clustering procedures.
In \S \ref{sec:ComputingInstability} we present an algorithm allowing us to experimentally obtain values of instability.
This leads in section in \S \ref{sec:InitialResults} to interesting experimental results which suggest that regions of relatively high instability correspond to structural changes in the Mapper output.
Hence local minima of the instability function with respect to parameter choices are good candidates for parameter selection allowing us to study Mapper through variations of all the parameters.

In the remainder of the paper, we develop theoretical tools to provide bounds on the instability of Mapper and to 
understand the main contributing factors. 
To do this, in \S \ref{sec:BounderyDistance} we introduce 
another similarity measure $D_\partial$,  Definition \ref{def:D_boundary}. 
The pseudo distance $D_\partial$ can be seen as a kind of interleaving distance, and it
relates the instability to the Mapper cover, enabling us to obtain useful bounds in \S \ref{sec:Upper_bounding_instability_of_Mapper},
Theorems \ref{thm:Bound_depending_mostly_on_gamma} and \ref{thm:Bound_depending_on_cover_sampleSize_etc}.  
These theoretical results unravel the main reasons for the 
instability, which are summarised  in Remarks \ref{rmk:Reasons_which_produce_instability_1} and \ref{rmk:Reasons_which_produce_instability_2}.
In \S \ref{sec:Sharpness_of_upper_bound}, we study how to sharpen the bounds on instability obtained in \S \ref{sec:Upper_bounding_instability_of_Mapper} and prove in Theorem \ref{thm:MapperStability} that for a large enough sample size and under reasonably constrained conditions these bounds can be arbitrarily small.
Implying that the Mapper instability under such conditions is also small.
This means that Theorem \ref{thm:MapperStability} might be seen as a kind of stability theorem for Mapper and justifies the central observations of \S \ref{sec:InitialResults}.
In \S \ref{sec:Verifying}, we present a number of experiments demonstrating our theoretically derived reasons for instability and explain how the reasons for instability cause the behaviour observed in \S \ref{sec:InitialResults}.

%%%%%%%%%%%%%%%%%%%%%%%%%%%%%%%%%
%%%%%%%%%%%%%%%%%%%%%%%%%%%%%%%%%
%%%%%%%%%%%%%%%%%%%%%%%%%%%%%%%%%

\section{Clustering stability}\label{section:Clustering_on_each_term_of_Cover}

The question of assessing the quality and stability of clustering procedures has attracted a lot of attention in recent years. In our discussion of Mapper stability, we will build on the foundational work on 
clustering stability by Ben-David and von Luxburg \cite{BenDavid_vonLuxburg08}. Therefore, we begin by introducing our setting in similar terms to theirs.

By a \emph{clustering} of a metric space $\left(U, D \right)$ we will mean a partition of $U$ into $s$ disjoint subsets or clusters. Equivalently, we may think of a clustering as a function from $U$ to a finite set of labels. In 
assessing the performance of a particular clustering procedure, the choice of labels to denote the clusters
will typically be unimportant, which motivates the following definition. 
\begin{definition}\label{def:Clustering}
	Let $\left(U, D \right)$ be a metric space and let $F$ denote the set of all functions	$f: U \longrightarrow \left\lbrace 1, 2, \ldots, s\right\rbrace$. 
	Then a clustering of $\left(U, D \right)$ is an element of
	\begin{equation*}
		\mathcal{F} \coloneqq {F}/{\sim},
	\end{equation*}
	where $f \sim g$ if there is a permutation $\pi$ of the set $\left\lbrace 1, 2, \ldots, s\right\rbrace $ 
	such that $f = \pi g$.
\end{definition}

To assess the efficiency of a particular clustering procedure we need a clustering quality function, which assigns a
notional cost or error to a clustering procedure. The objective of a clustering procedure is then to minimise the 
cost. Let $M_1(U)$ denote the space of all probability measures on $U$ (with
respect to the Borel $\sigma$-algebra).
For the purposes of this paper, a clustering quality function is a function which assigns a real number (the cost)
to a choice of clustering and a choice of a probability measure on $U$. In other words, a clustering 
quality function is a map
\begin{equation}
	Q \colon \mathcal{F} \times M_1(U) \longrightarrow \mathbb{R}.
\end{equation}

\begin{example}
	To make the previous statement more transparent, consider the $K$-means clustering. In this case, $Q(g, P)$ measures the expected distance between any point drawn according to the probability distribution $P$ and the cluster centre assigned to that point by the clustering function $g$. We give  the explicit formula for this quality function in (\ref{eq:kMeans_ContinuousObjectiveFunction}).
\end{example}

\begin{definition}\label{def:OptimalCluster}
	Given a probability measure $P \in M_1 \left( U \right)$, the \emph{optimal} clustering of $U$ is defined as the function $f\in \mathcal{F}$
	which minimizes $Q(-, P)$:
	\begin{equation}\label{eq:Argmin_Q^i}
		f = \argmin_{g \in \mathcal{F}} Q(g, P).
	\end{equation}
    The optimal clustering gives rise to a clustering map
	\begin{equation}\label{eq:OptimalClusteringFunction1}
		C \colon M_1(U) \to \mathcal{F},\quad P \mapsto \argmin_{g \in \mathcal{F}} Q(g, P).
	\end{equation}
\end{definition}

The clustering $f$ in Definition \ref{def:OptimalCluster} is only well defined if $Q(\cdot, P)$ has a unique global minimum, which will be our assumption in this paper. A main reason for this restriction is that in this 
work we want to understand the relation between the user-selected parameters of the input and the 
stability of the outcome. In the presence of more local minima of the quality function $Q(\cdot,P)$, 
clustering instability may be dominated by other phenomena, for example, the symmetry of the data. This case 
will be discussed in the follow-on work. In fact, as demonstrated by  \cite[Theorem 4]{BenDavidEtAl07}, $K$-means is stable if and only if there is a unique global minimiser, so this assumption is quite reasonable. More generally, in \cite[Theorem 15]{BenDavidEtAl06},
it is proved that multiple global minimisers with symmetry imply instability.

When working on a finite sample of $X= ( X_1, \ldots, X_n ) \in U^n$, we use another clustering quality function
\begin{equation}\label{eq:clustering_quality_function_on_finite_samples}
    Q_n: \mathcal{F}_n \times U^n \longrightarrow \mathbb{R},
\end{equation}
which we call the empirical quality function.
Unless stated otherwise, we assume that the quality function does not depend on the order of $X_1,\dots,X_n$. 

\begin{example}
	The empirical $K$-means quality function for $K=s$ clusters on a finite sample $X=( X_1, \ldots, X_n )  \in U^n$
	computes the average distance between points in the sample and their corresponding cluster centroid
	\begin{equation*}
		\label{eq:kMeans_EmpiricalObjectiveFunction}
		Q_n(f, X) = \frac{1}{n} \sum_{j=1}^n \sum_{k=1}^K \mathds{1}_{f(X_j) = k} D(X_j, c_k),
	\end{equation*}
	where $D(X_j, c_k)$ denotes the distance between the point $X_j \in U$ and the cluster centre $c_k$
	and $\mathds{1}_{f(X_j) = k}$ is an indicator function,
	$$ \mathds{1}_{f(X_j) = k} = 
	\begin{cases}
	1, \qquad f(X_j) = k \\
	0, \qquad f(X_j) \neq k .
	\end{cases}
	$$
	The continuous counterpart of $Q_n$ for $K$-means clustering is given by:
	\begin{equation*}\label{eq:kMeans_ContinuousObjectiveFunction}
		Q(f, P) = \sum_{k=1}^K \int_{x \in U} \mathds{1}_{f(x) = k} D(x, c_k) \,dP(x).
	\end{equation*}
\end{example}

\begin{remark}
    In practice, the clustering quality function and empirical quality function are related. Intuitively, $Q_n$ is a discretised version of $Q$, and we will make the additional assumption that
    $Q_n$ is uniformly consistent with $Q$ in the following sense. 
    For every $\gamma>0$, $ Q_n(f^n,X) \xrightarrow[n \rightarrow \infty ]{} Q(f, P)$ in probability, uniformly over probability distributions $P \in M_1(U)$.
    More precisely,
    $\forall \epsilon >0, \forall \delta>0, \exists N \in \mathbb{N}$ such that $\forall n \geq N, \forall P \in M_1(U)$, 
    \begin{equation*}
        P^n \left( \lvert Q_n(f^n, X) - Q(f, P) \rvert > \epsilon \right) \leq \delta.
    \end{equation*}
\end{remark}

\begin{definition}\label{def:OptimalFiniteCluster}
	Let $\mathcal{F}_n^X$ denote the space of clusterings of $X$.
	Given a point sample $X=( X_1, \ldots, X_n ) \in U^n$,
	define the \emph{optimal} empirical clustering $f \in \mathcal F_n^X$ of $U$  as 
	\begin{equation}\label{eq:Argmin_Q^i_n}
		f = \argmin_{g \in \mathcal{F}_n} Q_n(g, X),
	\end{equation}
	if $n \geq 1$ and set $f$ to be constant for $n=0$.
    The optimal empirical clustering gives rise to a clustering map
	\begin{equation}\label{eq:OptimalClusteringFunction}
		C_n \colon U^n \to \mathcal{F}_n^U, X \mapsto \argmin_{g \in \mathcal{F}_n} Q_n(g, X)
	\end{equation}
	where $\mathcal{F}_n^U$ is the union of all $\mathcal{F}_n^X$ for $X\in U^n$.
\end{definition}

Similarly to Definition \ref{def:OptimalCluster}, the clustering $f$ of definition \ref{def:OptimalFiniteCluster} may not exist.
In addition, even if such a global minimum exists, it may not be computable by the clustering algorithm.
For example, the empirical clustering quality function for the $K$-means clustering (\ref{eq:kMeans_EmpiricalObjectiveFunction})
need not have a global minimum.
However, nearest neighbour clusterings \cite{vonLuxburgEtAl08} or approximation schemes \cite{OstrovskyEtAl06} have empirical quality function
(\ref{eq:Argmin_Q^i_n}) with a unique global minimum and algorithms to compute them. For the theoretical part of this work, we will assume that $Q_n(-, X)$ has a unique global minimum.

We will need to be able to compare clusterings and 
for that we now recall the minimal matching distance. This is 
one of many measures of similarity developed for clusterings, and 
a good survey on this subject can be found in 
 \cite{Meila05}.

\begin{definition}\label{def:MinMatchDistence}
	The minimal matching distance is a map $D_m:\mathcal{F}_n \times\mathcal{F}_n \longrightarrow \mathbb{R}$ 
	that, for any two clusterings $f,g \in \mathcal{F}_n$
of a set of points $X = (X_1,\dots,X_n),$ is defined by
	\begin{equation*}\label{eq:D_Mapper_i}
		D_m(f, g) = \min_{\pi} \frac{1}{n} \sum_{j=1}^n \mathds{1}_{f(X_j) \neq \pi g(X_j)},
	\end{equation*}
	where $\pi$ runs over all permutations of the set $\left\lbrace 1, 2, \ldots, s\right\rbrace $ and
	$\mathds{1}_{f(X_j) \neq \pi g(X_j)}$ is an indicator function.
\end{definition}

It is well known that $D_m$ is a metric, and that it can be computed efficiently using a minimal bipartite matching algorithm.
Given a distance between clusterings of finite samples, we may define the instability with respect to an empirical quality function and a distance.
Here we consider the instability with respect to the minimal matching distance.

Any clustering $g\in\mathcal{F}_n^X$ on a finite point sample $X = ( X_1, \ldots, X_n ) \in U^n$ can be extended to a clustering $g\in\mathcal{F}$ on all of $U$  in the following  fashion.
Consider the order $X_1 \leq X_2 \leq \ldots \leq X_n$, and denote by $V_i$ the Voronoi cell of $X_i$, defined by
\begin{equation}\label{eq:Voronoi}
	V_ i \coloneqq \left\lbrace y \in U | D(y,X_i) \leq D(y, X_j) \; \forall \; j>i \text{ and } D(y,X_i) < D(y, X_k) \forall \; k<i \; \right\rbrace .
\end{equation}
Note that $\{V_i\}_i$ forms a partition of $U$.
In order to extend the clustering $g\in\mathcal{F}_n^X$ to $U$, we can simply assign the label $g(X_i)$ to all points of $U$ in the Voronoi cell of the point $X_i$. Which is, we extend $g\in\mathcal{F}_n^X$ so that it is constant on each Voronoi cell.

Given an empirical quality function $Q_n$, using the clustering function $C_n$ of Definition \ref{def:OptimalFiniteCluster}
and the minimal matching distance, we obtain the composition
\begin{equation}\label{eq:InstabilityVariable}
	\mathcal{I}(Q_n) \colon U^n \times U^n \xrightarrow{C_n \times C_n}\mathcal{F}_n^U \times \mathcal{F}_n^U \xhookrightarrow{i}
	\mathcal{F}_{2n}\times \mathcal{F}_{2n} \xrightarrow{D_m} \mathbb{R},
\end{equation}
where $\mathcal{F}_n^U$ is the union of all clustering functions $\mathcal{F}_n^X$ on $X$ for every $X\in U^n$. To define the inclusion $i$, we extend 
a clustering of $n$ points to a clustering of all of $U$ via Voronoi cells as just explained and focus only on the labels assigned to subsets of $2n$ points.

We would like the function $\mathcal{I}(Q_n)$ to be a random variable with respect to the probability measure on $U^n$,
induced by a probability measure on $U$.
From now on we restrict to quality functions such that $\mathcal{I}(Q_n)$ is a random variable, which we justify in the appendix.

\begin{definition}\label{def:Instability}
	Let  $(U,D)$ be a metric space equipped with an $n$-point  clustering quality function $Q_n \colon \mathcal{F}_n \times U^n \longrightarrow \mathbb{R}$ and a probability measure $P\in M_1(U)$. Then the clustering 
instability is given by
	\begin{equation*}
		\text{\rm{InStab}}_\text{\rm{Clustering}}(Q_n,P) = \mathbb{E}\left( \mathcal{I}(Q_n) \right),
	\end{equation*}
	where the expectation is taken over probability product measures of $P$ on pairs of $n$-samples in $U^n \times U^n$.
\end{definition}

%%%%%%%%%%%%%%%%%%%%%%%%%%%%
%%%%%%%%%%%%%%%%%%%%%%%%%%%%
%%%%%%%%%%%%%%%%%%%%%%%%%%%%

\section{Comparing Mapper functions}
\label{section:Distances_of_Mapper_networks}
We now pass to the main part of this work. Our first goal is to provide a description of Mapper functions analogous to the 
representation of clusterings as functions introduced
 in Definition  \ref{def:Clustering}. A key part of our construction 
is a generalization of the minimal matching distance given in Definition \ref{def:MinMatchDistence} to a form suitable for 
comparing  Mapper outputs. The extension works by taking into account the clustering information contained in the resulting complexes. Our new notion of distance between Mapper functions is then used to define instability of the Mapper procedure and to derive upper bounds for this instability in 
\S \ref{sec:Upper_bounding_instability_of_Mapper}.

Let $\left( \mathcal{X}, D\right)$ be a metric space
and let $\mathcal{U}= \left\lbrace U_i\right\rbrace _{i=1}^t$ be a cover of $\mathcal{X}$, that is $ \mathcal{X}=\bigcup_{i=1}^t U_i$. Following standard Mapper terminology, we refer to the sets $U_i$ as \emph{bins}.
In the classical Mapper algorithm, these bins are obtained by fixing a real valued function $h:\mathcal{X} \longrightarrow \mathbb{R}$ (known as a \emph{filter function} or a \emph{lens}), fixing a collection of intervals $\left\lbrace I_i\right\rbrace _{i=1}^t$ covering $h(\mathcal X)$, and setting $U_i \coloneqq h^{-1} (I_i)$, as in Figure \ref{fig_MapperIllustration}.
Here, however, we do not assume, as we do not need to, that the 
cover $\{U_i\}_{i=1}^t$ of $X$ is of this particular form. 

In this paper, we will deal with a discrete and finite
sample $X$ drawn from a metric space $\mathcal{X}$. 
The cover $\{U_i\}_{i=1}^t$
of $\mathcal{X}$ restricts to a cover $\{U_i \cap X\}_{i=1}^t$ of the space $X$, and
we will simply write $U_i$ rather than $U_i \cap X$ to lighten the notation. We now use a clustering 
procedure to cluster each of the sets $U_i$, so that 
we have 
\[
U_i = V^i_1 \cup \dots \cup V_s^i. 
\]
A \emph{Mapper output} is a simplicial complex where an 
$n$-simplex $\sigma$ is an $(n+1)$-tuple of clusters
\[
\sigma = (V^{i_1}_{j_1}, \dots, V^{i_{n+1}}_{j_{n+1}})
\]
with a nonempty intersection.

To avoid the labels of clusters in $U_i$ being mixed up with those of $U_j$ for $i\neq j$,
we cluster each $U_i$ separately, that is, a clustering of $U_i$ is of the form
\begin{equation*}\label{eq:CoverClusters}
    f_i: U_i \longrightarrow \left\lbrace c^i_1, c^i_2, \ldots, c^i_s \right\rbrace,
\end{equation*}
where the $c^i_j$ are cluster labels.
Similarly to \S \ref{section:Clustering_on_each_term_of_Cover},
denote by $F^i$ the collection of all functions of the form (\ref{eq:CoverClusters}) and
$\mathcal{F}^i = F^i/\sim$, where
\begin{equation*}
    f_i \sim g_i \Longleftrightarrow \exists \pi: f_i = \pi g_i,
\end{equation*}
with $\pi$ denoting some permutation of the set $\left\lbrace c^i_1, c^i_2, \ldots, c^i_s \right\rbrace $.
To simplify the notation, we assume that every $U_i$ is partitioned into the same number $s$ of clusters.
However, all results hold when choosing a different $s$ for each bin.

Given a probability measure on $\mathcal X$, $P\in M_1(\mathcal{X})$, we consider the probability measure induced on $U_i$ by restricting $P$ to $U_i$ and setting
\begin{equation}\label{eq:Restricted_probability}
    P_i = \frac{1}{P(U_i)} \cdot P \in M_1 \left( U_i \right),
\end{equation}
and setting $P_i$ as the zero measure if $P(U_i)=0$.
Denote by $Q^i: \mathcal{F}^i \times M_1(U_i) \longrightarrow \mathbb{R}$ the clustering quality function used in $U_i$, and denote by 
\begin{equation}\label{eq:EmpiricalMapperQuality}
    Q_n^i: \mathcal{F}^i \times U_i^n \longrightarrow \mathbb{R}
\end{equation}
its empirical counterpart on size-$n$ samples of $U_i$.
As in Definition \ref{def:OptimalCluster}, the clustering quality function $Q^i$ determines a unique optimal clustering 
for each set $U_i$, and taken together, these optimal solutions
create an optimal Mapper output and a clustering function $C^i\colon M_1(U_i) \to \mathcal{F}^i$, for every $i= 1, \dots, t$. In a similar way,  
Definition \ref{def:OptimalFiniteCluster} and 
an empirical quality function $Q^i_n$ 
determines a unique optimal empirical 
clustering for each $U_i$ from which we 
obtain an optimal Mapper output and a 
clustering function $C_n^i\colon U_i^n \to \mathcal{F}^{U_i}_n$.
        
\begin{remark}\label{rmk:MapperParameters}
    As is now apparent, a Mapper output (as well as a Mapper function which we will discuss shortly) depends on the choice of a cover, a quality function as well as the particular 
    sample drawn from the ambient metric space. Moreover, implicit in the 
    choice of a quality function is a choice of a clustering procedure. We will refer to these choices collectively as \emph{Mapper parameters}. 
    In practice, these various choices usually come down to a list of real parameters. 
    For example, in the standard Mapper algorithm, the cover $U_i$ is the pullback 
    of an interval cover of $\mathbb{R}$, which is specified through 
    a choice of two parameters, resolution and gain. In this case, resolution  
    is the number and size of intervals in the cover,
    while gain controls the size of the overlap of these intervals. 
\end{remark}

\begin{definition}\label{def:Mapper_functions}
Let $\mathcal{X}$ be a metric space equipped with a cover $U_i$.    Given a clustering $f_i \in \mathcal{F}^i$ for each member $U_i$ of the cover
    we define the corresponding \emph{Mapper function} as the function which assigns
    to each $x\in \mathcal{X}$, the set of clustering labels 
    given to $x$ by the clustering functions $f_i$, for $i=1, \dots, t$. In other words, we have
    \begin{equation*}
    f(x) = \left\lbrace f_i(x) \,|\, i=1,\dots,t, \; x\in U_i \right\rbrace,
    \end{equation*}
    for each $x \in \mathcal{X}$.
    We denote the set of all Mapper functions on $(\mathcal{X},\{U_i\}^t_{i=1})$ by $\mathcal{N}$ and $\mathcal{N}_n$ on a finite $n$-point sample $X\in \mathcal{X}^n$.
\end{definition}

Note that for each $x\in\mathcal{X}$, the size of $f(x)$ depends only on the cover, since it is equal to the number of sets $U_i$ that contain $x$. 

Notice as well that a Mapper function contains more information than a Mapper output, which is an abstract simplicial complex constructed on the set of clusters. 
A Mapper function contains the information about the number of points 
in each cluster, and also in every nonempty intersection of those clusters. 
A Mapper output will be equivalent to a Mapper function if we label every simplex $\sigma = (V_0, V_1, \dots, V_k)$ of the Mapper output by the number of points of $X$
contained in the intersection 
\[
V_0\cap \dots \cap V_k
\]
of the clusters that are the vertices of $\sigma$. 

Let $X=(X_1, \ldots, X_n)\in \mathcal{X}^n$ be a point sample of $\mathcal{X}$.
Then, for each $1 \leq i \leq t$,
denote 
\begin{equation*}
    X^i = \left\lbrace X_1, \ldots, X_n\right\rbrace  \cap U_i
\end{equation*}
and let $n_i=n_i(X)$ be the number of elements in $X^i$.
We now introduce a Mapper version of definition (\ref{def:MinMatchDistence}).

\begin{definition}\label{def:D_Mapper}
    Given a point sample $X=(X_1,\dots,X_n) \in \mathcal{X}^n$ of $\mathcal{X}$, we define a distance function
    $D_M:\mathcal{N}_n \times \mathcal{N}_n \longrightarrow \mathbb{R}$
    which, for any two Mapper functions $f, g \in \mathcal{N}_n$, is given by
    \begin{equation*}
        D_M(f, g) = \min_{\pi} \frac{1}{n} \sum_{j=1}^n \mathds{1}_{f(X_j) \neq \pi g(X_j)},
    \end{equation*}
    where $\pi = \bigoplus_{i=1}^t \pi^i$, each $\pi^i$ runs over all permutations of
    $\left\lbrace c_1^i,\dots,c_s^i \right\rbrace $, and
    $\mathds{1}_{f(X_j) \neq \pi g(X_j)}$ is an indicator function.
\end{definition}

\begin{remark}\label{rmk:D_MM-counts-edge-info-as-well}
    For two Mapper functions $f,g$ on $\mathcal{X}$ covered by $\{ U_i \}_{i=1}^t$,
    the matching distance $D_m(f_i, g_i)$ of definition \ref{def:MinMatchDistence} counts the proportion of points of $X^i$ for which $f_i$ and $g_i$ disagree. 
    Since the clustering of each $U_i$ corresponds to the vertices of the Mapper output,
    considering $D_m$ on each $U_i$ would give no information about the higher dimensional simplicies of the Mapper output.
    However,
    Definition \ref{def:D_Mapper} takes into account not only the points that fall into different vertices of
    $f$ and $g$, but also all the edges and the higher dimensional simplices to  which the Mapper functions $f$ and $g$ assign different values. 
    
    A drawback of $D_M$ it that it can see certain intuitively larger changes of vertex labeling as equally distant.
    Consider the following example. Assume that $\mathcal{X}$ is covered by 
    three sets $U_1, U_2, U_3$ and that each of these sets is clustered into two clusters 
     labeled $c^i_1, c^i_2$ for $i=1,2, 3$. Let $f(x)=\{ c^1_1,c^2_1,c^3_1 \}$, $g(x)=\{ c^1_1,c^2_1,c^3_2 \}$, $h(x)=\{ c^1_1,c^2_2,c^3_2 \}$, 
    and $f(y)=g(y)=h(y)$ for all other points $y\neq x$.
    Provided the clustering labels remain unchanged, then $D_M(f,g)=D_M(f,h)$,
    despite the fact that $h$ differs  from $f$ on two clusters, and it 
    differs from $g$ on only one cluster. 
    However $D_M$ does has the advantages of taking into account edge information and being simple to work with from both from a theoretical and a practical perspective.
\end{remark}

Since $D_M$ generalizes $D_m$, we use $D_M$ to generalize Definition (\ref{def:Instability}) to a notion of instability of Mapper.
As before, we assume that the metric space $\mathcal{X}$ is equipped 
with a cover $\{U_i\}_{i=1}^t$. We choose an empirical quality function 
$Q^i_{n_i}$ for each $U_i$.
Denote by $\mathcal{N}^{X}_n$ the set of all Mapper functions on $X\in \mathcal{X}^n$ with cover $\{ X^i \}_{i=1}^t$
and $\mathcal{N}^{\mathcal{X}}_{n}$ the union of $\mathcal{N}^{X}_n$ for all choices of $X$.
Each $Q^i_{n_i}$ determines a corresponding 
clustering function $C^i_{n_i}$ which we use to 
define an instability  function $\mathcal{I}(\{Q^i_{n_i}\}^t_{i=1})$ 
as the composite
\begin{equation}\label{eq:InstabilityVariableMapper}
	\mathcal{I}(\{Q^i_{n_i}\}^t_{i=1}) = \mathcal{X}^n \times \mathcal{X}^n \xrightarrow{\coproduct_{i=1}^t C_{-}^i \times C^i_{-}}\mathcal{N}_n^{\mathcal{X}} \times \mathcal{N}_n^{\mathcal{X}} \xhookrightarrow{i}
	\mathcal{N}_{2n} \times \mathcal{N}_{2n} \xrightarrow{D_M} \mathbb{R}.
\end{equation}
The function $\mathcal{I}(\{Q^i_{n_i}\}^t_{i=1})$ will be measurable if and only if each $\mathcal{I}(Q^i_{n_i})$ is measurable. This is because it follows from the definitions of $D_M$ and $D_m$ that the pre-image of a measurable set under $\mathcal{I}(\{Q^i_{n_i}\}^t_{i=1})$
is a union of the the pre-images of measurable sets for functions $\mathcal{I}(Q^i_{n_i})$.

\begin{definition}\label{def:Mapper_instability}
    Fix  Mapper parameters on $\mathcal X$ by choosing  quality functions $\{ Q^i_{n_i}\}^t_{i=1}$ defined on a cover $\left\lbrace U_i\right\rbrace _{i=1}^t$ of $\mathcal X$,  and a  probability measure $P\in M_1(\mathcal{X})$. 
    These choices are made so that $\mathcal I =\mathcal{I}(\{Q^i_{n_i}\}^t_{i=1})$ is a random variable, as discussed at the end of \S  \ref{section:Clustering_on_each_term_of_Cover}. 
    
    The instability of the Mapper algorithm on size-$n$ samples is defined as
    \begin{equation*}
        \text{\rm{InStab}}_\text{\rm{Mapper}}(\{ Q_{n_i}^i \}^t_{i=1},n, P) =
        \mathbb{E}\left( \mathcal{I}(\{Q^i_{n_i}\}^t_{i=1}) \right),
    \end{equation*}
  where the expectation is taken over the probability product measures of $P$ on pairs of $n$-samples in $\mathcal{X}^n \times \mathcal{X}^n$.
\end{definition}

%%%%%%%%%%%%%%%%%%%%%%%%%%%%%%%%%%%%%

%%%%%%%%%%%%%%%%%%%%%%%%%%%%%%%%%%%%%

%%%%%%%%%%%%%%%%%%%%%%%%%%%%%%%%%%%%%

\section{Computing Mapper instability}\label{sec:ComputingInstability}

In this section, we present a procedure for experimentally estimating the Mapper instability given in Definition~\ref{def:Mapper_instability}.
It is important to note that there is no standard procedure to determine clustering instability, and a discussion of the subject can be found in \cite{vonLuxburg10_ClustStability_overview}.
Our approach is to generalise to the Mapper setting a method for computing clustering instability detailed in \cite{BenHur02}, which is based on sub-sampling of the data.

To begin, we assume that all necessary Mapper parameters, as explained in Remark \ref{rmk:MapperParameters}, have been selected and that we have
a sample of $n$ points taken independently and identically distributed (i.i.d) from an underlying probability distribution.
Then we may computationally estimate the Mapper instability
based on the method of $k$-fold cross validation
as follows.
\begin{enumerate}
    \item
        Split the data into $k$ sub-samples. That is, choose $m,k \in \mathbb{N}$ such that $n=km$ and remove for each $1\leq i \leq k$ the $m$ points $m(i-1)+1$ to $mi$, leaving $k$ sub-samples of $(k-1)m$ points. 
    \item
        Compute the Mapper distance between the Mapper functions of each pair of sub-samples, on the $(k-2)m$ points of their intersection.
    \item
        Average the distances between Mapper functions restricted to the  sub-samples by summing the distances and dividing by $\frac{(k+1)k}{2}$.
\end{enumerate}

The outcome of this procedure is an approximation of
the instability of the Mapper function. 

Choosing a small $k$ leads to inaccurate results since there are too few samples and the intersection between the samples is small.   
However too large a choice of $k$ may result in samples that are too similar which in turn 
decreases the speed of computation 
as many more distances need to be calculated.
Hence the best results are achieved with values of $k$ and $m$ in the middle of their range, such that $m$ is not too large.
Greater accuracy can still be gained by averaging the results of the procedure applied to several randomly shuffled copies of the dataset.
We now  explain the details of the procedure for computing the Mapper distance between the Mapper functions on two sub-samples.

Given a dataset $X$, we describe in Algorithm \ref{alg:MapperMismatch} a procedure to compute $n$ times the Mapper distance $D_M(f,g)$ between two Mapper functions $f,g\in\mathcal{N}^X$ on a cover $\left\lbrace U_i\right\rbrace_{i=1}^t$ of $X$.
We denote by
\begin{equation*}
    \{ c_i^1,\dots,c_i^{k_i} \} \text{ and } \{ s_i^1,\dots,s_i^{k_i} \}
\end{equation*}
the clusters of $f$ and $g$ in each $U_i$ respectively,
where $k_i$ is the maximum number of clusters of either $f$ or $g$ in each $U_i$.
If $k_i$ is larger than the number of clusters, then the additional clusters are assumed to be empty.
Additionally, with $l=\sum_{i=1}^t{k_i}$ let,
\begin{equation}\label{eq:OrderedClusters}
    (c_{\zeta_1}^{\eta_1},\dots,c_{\zeta_{l}}^{\eta_{l}})
\end{equation}
be a size-ordered list of clusters of $f$, that is
$|c_{\zeta_1}^{\eta_1}|\geq |c_{\zeta_2}^{\eta_2}| \geq \dots \geq |c_{\zeta_l}^{\eta_l}|$.

Algorithm \ref{alg:MapperMismatch} is a recursive backtracking procedure, which is
initialized with an upper bound, and a possible 
choice here is the  total number of points in the sample.
However, we will indicate shortly how to significantly improve this choice which will 
greatly shorten the computation time.

The mismatch between two clusters $c_i^a$ and $s_i^b$ is the symmetric difference $c_i^a \triangle s_i^b$ of the sets, consisting of  the points that are elements of one of the sets but not the other.
Algorithm \ref{alg:MapperMismatch} takes in order each cluster of $(c_{\zeta_1}^{\eta_1},\dots,c_{\zeta_{l}}^{\eta_{l}})$, and looks for the first cluster of $g$ that has not yet been matched. We compute any additional 
mismatch that arises from any new matching. 
If the total mismatch exceeds the upper bound the
algorithm backtracks and looks for a better matching.
We replace the upper bound if a better one is obtained.
Ordering the clusters in (\ref{eq:OrderedClusters}) is therefore a good idea
because obtaining a large mismatch is only possible if at least one of the clutters is large.
If we obtain a large mismatch quickly, this reduces the execution time of the algorithm by reducing the number of possibilities that need 
to be checked. 

\begin{algorithm}
    \caption{Obtains $n$ times Mapper distance $D_M(f,g)$ between Mapper functions $f$ and $g$}
    \label{alg:MapperMismatch}
    \begin{algorithmic}
        \Input
            \Desc{$(c_{\zeta_1}^{\eta_1},\dots,c_{\zeta_{l}}^{\eta_{l}})$}{Size ordered list of cluster from $f$}
            \Desc{\textit{Bound}}{Upper bound on the Mapper distance}
            \Desc{$p$ $\gets$ $1$}{Cluster position $p$ in $(c_{\zeta_1}^{\eta_1},\dots,c_{\zeta_{l}}^{\eta_{l}})$}
            \Desc{\textit{Mismatch} $\gets$ $\emptyset$}{Set of mismatched points}
            \Desc{$U_i$-\textit{matches} $\gets$ $\{ s_i^1,\dots,s_i^{k_i} \}$}{Clusters of $g_i$ not yet matched with $f_i$ clusters}
        \EndInput
        \Output
            \Desc{$nD_M(f,g)$}{$n$ times Mapper distance between $f$ and $g$}
        \EndOutput
        \Procedure{Distance}{}
            \For {each member $S$ of $U_{\zeta_p}$-\textit{matches}}
                \State \textit{Mismatch} $\gets$ \textit{Mismatch} $\cup$ ($c_{\zeta_p}^{\eta_p} \triangle S$)
                \If {$|$\textit{Mismatch}$| \; <$ \textit{Bound}}
                    \If {$p$ $=$ $l$}
                        \State \textit{Bound} $\gets$ $|$\textit{Mismatch}$|$
                    \Else {}
                        \State $U_p$-\textit{matches} $\gets$ $U_p$-\textit{matches} $-\;S$
                        \State $p$ $\gets$ $p+1$
                        \State \textit{Bound} $\gets$
                        \Call{Distance}{$(c_{\zeta_1}^{\eta_1},\dots,c_{\zeta_{l}}^{\eta_{l}})$,  \textit{Bound},  $p$, \textit{Mismatch}, $U_i$-\textit{matches}}
                    \EndIf
                \EndIf
            \EndFor
            \Return $Bound$
        \EndProcedure
    \end{algorithmic}
\end{algorithm}

A drawback of Algorithm \ref{alg:MapperMismatch} is that despite executing significantly faster than a procedure that considers all cluster permutations, computation time can still be slow.
The main reason for this is that if the initial upper bound is large, improved bounds may only be obtained in small increments, requiring most permutations to be checked.

A very good upper estimate for the Mapper distance can be obtained by finding the permutations corresponding to the minimal matching distances $D_m(f_i,g_i)$ within each clustering of $U_i$.
Then finding the size of the set of mismatched points across the Mapper functions corresponds to the permutation obtained by combining the optimal permutations in each $U_i$.

In practice,  this upper bound can be obtained by performing Algorithm \ref{alg:MapperMismatch} restricted to each clustering on $U_i$ and
returning the  minimal $Mismatch$ in addition to the corresponding $Bound$.  
An upper bound is then given by the size of the union of the mismatches from each $U_i$.
Alternatively the optimal permutation within each $U_i$ could be obtained using the Hungarian algorithm.

%%%%%%%%%%%%%%%%%%%%%%%%%%%%%%%%%%%%%

\section{Initial experimental results}\label{sec:InitialResults}

In this section we demonstrate how the procedure detailed in the previous section
might be used to determine good Mapper outputs over varying parameter selections. 
Mapper is a standard tool from topological data analysis and there are several available implementations \cite{PythonMapper,TDAmapper,KeplerMapper}.
We obtained our results using the Kepler Mapper  \cite{KeplerMapper}.

{\centering
\begin{table}[ht!]
	\begin{tabular}{cc}
		\begin{subfigure}{0.48\textwidth}\centering\includegraphics[width=1\columnwidth]{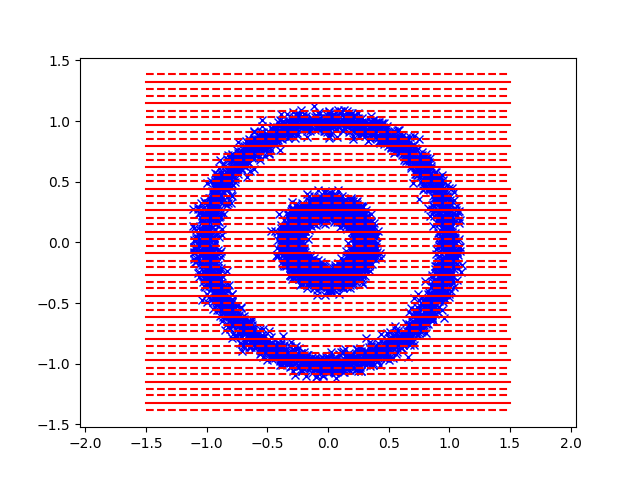}
		\end{subfigure}&
		\begin{subfigure}{0.48\textwidth}\centering\includegraphics[width=1\columnwidth]{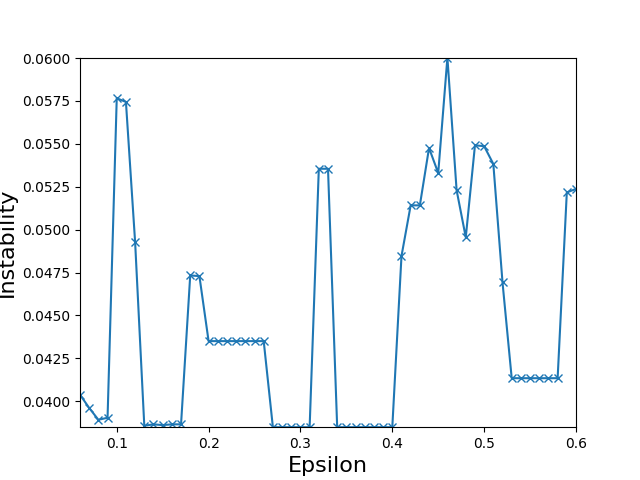}
		\end{subfigure}\\
	\end{tabular}
	\begin{tabular}{cccc}
	        \begin{subfigure}{0.18\textwidth}\centering\includegraphics[width=1\columnwidth]{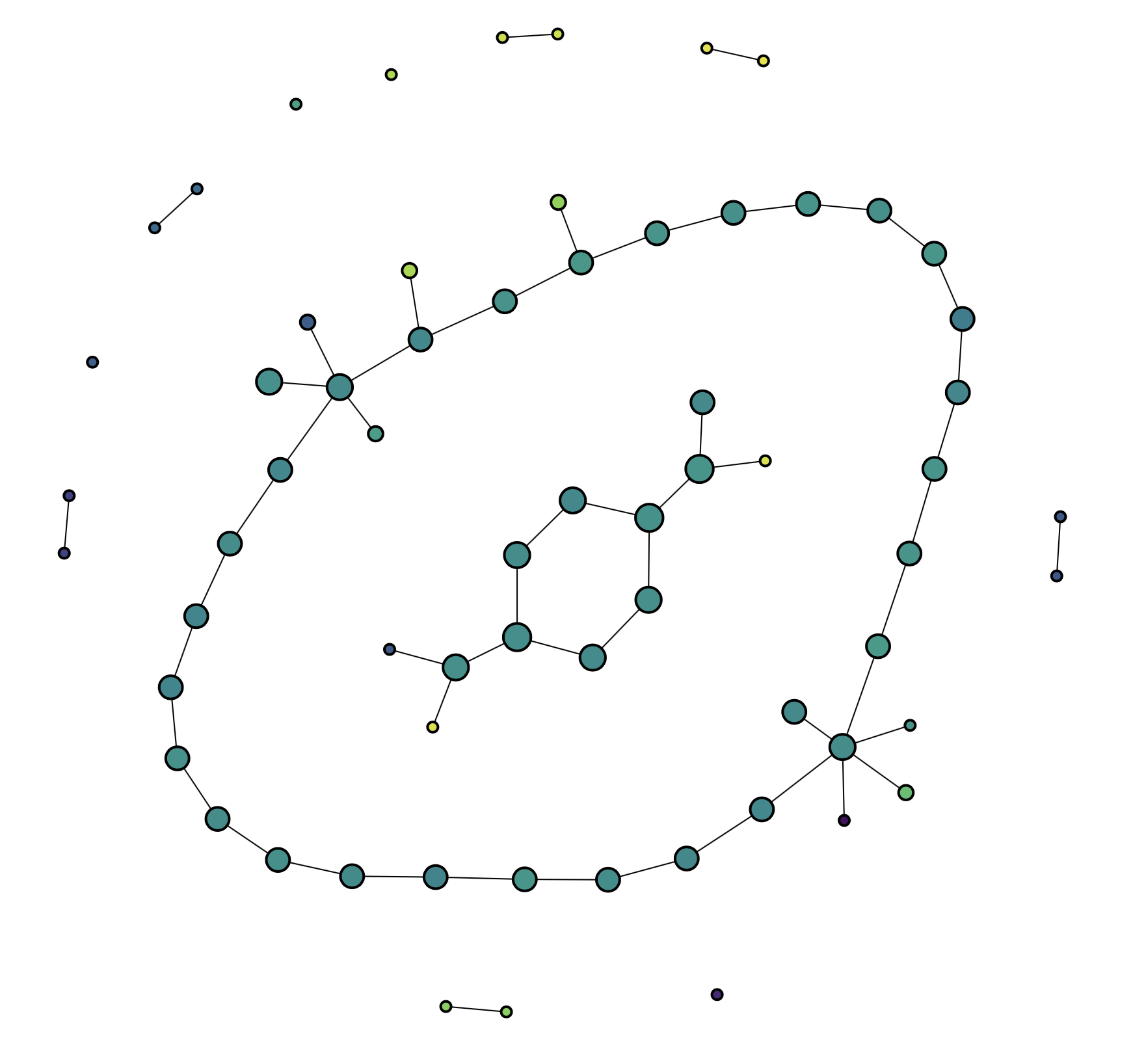}\caption*{Epsilon 0.06}
			\end{subfigure}&
			\begin{subfigure}{0.18\textwidth}\centering\includegraphics[width=1\columnwidth]{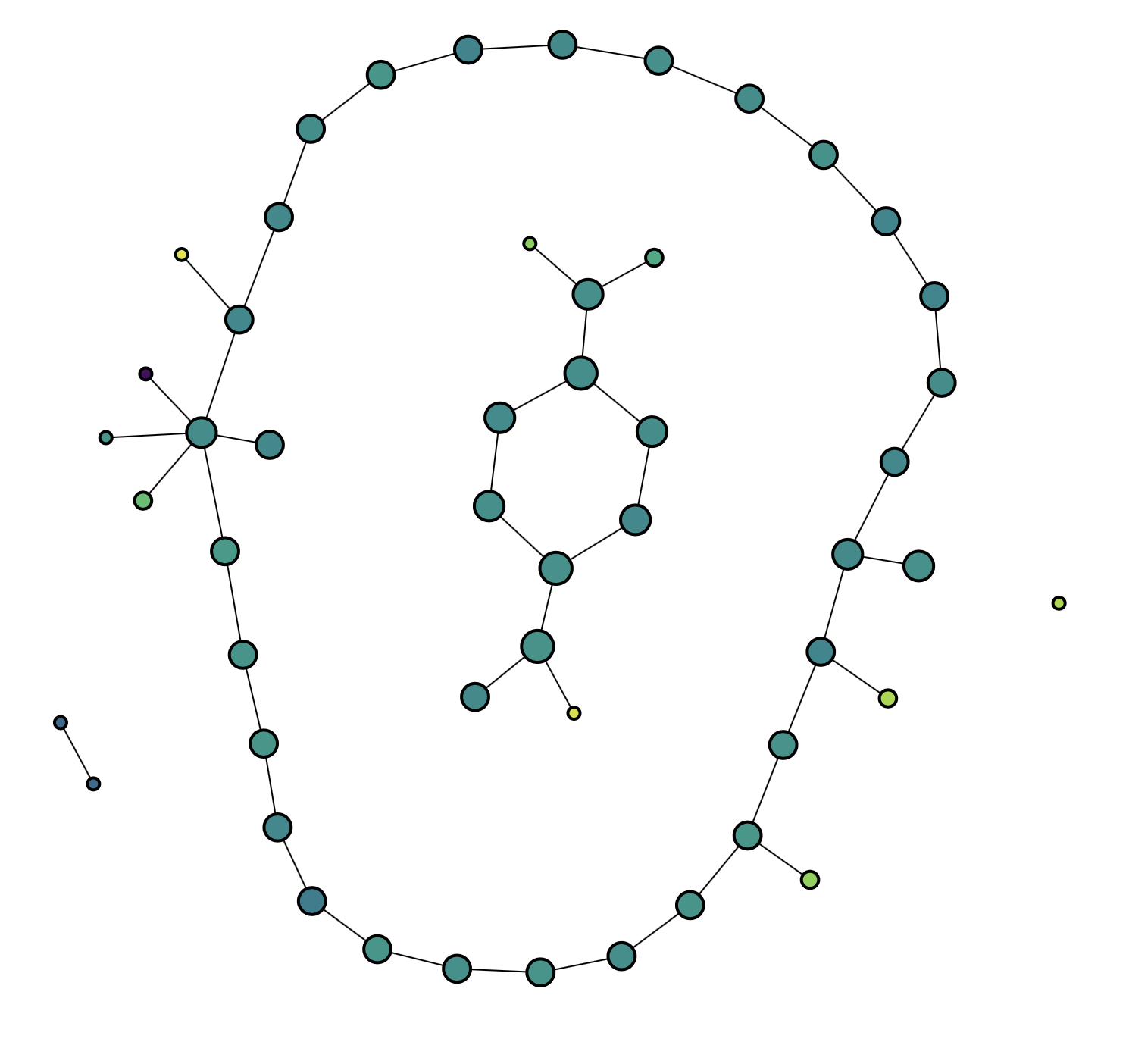}\caption*{Epsilon 0.07}
			\end{subfigure}&
			\begin{subfigure}{0.18\textwidth}\centering\includegraphics[width=1\columnwidth]{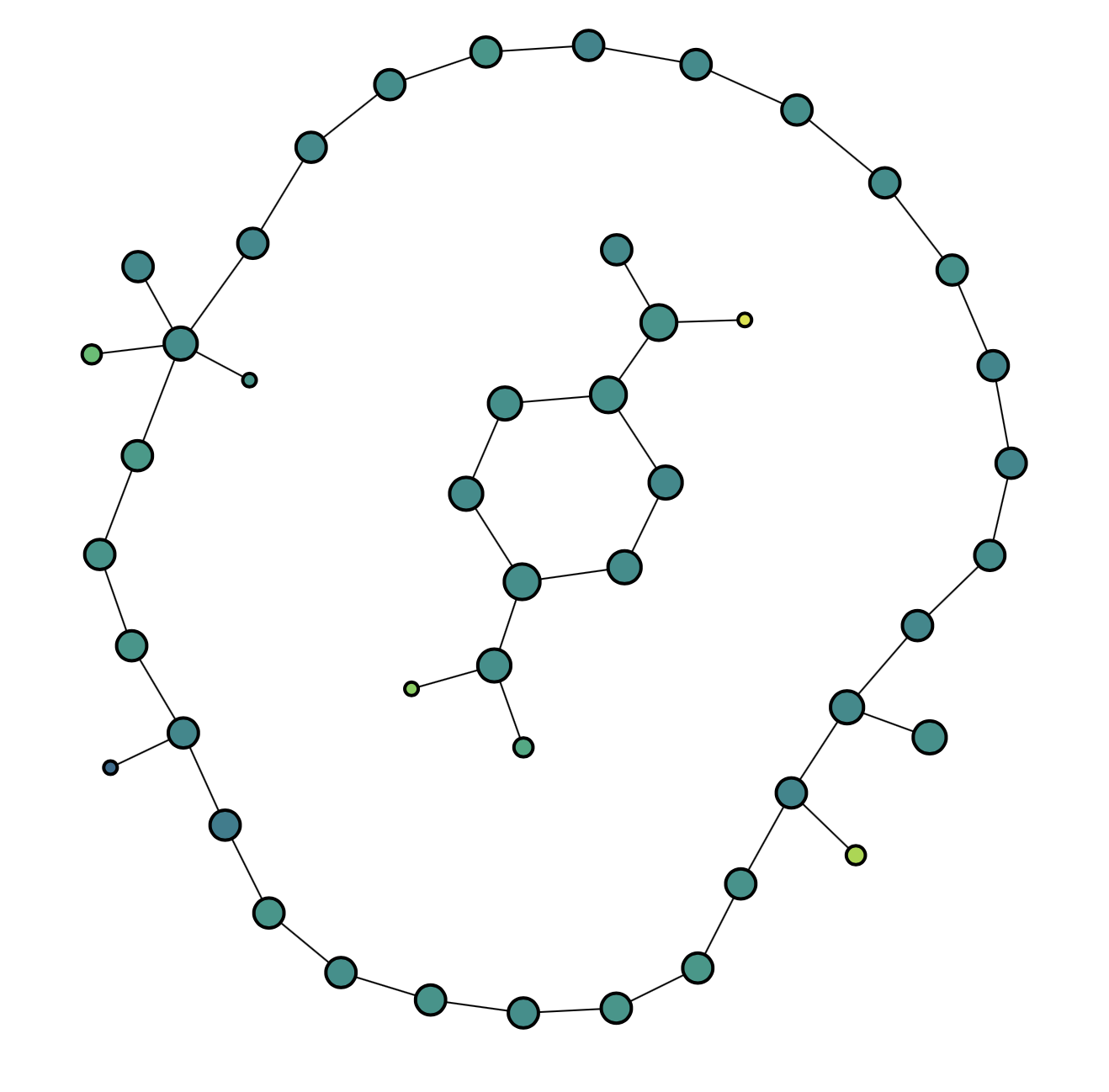}\caption*{Epsilon 0.09}
			\end{subfigure}&
			\begin{subfigure}{0.18\textwidth}\centering\includegraphics[width=1\columnwidth]{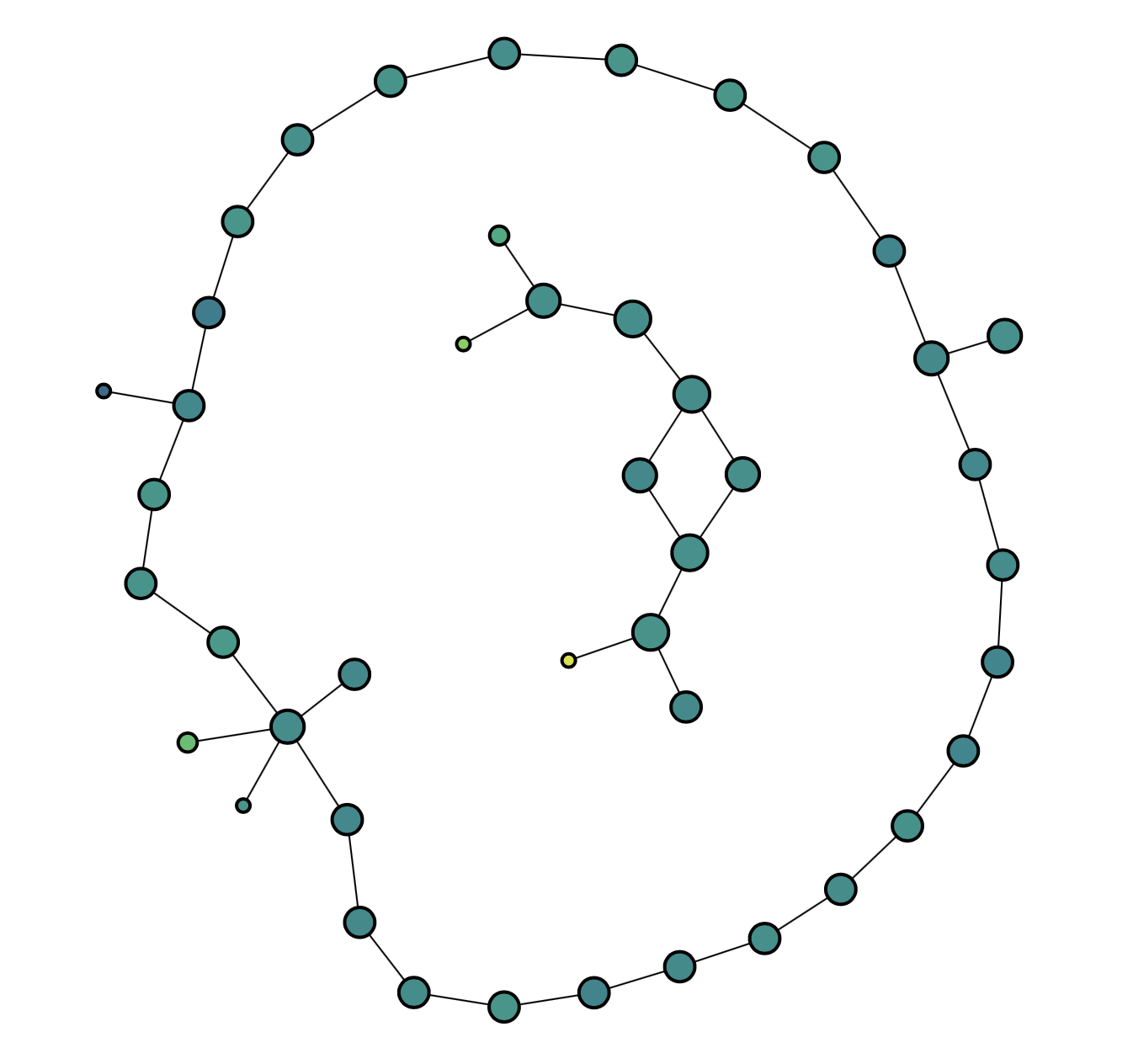}\caption*{Epsilon 0.1}
			\end{subfigure}\\
			\newline
			\begin{subfigure}{0.18\textwidth}\centering\includegraphics[width=0.8\columnwidth]{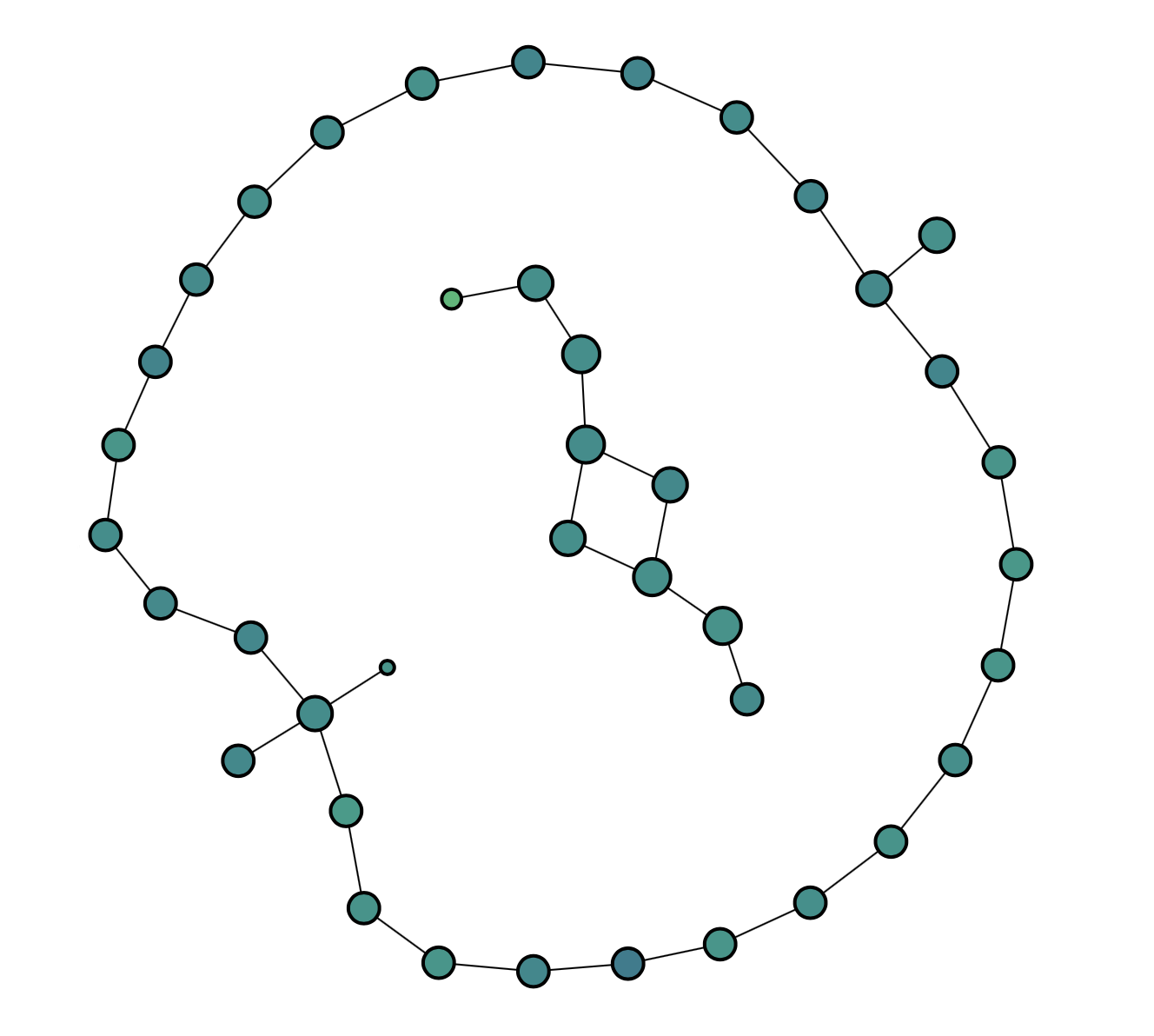}\caption*{Epsilon 0.15}
			\end{subfigure}&
			\begin{subfigure}{0.18\textwidth}\centering\includegraphics[width=0.8\columnwidth]{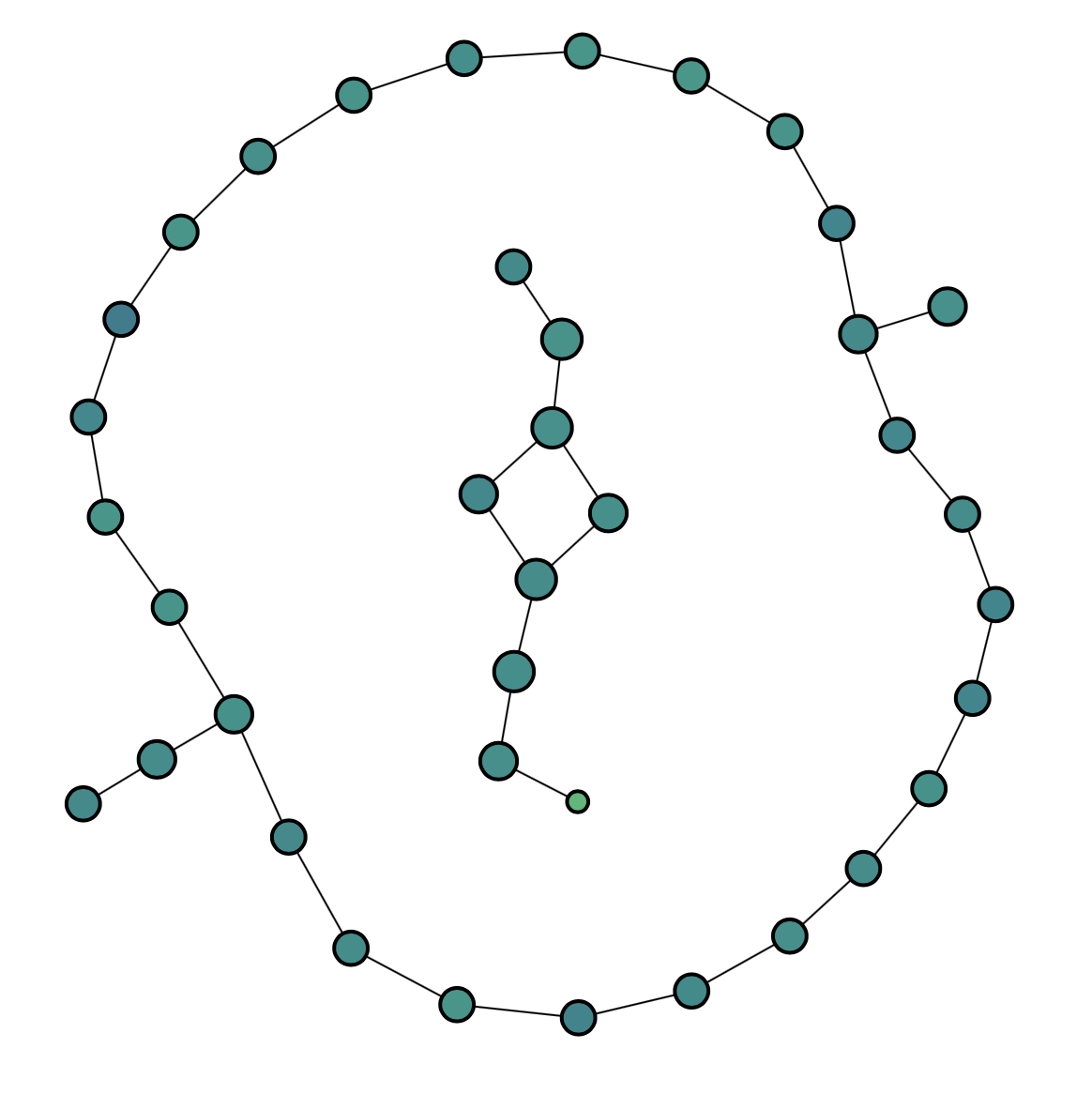}\caption*{Epsilon 0.2}
			\end{subfigure}&
			\begin{subfigure}{0.18\textwidth}\centering\includegraphics[width=0.8\columnwidth]{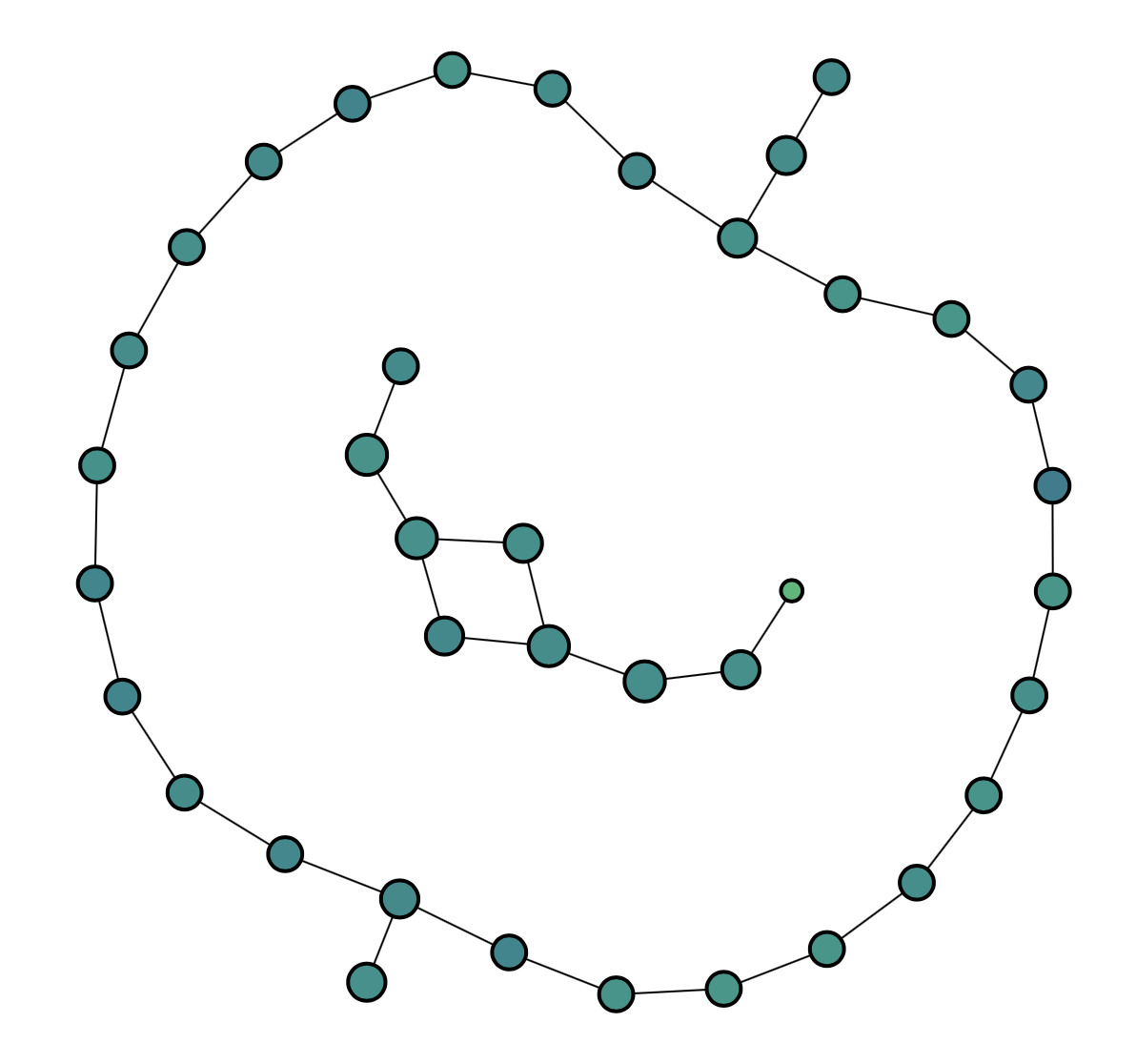}\caption*{Epsilon 0.3}
			\end{subfigure}&
			\begin{subfigure}{0.18\textwidth}\centering\includegraphics[width=0.8\columnwidth]{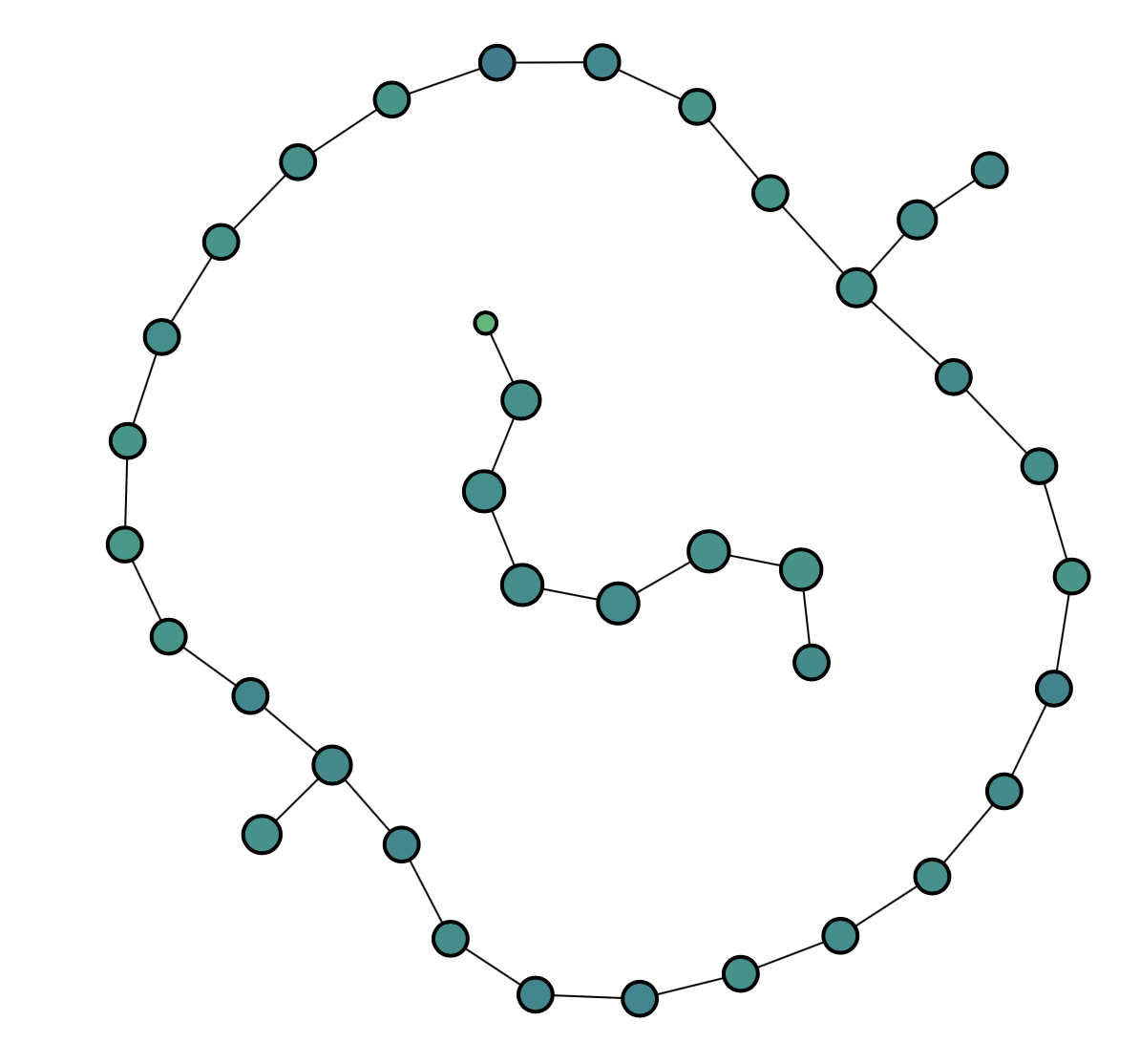}\caption*{Epsilon 0.35}
			\end{subfigure}\\
			\newline
			\begin{subfigure}{0.18\textwidth}\centering\includegraphics[width=0.7\columnwidth]{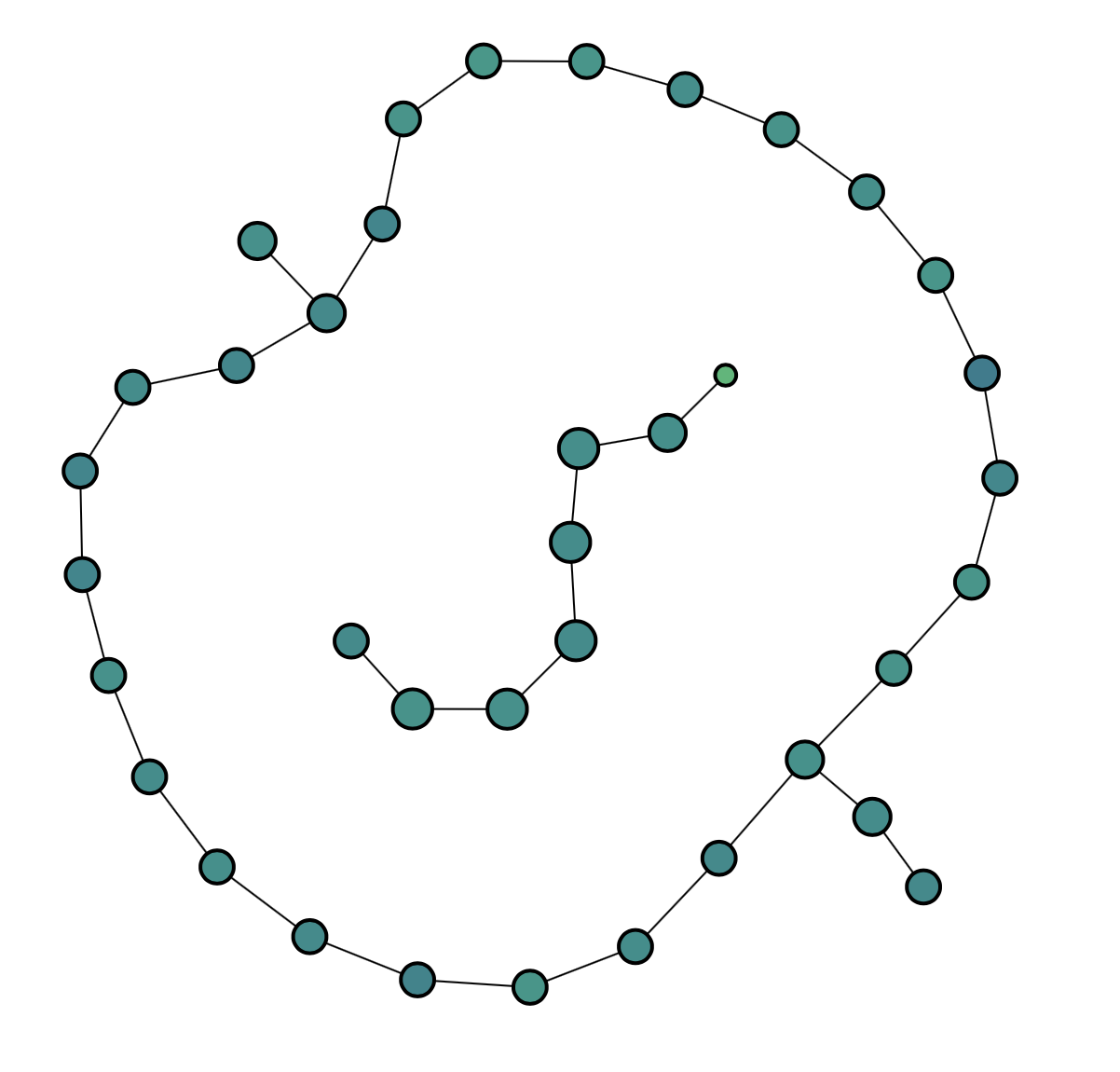}\caption*{Epsilon 0.4}
			\end{subfigure}&
			\begin{subfigure}{0.18\textwidth}\centering\includegraphics[width=0.7\columnwidth]{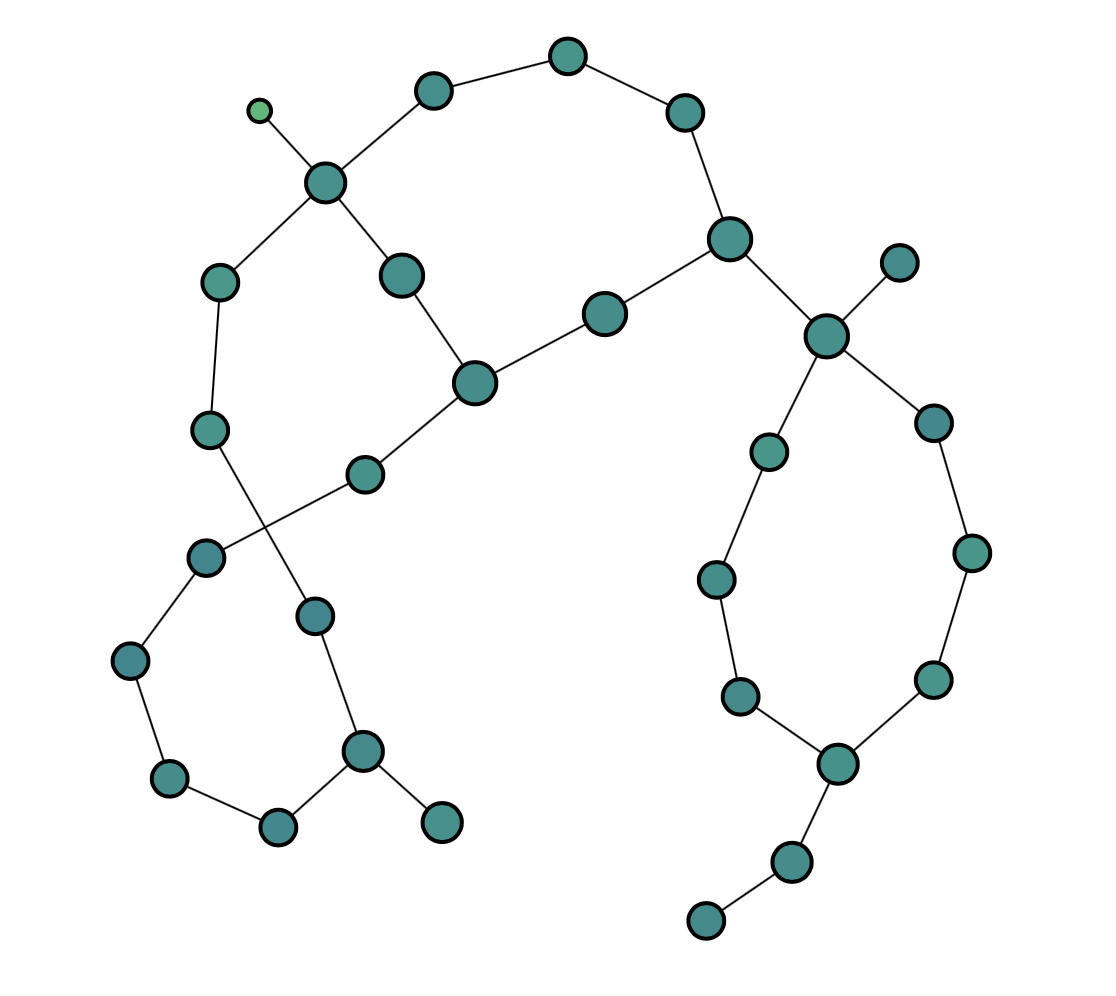}\caption*{Epsilon 0.45}
			\end{subfigure}&
			\begin{subfigure}{0.18\textwidth}\centering\includegraphics[width=0.7\columnwidth]{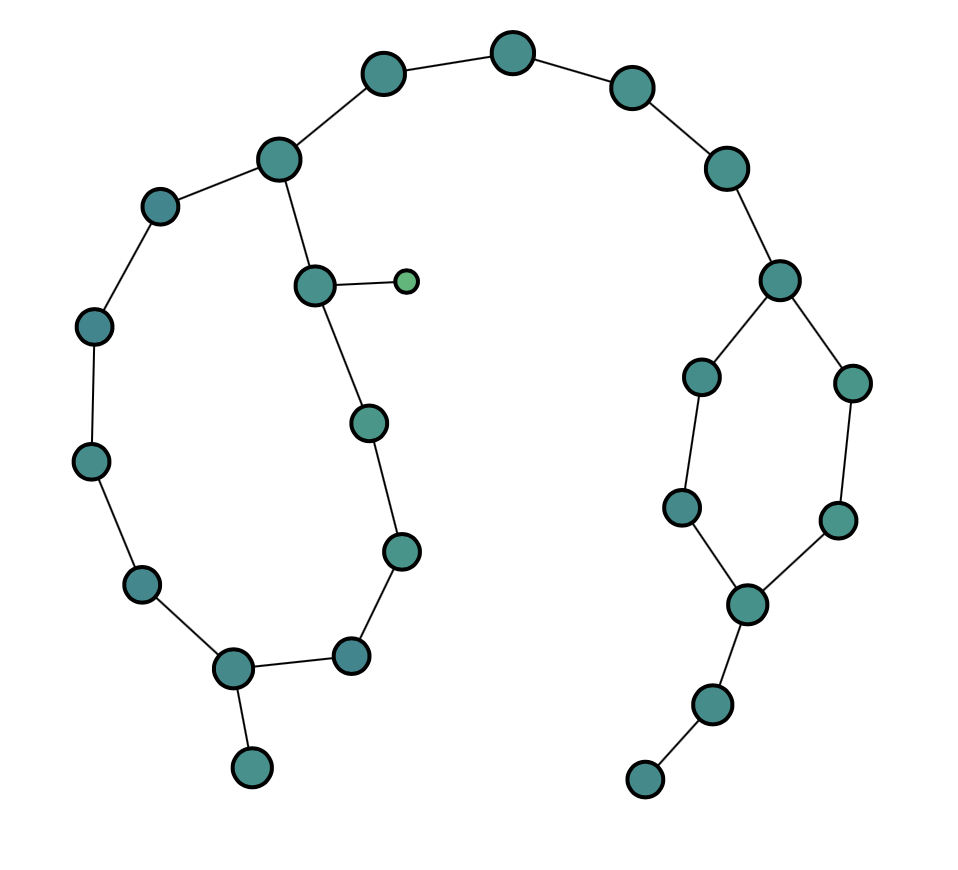}\caption*{Epsilon 0.5}
			\end{subfigure}&
			\begin{subfigure}{0.18\textwidth}\centering\includegraphics[width=0.7\columnwidth]{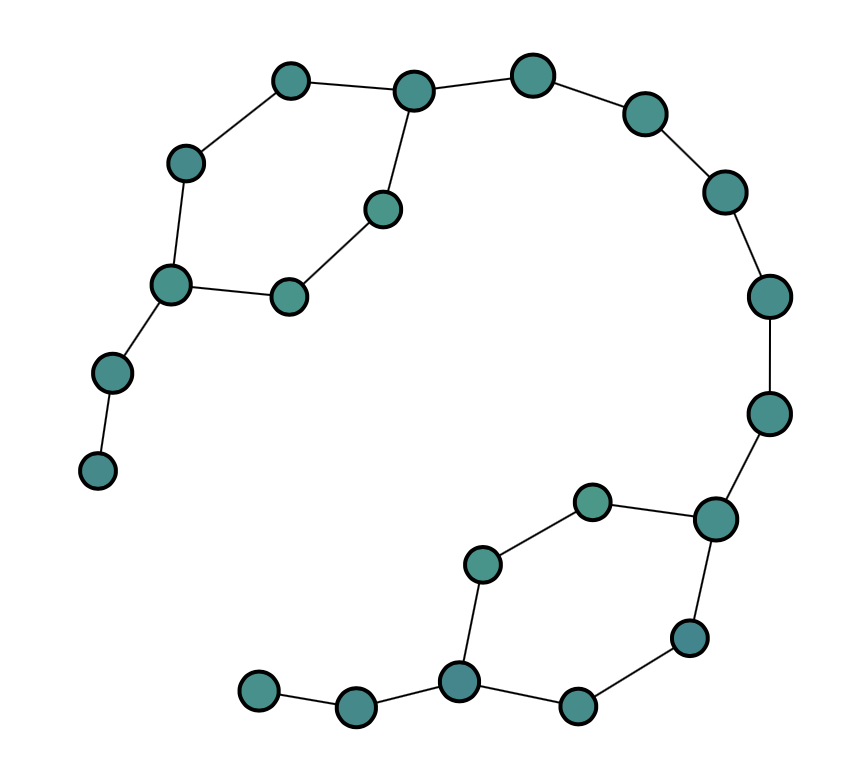}\caption*{Epsilon 0.6}
			\end{subfigure}\\
		\end{tabular}
	\caption{The panel in the top left shows a dataset of $5000$ points sampled with noise from two concentric circles. The dashed red lines denote the boundaries of the overlapping bins and the solid lines are the centres of the overlap.  Displayed below are the Mapper graphs corresponding to the  increasing values of $\epsilon$, the parameter guiding the $\epsilon$-neigbourhood clustering used to construct the Mapper outputs. The plot in the top right shows the instability as a function of $\epsilon$. The bin overlap was $35$ percent, with $17$ bins and the instabilities were averaged over $30$ different sub-samples, with $10$ sub-samples in each case. See \S \ref{sec:ComputingInstability} for details of the procedure.
	}
	\label{fig:ChangingEpsilon}
\end{table}
}

Table \ref{fig:ChangingEpsilon}  considers a dataset with two noisy concentric circles. We produce a family of Mapper
graphs using the $\epsilon$-neighbourhood clustering with varying values of epsilon. The specific clustering procedure used was DBSCAN
from the sklearn python package.
For the values of epsilon of $0.06$ and $0.09$, the instability decreases due to the disappearance of noise represented by spurious small connected components in the Mapper graph.
The major part of the structure of the Mapper graph remains the same, revealing both the inner and outer circle.
Above the epsilon values of $0.1$, there is a spike in the instability value corresponding to a loss of detail in the inner circle within the Mapper graph.
A similar spike occurs around the $0.32$ value of epsilon, corresponding to the loss of the inner circle from the Mapper graph.
The final large increase in instability occurs around the $0.45$ value of epsilon, and it corresponds to the gradual merging of the two circles in the Mapper graph.

\begin{figure}[htp!]
    \centering
    \def\svgwidth{1.\textwidth}
    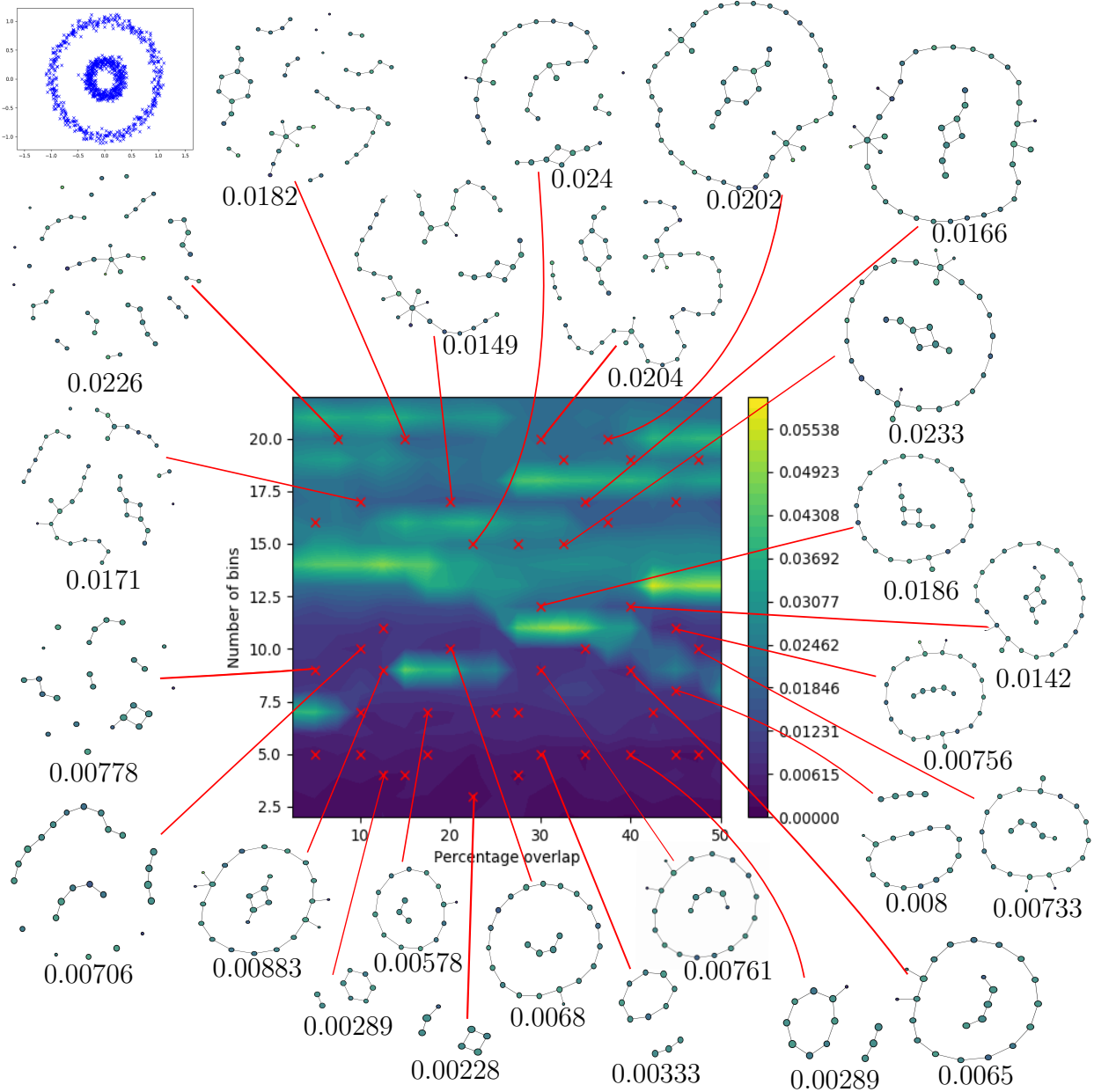
    \caption{
    We consider $1000$ points sampled with noise from two concentric circles. The centre of the figure shows a contour plot of instability values varying over the percentage overlap (gain) of the bins and the number of bins (resolution). 
    The red crosses  correspond to the local minima. The numbers next to the vertical bar on the right are values of instability. 
    Surrounding the plot are the Mapper graphs corresponding the various local minima.
    Below each Mapper graph is the corresponding instability value.
    The bin overlap was between $2.5\%$ and $50\%$ at $2.5\%$ interval steps.
    The number of bins varied  between $2$ to $22$. The instabilities were averaged over $10$ runs where we selected $10$ random sub-samples in each case. See \S \ref{sec:ComputingInstability} for details of the procedure.
    }
    \label{fig:ResolutionVsGain}
\end{figure}

We now pass to experiments that explore the dependence of 
the Mapper graph on the values of resolution and gain. 
Figure \ref{fig:ResolutionVsGain} presents a contour plot of the instability of Mapper on another dataset  consisting of noisy concentric circles created by  varying  the percentage overlap between bins (gain) and the number of bins (resolution).

Similarly to the discussion on Table \ref{fig:ChangingEpsilon}, it is possible to identify a number of global features within the plot with structural changes in the Mapper graph.

Running between bin numbers of $7$ and $13$, there is a diagonal of high peaks in instability.
Restricting to odd number of bins, this range of peaks appears to correspond to the emergence of the inner circle within the Mapper graph. 
 All graphs below the first distinct diagonal show the inner circle as a cluster without a cycle. Mapper graphs for odd bin numbers above the diagonal contain the structure of the inner circle. 

Along the horizontal value of $14$ bins, there is a relative rise in instability.
This appears to correspond to the fact that if we use an
even number of bins the correct structure of the inner circle is revealed.

The region determined by bin numbers from  $8$ to $12$ and percentage overlaps from $25$ to $50$ is a negatively sloped diagonal of relatively high instability.
  This appears to correspond to the emergence in the Mapper graph of a new 
  relatively large  cluster attached to the structure of the outer circle forming a flare corresponding to either a number of points at the top or at the bottom of the outer circle.

Running between bin numbers $14$ and $20$ is another diagonal range in peaks of instability. These peaks seem to appear when restricted to even numbers of bins and correspond to the emergence of a better defined structure of the inner circle within the Mapper graphs.

Finally, the high instability in the top left hand corner of the contour plot appears to capture the moment when the part of the Mapper graph corresponding to outer circle
breaks up.

We conclude that to infer the reliability of the Mapper graph the Mapper instability should be considered over the whole parameter space. While it is intuitively clear that 
a more complicated Mapper output often gives a more unstable result we show that 
 jumps in instability appear to correspond well with the structural changes in the Mapper output.
A jump in complexity accompanied by a relatively low jump in instability, suggests that the additional structure is indeed present in the data,
providing a method to determine the reliability of features present within relatively stable regions in the parameter space.

%%%%%%%%%%%%%%%%%%%%%%%%%%%%%%%%%%%%

%%%%%%%%%%%%%%%%%%%%%%%%%%%%%%%%%%%%

%%%%%%%%%%%%%%%%%%%%%%%%%%%%%%%%%%%%

\section{Mapper boundary distance}\label{sec:BounderyDistance}

To compare clustering functions on a metric space $(U,D)$, Ben-David and von Luxburg \cite{BenDavid_vonLuxburg08} introduced a distance function that captures the size of the regions of $U$ on which two clustering functions disagree.
We now expand upon and generalise this boundary distance to the Mapper setting
and use it to provide upper bounds of the Mapper instability in the following section.

\begin{definition}\label{def:Boundary}
	Let $(\mathcal{X},D)$ be a metric space with cover $\lbrace U_i \rbrace^t_{i=1}$.
	Then given a Mapper function $f \in \mathcal{N}$, define the boundary of each $f_i$ to be
	\begin{equation*}\label{eq:setOfDiscontinuities}
		\partial(f_i) =
		\partial(f^{-1}_i(c^i_1))\cup\cdots\cup\partial(f^{-1}_i(c^i_s)) \cup\partial U_i,
	\end{equation*}
	where each $\partial(f^{-1}_i(c^i_j))$ is the usual topological boundary of the subset
	$f^{-1}_i(c^i_j)$ taken over $U_i$ and $\partial U_i$ the boundary of $U_i$ in $\mathcal{X}$.
	Following the established conventions, we will refer to $\partial(f_i)$ as the \emph{decision boundary} of $f_i$. 
\end{definition}

Intuitively, $\partial(f_i)$ consists of the points of discontinuity of $f_i$, that is, the points lying in the boundary of some cluster, and an illustration of $\partial(f_i)$ is provided in Figure \ref{fig:D_boundary_a}. As $U_i$ is a metric space, 
$\partial(f_i)$ can be described using an equivalent metric condition, which defines the boundary $\partial A$
of any subset $A\subseteq U_i$ by
\begin{equation}\label{eq:MetricBoundaryCondition}
    \partial A = \{ x \in \bar{U}_i \mid  D(x, A) = D(x,A^c) =0\}, 
\end{equation}
where $\bar{U}_i$ is the closure of $U_i$ in $\mathcal{X}$ and $A^c = 
\mathcal{X}\setminus A$ is the complement of $A$ in the 
metric space $\mathcal{X}$ and
the distance of a point $x \in \bar{U}_i$ from a set $A\subseteq U_i$ is defined 
as usual by 
\begin{equation*}\label{eq:DistenceToBoundary}
    D(x, A) = \inf \left\lbrace  D(x, y) \;\middle|\;y \in A \right\rbrace.
\end{equation*}

\begin{remark}\label{rmk:ProperConnectedCover}
    If $t=1$ and $U_1=\mathcal{X}$, then $U_1=\bar{U}_1$ and we recover 
	the notion of boundary for clustering seen in \cite{BenDavid_vonLuxburg08}.
	In the case when any $U_i=\mathcal{X}$ and $f_i$ is the constant function,
	that is there is a single cluster in $U_i$,
	then 
	\begin{equation*}
	    \partial f_i = \emptyset.
	\end{equation*}
	If $U_i$ is connected this is the only way to obtain $\partial f_i = \emptyset$.
	For clustering it is not of particular interest to study data with a single cluster and so this does not cause many problems.
	Since no known Mapper constructions have disconnected $U_i$ or some $U_i=\mathcal{X}$ and neither exception seems practically reasonable,
	from now on unless stated otherwise, we assume that each $U_i$ is connected with no $U_i = \mathcal{X}$.
\end{remark}

To avoid unnecessary technicalities, two clusterings will be considered different if and only if their values differ outside the intersection of their boundaries.
Hence, we work on the set of all clusterings $f_i \in \mathcal F^i$ that represent 
elements in the space of equivalence classes 
\begin{equation*}\label{eq:BoundarylesClustering}
	\mathcal{F}_\partial^i
	= \mathcal F^i / \sim,
\end{equation*}
where $f_i \sim g_i$ if and only if
\begin{equation}\label{eq:different_on_the_boundary}
    \exists \pi: f_i(x) = \pi g_i(x), \forall x\in U_i - \partial(g_i), \hspace{25pt}\text{\rm{and}}
\end{equation}
\begin{equation*}
    \exists \pi': g(x) = \pi' f_i(x), \forall x\in U_i - \partial(f_i),
\end{equation*}
where $\pi, \pi'\in \Sigma_s$ are permutations of the set of labels.

\begin{definition}\label{def:MapperBoundary}
Let $f$ be a Mapper function constructed using clustering functions $f_i$ of the sets  $U_i$.
Then the decision boundary $\partial(f)$ of the Mapper function $f$ can 
is defined using  the decision boundaries of the functions $f_i$ by 
\begin{equation*}\label{eq:BoundaryMapper}
    \partial(f) = \bigcup_{i=1}^t \partial(f_i).
\end{equation*}
We denote by $\mathcal{N_\partial}$ the set of Mapper functions $f\in \mathcal{N}$ such that each $f_i$ is an element of
$\mathcal{F}_\partial^i$.
\end{definition}

For any $\gamma > 0$, we define the $\gamma$-tube of $f_i$ to be
\begin{equation*}\label{eq:TunbularNhood}
    T_\gamma(f_i) = \left\lbrace  x \in U_i \;\middle|\; D \left(x, \partial(f_i) \right) \leq \gamma \right\rbrace.
\end{equation*}
For $\gamma = 0$, we set $T_0(f_i) = \partial(f_i).$
Figure \ref{fig:D_boundary} illustrates the construction $T_\gamma(f_i)$.
If two clusterings $f_i, g_i \in \mathcal F_\partial^i$ agree outside the $\gamma$-tube of $f_i$, we will write $g_i \triangleleft T_\gamma(f_i)$. Thus the condition $g_i \triangleleft T_\gamma(f_i)$ holds if and only if
for all $x,y$ in the complement $U_i - T_\gamma(f_i)$ of the $\gamma$-tube of $f_i$ we have that 
\begin{equation*}
    f_i(x) = f_i(y) \Leftrightarrow g_i(x) = g_i(y).
\end{equation*}

\begin{remark}\label{rmk:connectedness}
The assumption of $U_i$ being connected is not a serious restriction, as the support of $P$ need not be connected.
One can relax this condition, and define $T_\gamma(f_i)$ directly as 
\[
 T_\gamma(f_i) = \left\lbrace  x\in U_i \,|\, \exists y\in {U_i} \colon d(x, y) \leq \gamma \text{ and } 
f_i(x) \neq f_i(y) \right\rbrace. 
\]
This, however, raises other technical issues that need to be treated with care, as hinted at in \cite{BenDavid_vonLuxburg08}.
\end{remark}

\begin{definition}\label{def:MapperGammaTube}
We define the $\gamma$-tube around a Mapper function $f\in \mathcal{N}$ to be
\begin{equation*}
    T_\gamma(f) = \left\lbrace  x \in \mathcal{X} \,|\, d \left(x, \partial(f) \right) \leq \gamma \right\rbrace, 
\end{equation*}
where 
\begin{equation*}
    d(x, \partial(f)) = \min_{i=1,\dots,t} d(x, \partial(f_i)),
\end{equation*}
and where the minimum in the last formula is taken over the indices $i$ such that $x\in U_i$. 
\end{definition}

Note that $T_\gamma(f)$ can be seen as the union of the individual tubes $T_\gamma(f_i)$:
\begin{equation}\label{eq:tube_as_union}
    T_\gamma(f) =  \bigcup_{i=1}^t T_\gamma(f_i).
\end{equation}
Consequently, the mass $P \left( T_\gamma(f) \right)$ of the 
$\gamma$-tube $T_\gamma(f)$ with respect to the probability measure $P$ depends on the overlap between the bins $U_i$.
We have the following natural estimate. 

\begin{proposition}\label{prop:boundsForMassfFTubes} 
   The mass $P \left( T_\gamma(f) \right)$ of the tube $T_\gamma(f)$ is bounded by the mass of the $\gamma$ tubes $T_\gamma(f_i)$ as follows: 
    \begin{equation*}
        \max_{i=1,\dots,t} P \left( T_\gamma(f_i) \right) \leq P \left( T_\gamma(f) \right) \leq \sum_{i=1}^t P \left( T_\gamma(f_i) \right).
    \end{equation*}
    
    The inequality on the left becomes an equality when all the elements $U_i$ are contained in 
    one of these sets, provided on that $U_j$ the boundary $\partial(f_j)$ is nonempty.
    The inequality on the right becomes an equality when the bins $U_i$ are all 
    disjoint. 
\end{proposition}

\begin{proof}
    Since $T_\gamma(f) = \bigcup_{i=1}^t T_\gamma(f_i),$ we have 
    $P \left( T_\gamma(f_i) \right) \leq P \left( T_\gamma(f) \right)$ for all $i=1,\dots,t$ and hence,
    $\max_{1\leq i \leq m} P \left( T_\gamma(f_i) \right) \leq P \left( T_\gamma(f) \right)$, proving the inequality on the left. The other inequality follows 
    in a similar way. 
   
    Turning to the second part of the Proposition, if 
    there is some $1\leq i_0 \leq t$ such that $T_\gamma(f_j) \subseteq T_\gamma(f_{i_0})$ for all $1\leq j \leq t$, then $T_\gamma(f) = T_\gamma(f_{i_0})$. Hence
    $P \left( T_\gamma(f) \right) = P \left( T_\gamma(f^{i_0}) \right) = \max_{1 \leq i \leq t} P \left( T_\gamma(f_i) \right)$, realizing the lower bound.
    
    Similarly, if $T_\gamma(f_i) \cap T_\gamma(f_j) = \emptyset$ for all $i\neq j$, then $P \left( T_\gamma(f) \right) = \sum_{i=1}^t P \left( T_\gamma(f_i) \right)$, realizing the upper bound.
\end{proof}
\begin{definition}\label{def:Triangle}
Given Mapper functions $f, g \in \mathcal{N}_\partial$, we say that $g$ is contained in the $\gamma$-tube of $f$, written $g \triangleleft T_\gamma(f)$, if for 
all $x,y$ in the complement of $T_\gamma(f)$
\begin{equation*}
    f(x) = f(y) \Leftrightarrow g(x) = g(y).
 \end{equation*}
 \end{definition}
 
It is clear that this statement is equivalent to saying that for 
all $i=1,\dots,t$,
\begin{equation*}\label{eq:g_triangle_tube_forall_i}
    g \triangleleft T_\gamma(f) \Longleftrightarrow g_i \triangleleft T_\gamma(f_i).
\end{equation*}

\begin{figure}[ht!]
    \centering
    \begin{subfigure}[b]{0.3\textwidth}
    \centering
\def\svgwidth{1.\textwidth}
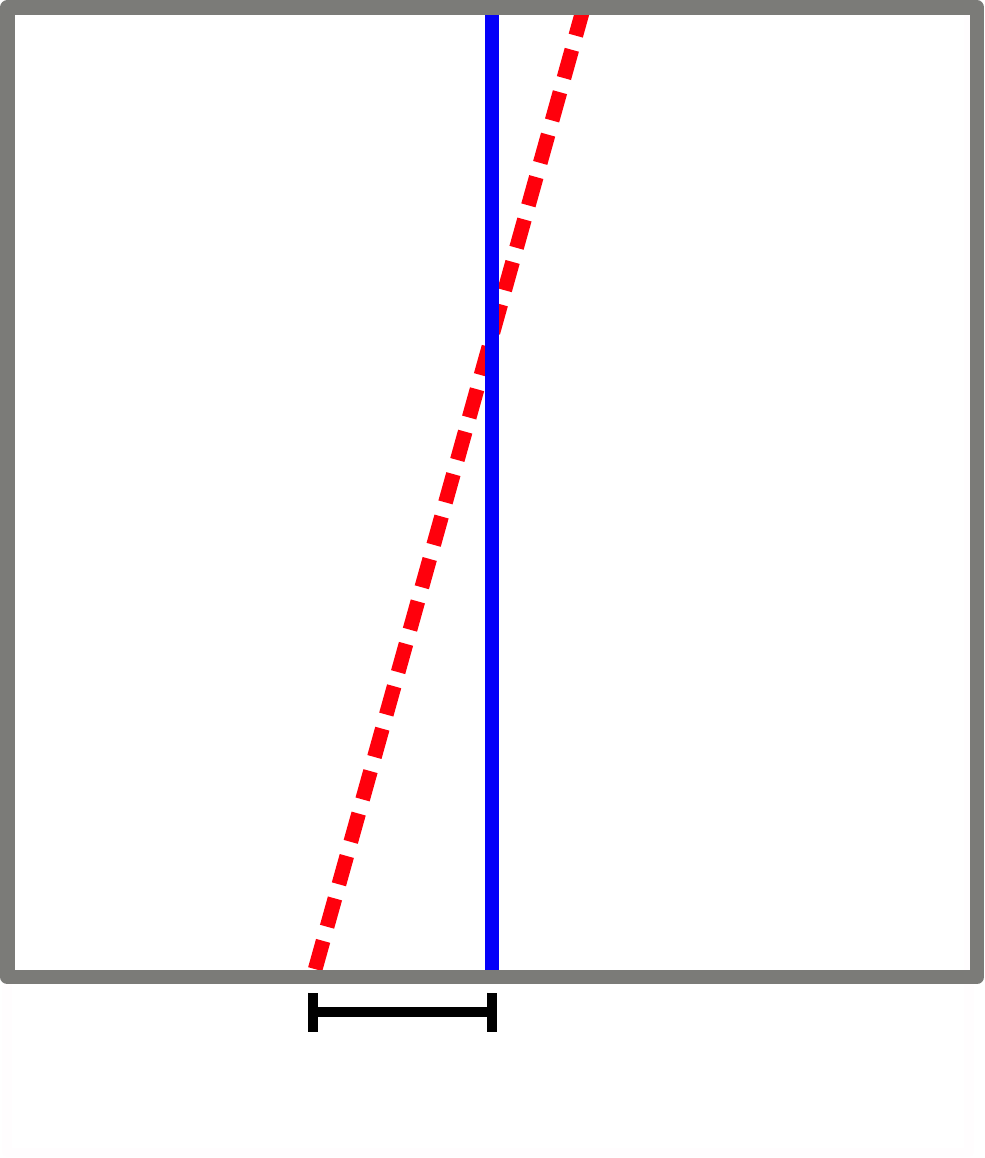
\caption{}
\label{fig:D_boundary_a}
    \end{subfigure}
    \begin{subfigure}[b]{0.3\textwidth}
    \centering
\def\svgwidth{1.\textwidth}
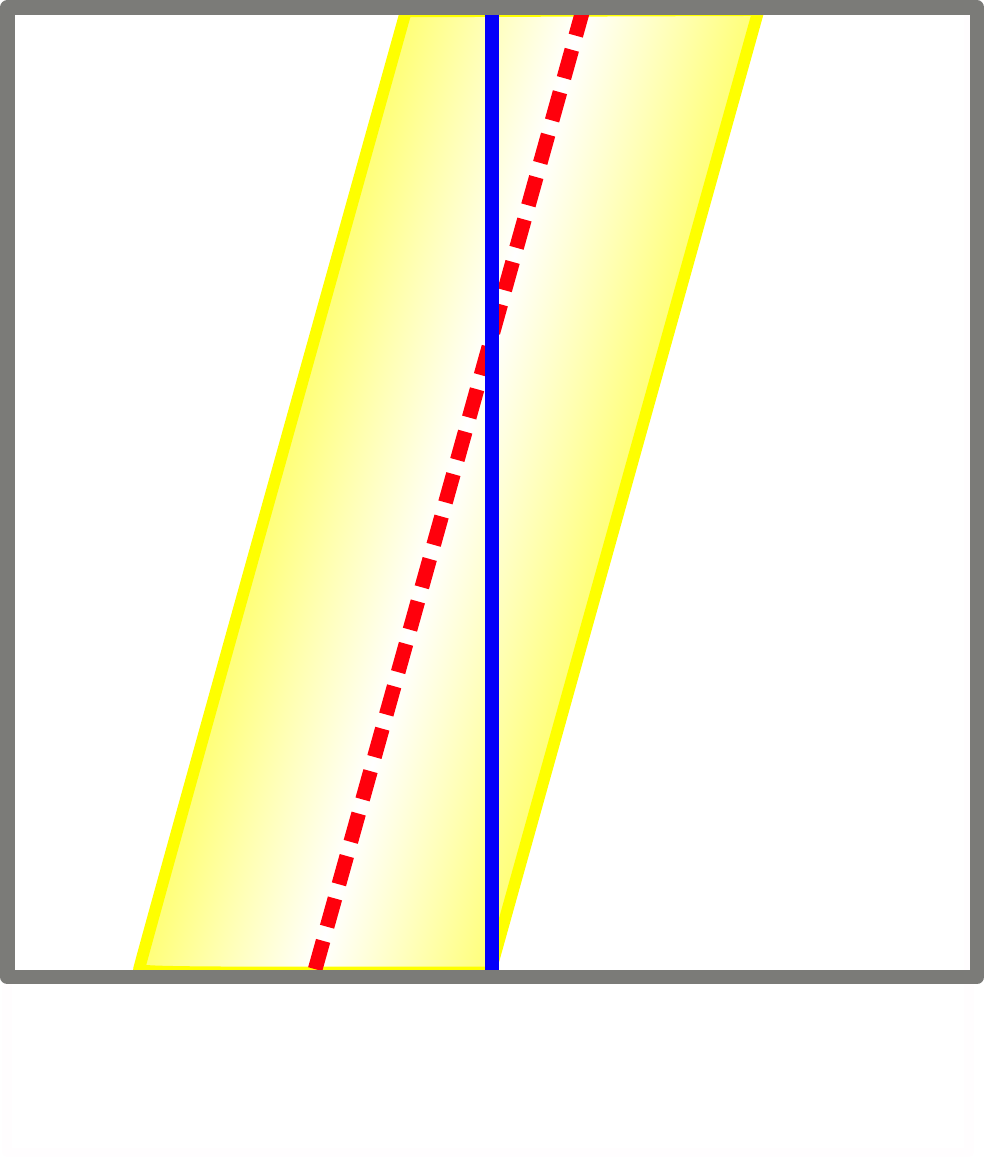
        \caption{}
        \label{fig:D_boundary_b}
    \end{subfigure}
    \begin{subfigure}[b]{0.3\textwidth}
    \centering
\def\svgwidth{1.\textwidth}
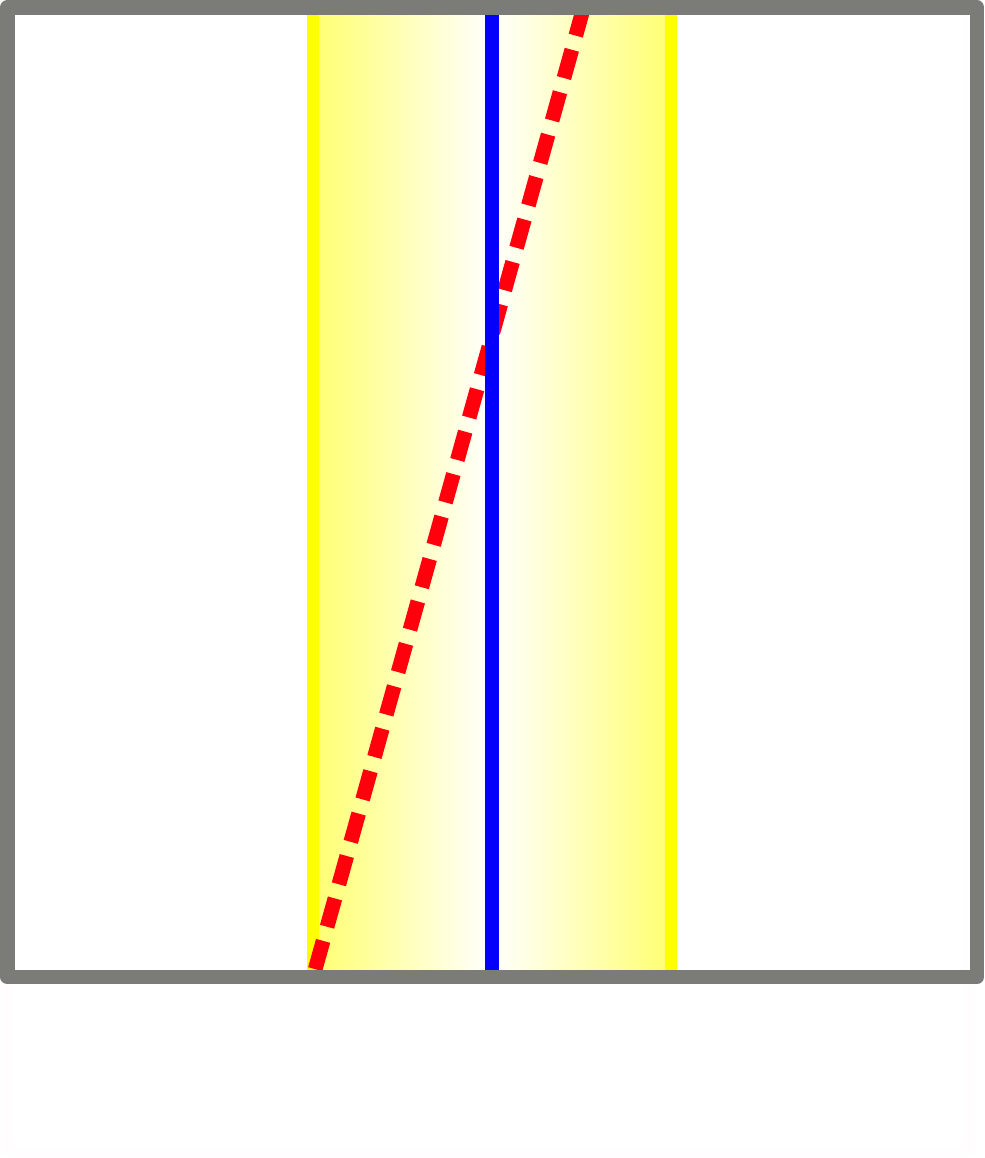
        \caption{}
        \label{fig:D_boundary_c}
    \end{subfigure}
    \caption{Assume that $t=1$, so that a Mapper function $f$ coincides with the clustering function $f_1$. In (a), everything left from the solid vertical blue line is labelled by $f$ as cluster 1 and everything right from that line is labelled by $f$ as cluster 2. Hence, $\partial(f)$ coincides with the solid line. Analogously, the left and right side of the dashed tilted line correspond to clusters 1 and 2 of $g$, respectively. Hence, $\partial(g)$ is precisely the dashed line. $D_\partial (f, g)$ is also illustrated in (a).
    In particular, $\forall \gamma > D_\partial (f, g)$, $f$ and $g$ agree both outside the $\gamma$-tube of $g$ (\emph{i.e.,} $f \triangleleft T_\gamma(g)$; see (b)) and outside the $\gamma$-tube of $f$ (\emph{i.e.,} $g \triangleleft T_\gamma(f)$; see (c)).}
    \label{fig:D_boundary}
\end{figure}

\begin{definition}\label{def:D_boundary}
Let $f$ and $g$ be Mapper functions in $\mathcal{N}_\partial$. The boundary 
distance $D_\partial$ is defined by 
    \begin{equation*}\label{eq:D_boundary}
        D_\partial (f, g) = \inf \left\lbrace  \gamma>0 \,|\, f \triangleleft T_\gamma(g) \text{ and } g \triangleleft T_\gamma(f) \right\rbrace .
    \end{equation*}
\end{definition}

The metric $D_\partial$ is therefore an interleaving distance between the 
$\gamma$-tubes of the functions $f$ and $g$. 

\begin{remark}\label{rmk:pseudo-distance}
    If some $U_i$ is unbounded then $D_\partial$ may be infinite. 
    In practice the support of $P$ will always be bounded, hence if $U_i$ is unbounded, then we may restrict $\mathcal{X}$ and $U_i$ to some bounded subset containing the support of $P$.
    Therefore unless stated otherwise, we assume from now on that each $U_i$ is bounded. 

    In this case we note (and leave it to the reader to check) that condition 
    (\ref{eq:different_on_the_boundary}) makes $D_\partial$ a metric.
    Without this restriction, $D_\partial$ is only a pseudo-metric, as is also the case for clusterings. 
\end{remark}

To get more information regarding the boundary metric,  we need to examine 
in a bit more detail the relationship between the space $\mathcal{N}$ of all 
Mapper functions and the spaces $\mathcal{F}_i$ of clusterings of the 
individual sets $U_i$ in the cover of $\mathcal{X}$. We have the following. 

\begin{lemma}\label{lemma:Bijection}
There exists a bijection 
\[
\phi: \mathcal{N} \longrightarrow  \prod_{i=1}^t \mathcal{F}^i. 
\]
\end{lemma}
\begin{proof}
    Let $\phi$ be a map 
\begin{equation}\label{eq:varphi}
    \begin{array}{rccc}
        \varphi: & \mathcal{N} & \xrightarrow{\hspace*{1cm}} &
        \prod_{i=1}^t \mathcal{F}^i\\
        & f & \xmapsto{\hspace*{0.6cm}} & \left(f_1, \ldots, f_t \right),
    \end{array}
\end{equation}
where each $f_i \colon U_i \longrightarrow \left\lbrace c^i_1 ,\dots ,c^i_s \right\rbrace$ is a function defined as follows: for every $x \in U_i$,
$f_i(x)$ is the only value in the singleton set $f(x) \cap \{ c^i_1 ,\dots ,c^i_s \} $.

The inverse map to $\varphi$ is given by the construction of a 
Mapper function $f$ from clustering functions $f_1, \ldots, f_t$ as 
described in 
Definition \ref{def:Mapper_functions}. 
\end{proof}

It follows that  we can view $\mathcal{N}_\partial$ as the product $\prod_{i=1}^t \mathcal{F}_\partial^i$ and so the space
 $\mathcal{N}_\partial$ is naturally a  product metric space in the following way. If $f$ and $g$ are represented
as $f=(f_1, \dots, f_t)$ and $g = (g_1, \dots, g_t)$, then by the 
bijection of Lemma \ref{lemma:Bijection}, 
it is straightforward to check that  
\begin{equation}\label{eq:D_boundary_as_max}
    D_\partial (f, g) = \max_{i=1,\dots,t} D_\partial (f_i, g_i),
\end{equation}
where $D_\partial (f_i, g_i)$ is the Mapper distance $D_\partial$  restricted to clustering functions on a single $U_i$.
From now on, we will think of $\mathcal{N}_\partial$ as the product metric spaces $\left( \mathcal F_\partial^i, D_\partial \right)$.

The metric $D_\partial$ has several nice properties exhibited in the next Proposition, which 
provides  a crucial step in the proof of Theorem \ref{thm:Bound_depending_mostly_on_gamma} that provides an upper bound for the instability of Mapper.

\begin{proposition}\label{prop:Properties_of_D_boundary} 
    Denote by $\left\lbrace U_i\right\rbrace _{i=1}^t$ a connected, bounded cover of the metric space $\mathcal{X}$ such that no $U_i=\mathcal{X}$.
    Then, with the notation above, the following properties hold:
    \begin{enumerate}
        \item
        Let $f, g \in \mathcal{N}_\partial$ and $\gamma > 0$. Then, $g \triangleleft T_\gamma(f)$ implies that $\partial(g) \subseteq T_\gamma(f)$.
        \item
        Let $f, g \in \mathcal{N}_\partial$ and $\gamma > 0$ be such that $D_\partial(f, g) \leq \gamma$ and that the clusters determined by $f$ and $g$ are connected in each $U_i$. Then, for any choice of representatives $f_i,g_i\in F_i$, there exists $\pi$ such that for all $x \in \mathcal{X}$,
        \begin{equation*}
            f(x) \neq \pi (g(x)) \Longrightarrow x \in T_\gamma(g),
        \end{equation*}
        where $\pi = \bigoplus_{i=1}^t \pi^i$, and $\pi^i$ denotes a permutation of the set $\left\lbrace c^i_1 ,\dots ,c^i_s \right\rbrace $.
        \item \label{propItem:Properties_of_D_boundary_relativeCompactness}
        If $\mathcal{X}$ is a subset of $\mathbb{R}^a$, the metric on $\mathcal X$ is induced by a norm on $\mathbb R^a$ and each $U_i \subseteq \mathcal X$ is compact, then
        $\left( \mathcal{N}_\partial, D_\partial \right)$ is relatively compact.
    \end{enumerate}
\end{proposition}

\begin{proof}
    Let $f, g \in \mathcal{N}_\partial$ and $\gamma > 0$ be such that $g \triangleleft T_\gamma(f)$. 
    By definition, this means that 
    for all $i$, $g_i \triangleleft T_\gamma(f_i).$
    The points in $\partial U_i$ contained in $\partial(g)$ are also contained in $\partial f$ and since $\partial f \subseteq T_\gamma(f)$ are contained in $T_\gamma(f)$ too.
    By definition of the clustering boundary using the metric condition (\ref{eq:MetricBoundaryCondition}), for every $x\in \partial(g_i)-\partial U_i$ and every $\epsilon >0$, the open ball $B(x,\epsilon)$ in $U_i$ contains two points $y$ and $z$ such that $f_i(y)\neq f_i(z)$. 
    Hence by definition of $g \triangleleft T_\gamma(f_i)$, for every $\epsilon >0$, we have that $B(x,\epsilon) \cap T_{\gamma}(f)$ is nonempty.
    Then since $T_{\gamma}(f_i)$ is a closed set, we have that $x\in T_{\gamma}(f_i)$, which implies 
     that $\partial(g_i) \subseteq T_\gamma(f_i)$ for all $i$.
    Therefore, using Definition \ref{eq:BoundaryMapper} and equality (\ref{eq:tube_as_union}), we have that
    \begin{equation*}
        \partial(g) = \bigcup_{i=1}^t \partial(g_i) \subseteq \bigcup_{i=1}^t T_\gamma(f_i) = T_\gamma(f),
    \end{equation*}
    which proves (1).
    
    Let $f, g \in \mathcal{N}_\partial$ and $\gamma > 0$ be such that $D_\partial(f, g) \leq \gamma$, with clusters connected in each $U_i$.
    This means that $D_\partial(f_i, g_i) \leq \gamma$ for all $i$.
    By part (1) of the Proposition, for each $i$, the $\gamma$ tube
    $T_{\gamma}(g_i)$ contains $\partial(f_i)$ and contains $\partial(g_i)$ by construction.   
    Since we assume the clusters are connected, the intersections
    \begin{equation*}
        \left( f^{-1}_i(c^i_j) - T_{\gamma}(g_i) \right)
        \cap \left( g^{-1}_i(c^i_k) - T_{\gamma}(g_i) \right)
    \end{equation*}
    are either empty or connected and the functions $f_i$, $g_i$ are constant on these sets for any cluster labels
    $c^i_j$, $c^i_k$ respectively.
    Therefore for every $i$ and any choice of representatives $f_i,g_i\in F_i$, there exists $\pi^i$ such that for all $x \in U_i$,
    \begin{equation*}
        f_i(x) \neq \pi^i (g_i(x)) \Longrightarrow x \in T_\gamma(g_i),
    \end{equation*}
    where $\pi^i$ denotes a permutation of the set $\left\lbrace c_1^i, c_2^i \ldots, c_s^i \right\rbrace $.
    Setting $\pi = \bigoplus_{i=1}^t \pi^i$, the following holds for all $x \in \mathcal{X}$,
    \begin{equation*}
        f(x) \neq \pi (g(x)) \Longrightarrow \exists i: f_i(x) \neq \pi^i (g_i(x)) \Longrightarrow \exists i: x \in T_\gamma(g_i) \Longrightarrow x \in T_\gamma(g),
    \end{equation*}
    where the last implication follows from (\ref{eq:tube_as_union}), proving (2).
    
    Under the additional assumption of each
    $U_i \subseteq \mathbb{R}^s$ being compact, \cite[Proposition 1]{BenDavid_vonLuxburg08} (whose proof is that same in our setting) shows that each $\mathcal{F}_\partial^i$ is relatively compact. Since $\mathcal{N}_\partial$ is endowed with a product metric $D_\partial$, it follows that is $\mathcal{N}_\partial$ is relatively compact too, which proves (3). 
\end{proof}

\begin{remark}\label{rmk:compactness}
    In Proposition \ref{prop:Properties_of_D_boundary}, point  (\ref{propItem:Properties_of_D_boundary_relativeCompactness}), we assumed each bin $U_i$ to be compact. Consider the classical Mapper algorithm, 
    where a real-valued function $h \colon \mathcal{X} \longrightarrow \mathbb{R}$ and a collection of intervals $\left\lbrace I_i\right\rbrace _{i=1}^t$ covering $h(\mathcal X)$ are used to define each bin $U_i$ as $h^{-1}(I_i)$.
    Basic topology shows that a sufficient condition for each $U_i$ to be compact consists of each interval $I_i$ being of the form $[a_i, b_i]$ for some $a_i, b_i \in \mathbb{R}$, and the function $h$ being a proper map, that is
    a function such that inverse images of compact subsets are compact.
    Furthermore, it is enough to assume $\mathcal{X}$ to be compact and $h$ to be continuous to guarantee $h$ to be a proper map.
    Notice also that if all bins are compact, so is $\mathcal{X}$, as a finite union of compact sets.
\end{remark}

%%%%%%%%%%%%%%%%%%%%%%%%%%%%%%%

%%%%%%%%%%%%%%%%%%%%%%%%%%%%%%%

%%%%%%%%%%%%%%%%%%%%%%%%%%%%%%%

\section{Mapper stability as a function of Mapper parameters}
\label{sec:Upper_bounding_instability_of_Mapper}

In this section, in Theorems \ref{thm:Bound_depending_mostly_on_gamma} and \ref{thm:Bound_depending_on_cover_sampleSize_etc} we prove two results that provide estimates  of the instability of Mapper.
Moreover, as we shall see, these results provide 
practical insights into how the stability of the Mapper algorithm can be affected by the specific choice of the Mapper parameters, including the filter function, the cover, the clustering algorithm, the metric and the sample size.

Throughout this section and the remainder of the paper, we assume 
that $\mathcal{X}$ is a metric space equipped with a probability measure
$P\in M_1(\mathcal X)$. We further assume that
$\mathcal{X}$ is given a cover
 $\mathcal{X} = \bigcup_{i=1}^t U_i$ such that each $U_i$ is bounded, connected, has nonzero mass and  $U_i \neq \mathcal{X}$ for all $i$. As before, we assume given quality 
 functions
$Q^i$  on each $U_i$, together with 
the empirical quality functions
$Q_{n}^i$. Furthermore, we will 
 use the following additional notation.
\begin{itemize}
    \item
    Denote by $f$ the unique optimal Mapper function of $\mathcal X$, given by
    \begin{equation*}
        f =  \prod_{i=1}^t{C^i(P_i)}.
    \end{equation*}
    where $C^i$ is an optimal clustering function defined in 
    (\ref{eq:OptimalClusteringFunction1})
    \item
    Denote by $f^n$ the unique optimal empirical Mapper function, that is the function
    \begin{equation*}
        f^n = \prod_{i=1}^t{C_{n_i}^i}(X)
    \end{equation*}
    obtained from size-$n$ samples $X\in \mathcal{X}^n$ using 
    the empirical clustering functions $C_{n_i}^i$.
\end{itemize}
Following Proposition \ref{prop:Properties_of_D_boundary} (2), we will also assume that all clusterings on $U_i$ have connected clusters, however this automatically is the case for all common clustering procedures.
We begin by generalizing the estimates obtained in \cite[Proposition 2]{BenDavid_vonLuxburg08}. 

\begin{theorem}\label{thm:Bound_depending_mostly_on_gamma}
    Using the above assumption and notation, 
    the instability of the Mapper algorithm satisfies
    \begin{equation*}
        \text{\rm{InStab}}_\text{\rm{Mapper}}(\left\lbrace Q^i_{n_i} \right\rbrace _{i=1}^t,n, P) \leq 
        2 \bigg( P(T_\gamma(f)) + P \left(  D_\partial(f^n, f) > \gamma \right) + P \left(  n_i = 0 \right) \bigg),
    \end{equation*}
    where $\gamma \geq 0$ and
    \begin{itemize}
        \item
        $P(T_\gamma(f))$ denotes the mass of the $\gamma$-tube of $f$, 
        \item
        $P \left(  D_\partial(f^n, f) > \gamma \right)$ denotes the probability that the optimal empirical Mapper function $f^n$ satisfies $D_\partial(f^n, f) > \gamma$,
        \item
        and $P \left(n_i = 0 \right)$ is the probability that some $n_i$ is $0$, for $i=1,\dots,t$.
    \end{itemize}
\end{theorem}

\begin{proof}
    
    Define the following three collections of size-$n$ samples $X'=( X_1, \ldots, X_n ) \in \mathcal{X}^n$:
    \begin{itemize}
    \item Let $M_{\leq \gamma}$ be the set of $X'\in\mathcal{X}^n$ for which $D_\partial(f^n, f) \leq \gamma$.
    \item Let $M_{> \gamma}$ be the set of $X'\in\mathcal{X}^n$ for which $D_\partial(f^n, f) > \gamma$.
    \item Let $M_{\emptyset}$ be the set of $X'\in\mathcal{X}^n$ for which $D_\partial(f^n, f)$ is not defined, which is the set of those $X'$  for which $X'\cap U_i$ no elements, for some $i=1,\dots,t$.
    \end{itemize}
    In particular, we have that $\mathcal{X}^n=M_{\leq \gamma} \cup M_{> \gamma} \cup M_{\emptyset}$.
    Without loss of generality, let us assume that the permutation $\pi$ for which the minimum value of $D_M$ is attained (see Definition \ref{def:D_Mapper}) is the identity.
    By Definition \ref{def:Mapper_instability},
    \begin{equation*}
        \text{\rm{InStab}}_\text{\rm{Mapper}}(\left\lbrace Q^i_{n_i} \right\rbrace _{i=1}^t,n, P) =
    \mathbb{E} \left( \mathcal{I}(\{Q^i_{n_i}\}^t_{i=1}) \right).
    \end{equation*}
   To simplify notation, we will write $\textrm{InStab}$ for 
   the left hand side of the above equation. Let $f^n=\prod_{i=1}^t C^i_{n_i(X')}(X')$ and $g^n=\prod_{i=1}^t C^i_{n_i(X'')}(X'')$ denote the optimal empirical Mapper functions for samples $X',X'' \in \mathcal{X}^n$, respectively.
   Recall that, using the Voronoi cell construction (\ref{eq:Voronoi}) Mapper functions $f^n$ and $g^n$ can be extended to the $2n$-point sample $X=(X',X'')\in \mathcal{X}^{2n}$.
   Then by (\ref{eq:InstabilityVariableMapper}), taking $D_M$ over all point in $X=(X_1,\dots,X_{2n})$ and using the triangle inequality,
    \begin{align*}
        \textrm{InStab}
        = & \int_{X\in \mathcal{X}^{2n}} D_M(f^n(X),g^n(X)) dP^{2n}(X)
        \\
        \leq & \int_{X\in \mathcal{X}^{2n}} \left(D_M(f^n(X),f(X)) + D_M(f(X),g^n(X))\right) dP^{2n}(X)
    \end{align*}
    where we note that each of the two terms now depends only 
    on the variables either $f^n$ or $g^n$, respectively. Therefore,  we can now write
    \begin{align*}
        \textrm{InStab}
        = & \; 2 \int_{X\in \mathcal{X}^{2n}} D_M(f^n(X),f(X)) dP^{2n}(X)
        \\
        = & \; 2 \Bigg{(} \int_{\substack{X'\in M_{\leq \gamma}, \\ X''\in\mathcal{X}^n}} D_M(f^n(X),f(X)) dP^{2n}(X) +
        \int_{\substack{X'\in M_{> \gamma}, \\ X''\in\mathcal{X}^n}} D_M(f^n(X),f(X)) dP^{2n}(X)
        \\ & +
        \int_{\substack{X'\in M_{\emptyset}, \\ X''\in\mathcal{X}^n}} D_M(f^n(X),f(X)) dP^{2n}(X) \Bigg{)}.
    \end{align*}
    Since $D_M(f^n(X),f(X))\in [0,1]$ and using Defintion \ref{def:D_Mapper} for $D_M$, we obtain
    \begin{align*}
        \textrm{InStab}
        \leq & \; 2 \left( \frac{1}{2n}\int_{\substack{X'\in M_{\leq \gamma}, \\ X''\in\mathcal{X}^n}} \sum_{i=1}^{2n}{\mathds{1}_{f^n(X_i)\neq f(X_i)}} dP^{2n}(X) +
        P \left(X'\in M_{> \gamma} \right) +
        P \left(X'\in M_{\emptyset} \right) \right).
    \end{align*}
    If $f^n$ is obtained from a sample in $M_{\leq \gamma}$, then
    Proposition \ref{prop:Properties_of_D_boundary} (3) gives that for all $x \in \mathcal{X}$,
    $$f^n(x) \neq f(x) \Longrightarrow x \in T_\gamma(f).$$
    On the other hand, by definition, we have
    $P \left(M_{> \gamma} \right) = P \left( D_\partial(f^n, f) > \gamma  \right)$ and
    $P \left( M_{\emptyset} \right) =$  \\$P \left(  n_i = 0 \right)$.
    Therefore, we conclude
    \begin{align*}
        \textrm{InStab}
        \leq & \; 2 \left( \frac{1}{2n}\int_{\substack{X'\in M_{\leq \gamma}, \\ X''\in\mathcal{X}^n}} \sum_{i=1}^{2n}{\mathds{1}_{X_i \in T_{\gamma}(f)}} dP^{2n}(X) +
        P \left(  D_\partial(f^n, f) > \gamma \right) + P \left(  n_i =0  \right) \right)
        \\
        = & \; 2 \left( \frac{1}{2n}\cdot 2n\int_{\substack{x\in M_{\leq \gamma}}} {\mathds{1}_{x \in T_{\gamma}(f)}} dP + P \left(  D_\partial(f^n, f) > \gamma \right) + P \left(  n_i =0  \right) \right)
        \\
        = & 2 \bigg( P(T_\gamma(f)) + P \left(  D_\partial(f^n, f) > \gamma \right) + P \left(  n_i =0  \right) \bigg).
    \end{align*}
\end{proof}

If condition of Theorem \ref{thm:Bound_depending_on_cover_sampleSize_etc} that no $U_i$ is all of $\mathcal{X}$ is relaxed, then the $2P(n_i = 0)$ term of the bound is replaced by term
\begin{equation}\label{eq:ExtraSummand}
    2P(f_i^n=c^i_1),
\end{equation}
where $P(f_i^n=c^i_1)$ is the probability that the optimal empirical clustering function has a single cluster.
For a clustering procedure that returns at least two clusters,  this term becomes 
\begin{equation}\label{eq:ExtraSummandBetter}
    2P(n_i \leq 1),
\end{equation}
where $P(n_i \leq 1)$ is the probability that less then $2$ points of the $n$-point sample are contained in $U_i$.
In the case when $t=1$ and the cover of $\mathcal{X}$ consists of a single member, then the Mapper algorithm reduces to a clustering algorithm and  Theorem  \ref{thm:Bound_depending_mostly_on_gamma} recovers \cite[Proposition 2]{BenDavid_vonLuxburg08}.

\begin{remark}
    \label{rmk:Reasons_which_produce_instability_1}
    \emph{\textbf{(Reasons for instability - Part I)}}
    Theorem \ref{thm:Bound_depending_mostly_on_gamma} can be used to identify the effect of particular parameter choices on the instability of the Mapper output as follows.
    Since each $U_i$ is assumed to have nonzero mass,
    the term $P(n_i = 0)$ will be insignificant provided no $U_i$ has extremely low mass or the sample size $n$ is very small. 
    We therefore omit this term form the remainder of the paper.
    The bound becomes large if $P(T_\gamma(f))$ is large, when the mass is concentrated around the decision boundary $\partial(f)$ of the optimal clustering $f$.
    This may happen if any of the following hold:
    \begin{itemize}
        \item [(a)]
        The decision boundaries $\partial(f_i)$ lie in a highly dense area.
        \item [(b)]
        The decision boundaries $\partial(f_i)$ are \lq long' in the 
        sense of a suitably defined path distance along $\partial(f_i)$.
        \item [(c)]
        There is low overlap between bins.
    \end{itemize}
   Moreover, Proposition \ref{prop:boundsForMassfFTubes} suggests $P(T_\gamma(f))$ can also be large if this holds:
    \begin{itemize}
        \item [(d)]
        The decision boundaries of different members of the cover are relatively far apart.
    \end{itemize}
    Indeed, 
    small changes to decision boundaries that are far apart necessarily increase the distance between the Mapper functions.  This is not  always true for decision boundaries that are close since they are more likely to mismatch on the same points, see Figure \ref{fig:FarBoundary} for an illustration. 
    
    Even if $P(T_\gamma(f))$ is small, $P \left( D_\partial(f_n, f) > \gamma \right)$ can still increase the bound, which happens if:
        \begin{itemize}
        \item [(e)]
        The decision boundaries $\partial (f_i^n)$ vary a lot with the choice of the sample.
        \end{itemize}
    While points 
    (a), (b) and (e) above are generalizations of those stated in \cite[\S 3]{BenDavid_vonLuxburg08}, the phenomena described 
   in  (c) and (d) are unique to Mapper, since they involve interactions within the cover.
\end{remark}

\begin{figure}[ht!]
    \centering
    \def\svgwidth{1.\textwidth}
    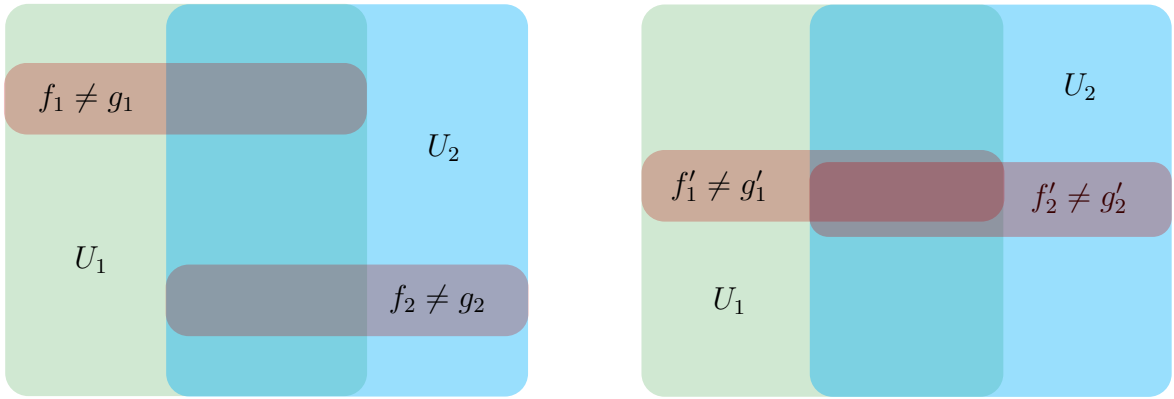
    \caption{The images show two overlapping bins $U_1$ and $U_2$ (green and blue respectively) as well as 
    the regions (red) where clustering functions assigned to each bin do not agree. On the left, the region in $U_1$ where $f_1$, $g_1$ do not agree, does not
    intersect the region where $f_2$, $g_2$ disagree in $U_2$.
    On the right, the mismatch regions between $f'_1$, $g'_1$ in $U_1$ and $f'_2$, $g'_2$ in $U_2$ have the same size as their counterparts on the left diagram. However, the ones on the right have a large intersection. Therefore assuming that point samples in the regions are similar, $D_M(f,g)>D_M(f',g')$ while $D_m(f_1,g_1)$ and $D_m(f_2,g_2)$ have similar values to $D_m(f'_1,g'_1)$ and $D_m(f'_2,g'_2)$, respectively.}
    \label{fig:FarBoundary}
\end{figure}

We now explore in more detail the instability of the 
Mapper output that result from parameter choices. 
High instability suggests  that the Mapper output varies significantly with small variations of the input data.
In particular, it is not surprising that Mapper instability increases  if the decision boundaries $\partial (f_i)$ vary a lot
with slight changes in the input sample. However, it is hard to identify explicitly the situations that make the term $P \left( D_\partial(f^n, f) > \gamma \right)$ large.
To deal with this,  in Theorem \ref{thm:Bound_depending_on_cover_sampleSize_etc} we provide 
an upper bound for $P \left( D_\partial(f^n, f) > \gamma \right)$ in terms that more clearly depend on the choice of Mapper parameters. 
While in general this leads to a less sharp bound, we gain a greater insight into how these variables affect the instability of Mapper.

\begin{remark}\label{rmk:Conditions}
    To state Theorem \ref{thm:Bound_depending_on_cover_sampleSize_etc}, we make the following assumptions on the quality functions $Q^i_{n} \colon \mathcal{F}_n \times  U_i^n \longrightarrow \mathbb{R}$
    and $Q^i \colon \mathcal{F}^i \times  M_1(U_i) \longrightarrow \mathbb{R}$.
    \begin{enumerate}
        \item \label{item:Q_Assum:Unique_global_min}
        The  functions $Q^i_{n}$ and $Q^i$ have a unique global minimizer $f_i \in\mathcal{F}_n$, as we are assuming throughout the paper.
        \item \label{item:Q_Assum:Continuous}
        The  functions $Q^i_{n}$ are continuous, with respect to the topology on $\mathcal{F}_n\times U^n_i$ given by the metric $D_\partial$.
        \item \label{item:Q_Assum:Uniformly_consistent} 
        The functions $Q^i_n$ are uniformly consistent with the functions $Q^i$ in the sense that for every $i$ and every $\gamma>0$, $ Q^i_n(f_i^n,X) \xrightarrow[n \rightarrow \infty ]{} Q^i(f_i, P_i)$ 
        in probability, uniformly over all probability distributions $P_i \in M_1(U_i)$.
        That is
        $\forall \epsilon >0, \forall \delta>0, \exists N \in \mathbb{N}$ such that $\forall n \geq N, \forall P_i \in M_1(U_i)$, 
        \begin{equation}\label{eq:Uniform_consistency_Q^i}
            P_i \left( \lvert Q^i_{n_i}(f_i^n, X) - Q^i(f_i, P_i) \rvert > \epsilon \right) \leq \delta.
        \end{equation}
        Note that $N$ does not depend on $P_i$. For future
        reference, we denote by $N^i(\epsilon,\delta)$ the minimum of the set of the numbers $N$ for which condition 
        (\ref{eq:Uniform_consistency_Q^i}) is satisfied. 
        
        Ben-David showed that uniform consistency holds for the algorithm constructing the global minimum of the $K$-means objective function \cite{BenDavid07}.
        Similar results occur with the normalized cut used in spectral clustering \cite{vonLuxburgEtAl08}.
        For more on consistency of clustering algorithms, see \cite{vonLuxburgEtAl08}.
    \end{enumerate}
\end{remark}

The next proposition shows that for a large enough sample size, formula (\ref{eq:Uniform_consistency_Q^i})
guarantees that the minimal quality function and empirical minimal quality functions will be  close in the boundary metric.

\begin{proposition}\label{prop:MinFunctionLargeN}
    In addition to the assumptions above, let us also assume that
    \begin{itemize}
        \item
        each $U_i$ is compact and of nonzero mass; and
        \item
        for every $0< \eta <1$ and $\zeta > 0$ there is $N \in \mathbb{N}$ such that for all $n\geq N$,
        \begin{equation*}
            P_i(|Q^i(f^n_i,P_i)-Q^i_{n_i}(f^n_i, X)|\leq\zeta) \geq \eta,
        \end{equation*}
        for each $i=1,\dots,t$.
    \end{itemize}
    Then for each $\epsilon >0$,
    \begin{equation*}
        P(D_{\partial}(f_i^n,f_i) \leq \epsilon) \xrightarrow[n \rightarrow \infty ]{} 1.
    \end{equation*}

\end{proposition}

The second assumption in the proposition is similar to the uniform consistency assumption (\ref{eq:Uniform_consistency_Q^i}), in that it assumes for a large enough $n$ the functions $Q^i$ and $Q^i_{n_i}$ are similar, in the case of the proposition that the functions take similar values at $f_i^n$.

\begin{proof}
    By \cite[Proposition 3]{BenDavid_vonLuxburg08} (whose proof is the same under our conditions)
    for all $\epsilon \geq 0$ there is an $\xi \geq 0$, such that for each $g \in \mathcal{F}_{\partial}^i$,
    \begin{equation*}\label{eq:QiToPartial}
        |Q^i(g,P_i)-Q^i(f_i,P_i)|\leq \xi \Longrightarrow D_{\partial}(f_i^n,f_i)\leq \epsilon.
    \end{equation*}
    Hence by the triangle inequality
    \begin{align}\label{eq:TriangleProb}
        P_i(D_{\partial}(f^n_i,f_i)\leq \epsilon)
        & \geq P_i \left( \lvert Q^i(f^n_i,P_i)-Q^i(f_i,P_i) \rvert \leq \xi \right) \nonumber \\
        & \geq P_i \left( \lvert Q^i_{n_i}(f_i^n, X) - Q^i(f_i, P_i) \rvert + \lvert Q^i(f_i^n,P_i)-Q^i_{n_i}(f_i^n,X)\rvert \leq \xi \right) \nonumber \\
        & \geq P_i \left( \lvert Q^i_{n_i}(f_i^n, X) - Q^i(f_i, P_i) \rvert \leq \frac{\xi}{2} \text{ and } |Q^i(f_i^n,P_i)-Q^i_{n_i}(f_i^n,X)| \leq \frac{\xi}{2} \right)
        .
    \end{align}
    On the other hand from (\ref{eq:Uniform_consistency_Q^i}), since for any $0 \leq \delta < 1$ there is an $N\in \mathbb{N}$ such that for $n\geq N$,
    \[
         P_i \left( \lvert Q^i_{n_i}(f_i^n, X) - Q^i(f_i, P_i) \rvert > \zeta \right) \leq 1-\delta 
        \]
     which implies that 
        \[
         P_i \left( \lvert Q^i_{n_i}(f_i^n, X) - Q^i(f_i, P_i) \rvert \leq \zeta \right) \geq \delta.
    \]
    Therefore picking $\zeta=\frac{\xi}{2}$, combining with the second assumption in the proposition and (\ref{eq:TriangleProb}),
    since $U_i$ have nonzero mass we obtain that
    \begin{equation*}
        P_i(D_{\partial}(f_i^n,f_i) \leq \epsilon) \xrightarrow[n \rightarrow \infty ]{} 1.
    \end{equation*}
    The statement of the proposition now follows.
\end{proof}

To state Theorem \ref{thm:Bound_depending_on_cover_sampleSize_etc}, we now introduce the term $\iota(n)$,
which describes in probabilistic terms the dependence of the behaviour
of the Mapper function on the properties of the clusterings for each $U_i$.

\begin{definition}\label{def:iota}
    If $P \left( D_\partial(f_i^{n}, f_i) \leq \gamma \right) \neq 0$ for all $i=1,\dots,t$,
    denote by $\iota(n)$ the real number $\iota(n) \geq 0$ such that 
    \begin{equation}\label{eq:Prod_of_probabilitites}
        P \left( D_\partial(f^n, f) \leq \gamma \right) = 
        \iota(n) \prod_{i=1}^t P \left( D_\partial(f_i^{n}, f_i) \leq \gamma \right).
    \end{equation}
    If $P \left( D_\partial(f_i^{n}, f_i) \leq \gamma \right) = 0$ for some $i=1,\dots,t$, then define $\iota(n) = 1$.
\end{definition}

The relationship between $\iota(n)$ and $n$ is not necessarily monotone. To see this, recall that, as stated in (\ref{eq:D_boundary_as_max}), for any $g, h \in \mathcal N_{\partial}$,
we have that 
\[
D_\partial (g, h) = \max_{i=1,\dots,t} D_\partial (g_i, h_i).
\]
For example, if for some $i=1,\dots ,t$, $P \left( D_\partial(f_i^{n}, f_i) \leq \gamma \right)$
decreases at a slower rate than the others with respect to $n$,
then the value of $\iota(n)$ will rise.
Under the assumptions of Proposition \ref{prop:MinFunctionLargeN}, we have that
$P(D_{\partial}(f_i^n,f_i) \leq \epsilon) \xrightarrow[n \rightarrow \infty ]{} 1$, implying that
\begin{equation}\label{eq:i(n)Limit}
    \iota(n) \xrightarrow[n \rightarrow \infty ]{} 1,
\end{equation}
so the behaviour of $\iota(n)$ for a large enough $n$ is determined.

\begin{remark}\label{rmk:DependceOfBoundarys}
    It is however not clear what range of values $\iota(n)$ may take.
    Intuitively given a large enough point sample $X\in \mathcal{X}^n$, if $D_\partial(f_i^{n}, f_i) \leq \gamma$ for some $i=1,\dots,t$, this would indicate that the sample well represented the underlying probability distribution $P_i$ on $U_i$. 
    So the subset of the point sample contained in another bin $U_j$ intersecting $U_i$ would be more likely to well represent $P_j$.
    This in turn should result in a lower value of $D_\partial(f_j^{n}, f_j) \leq \gamma$.
    More precisely for each $i=1,\dots,t$, we would expect that
    \begin{equation*}
        P(D_\partial(f_j^{n}, f_j) \leq \gamma\;|\;D_\partial(f_i^{n}, f_i) \leq \gamma)
        \geq P(D_\partial(f_j^{n}, f_j) \leq \gamma).
    \end{equation*}
    In this case, since the event $D_\partial(f^n, f)\leq \gamma$ is the intersection of events  $D_\partial(f^n_i, f_i)\leq \gamma$, using conditional probability we obtain that
    \begin{equation*}
        P \left( D_\partial(f^n, f) \leq \gamma \right) \geq
        \prod_{i=1}^t P \left( D_\partial(f_i^{n}, f_i) \leq \gamma \right).
    \end{equation*}
    In particular by Defintion \ref{def:iota}, this implies that
    \begin{equation*}
        \iota(n)\geq 1.
    \end{equation*}
    It then follows from (\ref{eq:i(n)Limit}), that $\iota(n)$ is minimised as $n$ grows.
    To make these points more precise we would require more information, especially regarding the properties of the clustering functions. 
\end{remark}

To find an upper bound on the term 
$P(D_\partial(f^n, f) > \gamma)$ of Theorem \ref{thm:Bound_depending_mostly_on_gamma},
we use properties of the cluster quality function $Q^i(-,P_i)$ in a neighbourhood of the global minimum $f_i$.
Assuming that each $U_i$ is compact, by \cite[Proposition 3]{BenDavid_vonLuxburg08} (whose proof is the same under our conditions) for every $\gamma>0$ and every $i=1,\dots,t$, there exists $\epsilon > 0$ such that for all 
$g\in \mathcal{N}$, written as  $g = \prod_{i=1}^t g_i$, 
the condition that
\[
|Q^i(g_i, P_i) - Q^i(f_i, P_i)| \leq \epsilon
\]
for all $i=1, \dots, t$ implies that
\[
 D_\partial(g, f) \leq \gamma.
\]

Let us denote by $S_{P_i}^{Q^i}(\gamma)$ the supremum of the 
set of all such $\epsilon$. 
See Figure \ref{fig_uniquenessOfMinimizers} for an illustration of what $S_{P_i}^{Q^i}(\gamma)$ represents.

\begin{figure}[ht!]
    \centering
    \def\svgwidth{1.\textwidth}
    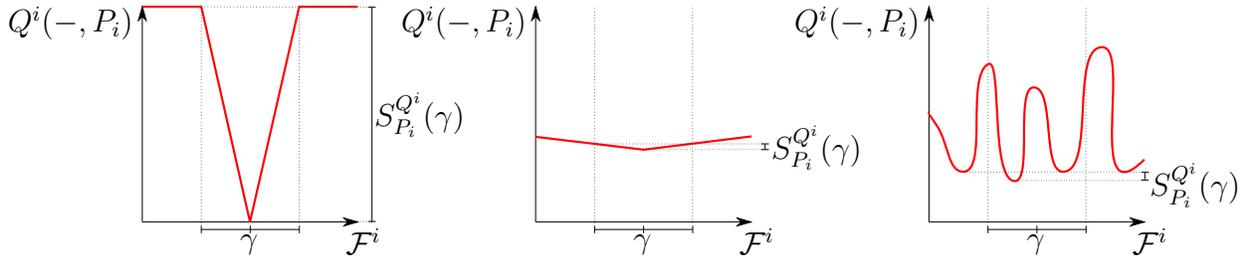
    \caption{$S_{P_i}^{Q^i}(\gamma)$ measures how distinctly unique the global minimum of $Q^i(-, P_i)$ is, by looking at a ball of radius $\gamma$ around the minimum of $Q^i(-,P_i)$. In the illustration, we identify $(\mathcal{F}^i, D_\partial)$ with a subset of the reals $(\mathbb{R}, d_\text{Euclidean})$. The quality function on the left has a very distinct global minimum, and hence a large $S_{P_i}^{Q^i}(\gamma)$. The other two functions exhibit different ways in which points can have values close to the global minimum and hence have a small $S_{P_i}^{Q^i}(\gamma)$.}
    \label{fig_uniquenessOfMinimizers}
\end{figure}

We can now express an upper bound on instability which involves, among others, the mass of the bins that form the cover of $\mathcal{X}$.

\begin{theorem}\label{thm:Bound_depending_on_cover_sampleSize_etc}
    Fix a sample size $n$, given the assumptions and notations presented at the beginning of the section, Remark  \ref{rmk:Conditions}
    and that each $U_i$ is a compact subset of $\mathbb{R}^a$.
    Then, for all $\gamma>0, \delta>0$, the instability of the Mapper algorithm satisfies
    \begin{equation}\label{eq:Upper_bound_phi}
        \text{\rm{InStab}}_\text{\rm{Mapper}}(\left\lbrace Q^i_{n_i} \right\rbrace _{i=1}^t,n,P) \leq 2P(T_\gamma(f)) + 2 \phi,
    \end{equation}
    where $\phi \in [0, 1]$ has the form
    
    \begin{equation*}\label{eq:phi}
        \phi = 1 - \iota(n) \left(1 - \delta \right)^t \prod_{i=1}^t P(U_i) \cdot P \left( n_i \geq N^i(S_{P_i}^{Q^i}(\gamma), \delta) \right)
    \end{equation*}
    and the function $N^i$ is defined in part (3) of Remark \ref{rmk:Conditions}.
\end{theorem}

\begin{proof}
    Fix $\gamma>0$ and some $\delta>0$. 
    We first find a lower bound for $P(D_{\partial}(f^n,f)\leq \gamma)$.
    This yields the upper bound $\phi$ for the term $P \left( D_\partial(f_n, f) > \gamma \right)$ of 
    Theorem \ref{thm:Bound_depending_mostly_on_gamma}, from which we will conclude that
    \begin{equation*}
        \text{\rm{InStab}}_\text{\rm{Mapper}}(\left\lbrace Q^i_{n_i}\right\rbrace _{i=1}^t,n,P) \leq 2P(T_\gamma(f)) + 2 \phi.
    \end{equation*}
    We denote by $n_i \geq N^i(S_{P_i}^{Q^i}(\gamma ), \delta )$ the event consisting of picking $X\in \mathcal{X}^n$ according to $P^n$ such that $N^i(S_{P_i}^{Q^i}(\gamma ), \delta )$ is greater than $n_i$.
    Since for any events $A$ and $B$, $P(A) \geq P (A \cap B)$, 
    for each $i=1,\dots,t$:
    \begin{equation}\label{eq:Adding_ni_Ci}
        P \left(  D_\partial(f_i^{n}, f_i) \leq \gamma \right) \geq
        P \left( \left\lbrace D_\partial(f_i^{n}, f_i) \leq \gamma\right\rbrace  \cap \left\lbrace n_i \geq N^i(S_{P_i}^{Q^i}(\gamma), \delta)\right\rbrace  \right).
    \end{equation}
    By conditional probability, $P \left( A \cap B \right) = P \left( A \;\middle|\; B \right) P(B)$ for any events $A$ and $B$.
    In particular, the expression on the 
    right hand side of (\ref{eq:Adding_ni_Ci})
    is equal to
    \begin{equation}\label{eq:2multiplicands}
        P \left(  D_\partial(f_i^{n}, f_i) \leq \gamma  \;\middle|\;  n_i \geq N^i(S_{P_i}^{Q^i}(\gamma), \delta)   \right) 
        \cdot
        P \left( n_i \geq N^i(S_{P_i}^{Q^i}(\gamma), \delta) \right).
    \end{equation}
    We now find a lower bound for the left multiplicand in (\ref{eq:2multiplicands}).
    By definition of $S_{P_i}^{Q^i}(\gamma)$,
    \begin{equation*}
        \lvert Q^i(f_i^{n}, P_i) - Q^i(f_i, P_i) \rvert \leq S_{P_i}^{Q^i} (\gamma)
        \Longrightarrow D_\partial(f_i^{n}, f_i) \leq \gamma.
    \end{equation*}
    Hence,
    $P_i \left(  D_\partial(f_i^{n}, f_i) \leq \gamma  \;\middle|\; n_i \geq N^i(S_{P_i}^{Q^i}(\gamma), \delta) \right) $ is bounded from below by
    \begin{equation}\label{eq:Q-Qconditional}
        P_i \left( \lvert Q^i_{n_i}(f_i^{n}, X) - Q^i(f_i, P_i) \rvert \leq S_{P_i}^{Q^i} (\gamma)  \;\middle|\; n_i \geq N^i(S_{P_i}^{Q^i}(\gamma), \delta) \right),
    \end{equation}
    Additionally, by definition of $N^i(S_{P_i}^{Q^i}(\gamma), \delta) $, if 
    $n_i \geq N^i(S_{P_i}^{Q^i}(\gamma), \delta)$
    then 
    \begin{equation*}
         P_i \left( \lvert Q^i(f_i^{n}, P_i) - Q^i(f_i, P_i) \rvert \leq S_{P_i}^{Q^i} (\gamma) \right) \geq 1 - \delta,
    \end{equation*}
    and therefore, the expression in (\ref{eq:Q-Qconditional}) is bounded below by $1 - \delta$.
    Hence by (\ref{eq:Restricted_probability}) the left factor in (\ref{eq:2multiplicands}) is bounded below by
    \begin{equation*}
        \left(1 - \delta \right) \cdot P(U_i).
    \end{equation*}
    This provides the following lower bound for the full expression in (\ref{eq:2multiplicands}):

    \begin{equation*}
        \left(1 - \delta \right) \cdot P(U_i )  
        \cdot
        P \left( n_i \geq N^i(S_{P_i}^{Q^i}(\gamma), \delta) \right),
    \end{equation*}
    and hence, using (\ref{eq:Prod_of_probabilitites}), the following lower bound for $P \left(  D_\partial(f_i^{n}, f_i) \leq \gamma \right)$:
    \begin{equation}\label{eq:1_minus_phi}
        \iota(n) \left(1 - \delta \right)^t \prod_{i=1}^t P(U_i) \cdot P \left( n_i \geq N^i(S_{P_i}^{Q^i}(\gamma), \delta) \right).
    \end{equation}
    If we define $\phi$ so that $1-\phi$ is the expression in (\ref{eq:1_minus_phi}), then,
    \begin{equation*}
    P \left( D_\partial(f^n, f) > \gamma \right) \leq \phi,
    \end{equation*}
    which combined with Theorem \ref{thm:Bound_depending_mostly_on_gamma}, provides us with (\ref{eq:Upper_bound_phi}).
    
    Finally we show that $\phi \in [0, 1]$.
    First note that
    $\phi \geq P \left( D_\partial(f_n, f) > \gamma \right)\geq 0$.
    On the other hand, the terms 
    $\iota(n), \left(1 - \delta \right), P(U_i)$ and $P \left( n_i \geq N^i(S_{P_i}^{Q^i}(\gamma), \delta) \right)$ are non-negative.
    Therefore (\ref{eq:1_minus_phi}) is non-negative.
    Subtracting this expression from 1 yields a result no larger than 1 and this is $\phi$ by definition.
\end{proof}

We now discuss the consequences of Theorem \ref{thm:Bound_depending_on_cover_sampleSize_etc} on the instability of Mapper.

\begin{remark}\label{rmk:Reasons_which_produce_instability_2}
    \emph{\textbf{(Reasons for instability - Part II)}}
    Theorem \ref{thm:Bound_depending_mostly_on_gamma} revealed that a large mass $P \left(T_\gamma (f) \right)$ around the minimizer $f$ of a Mapper quality function corresponded to an unstable Mapper output.
    Assuming the mass $P \left(T_\gamma (f) \right)$ to be small, a closer look at the term  $\phi$ introduced in 
    Theorem \ref{thm:Bound_depending_on_cover_sampleSize_etc} reveals the following additional reasons for the  instability of a Mapper algorithm. 
    \begin{enumerate}
        \item[(A)] \label{item:SmallSampleSize}
        A small sample size $n$ makes $\phi$ large.
        The term $P \left( n_i \geq N^i(S_{P_i}^{Q^i}(\gamma), \delta) \right)$
        decreases as $n$ increases.
        While $\iota(n)$ need not decrease monotonically with $n$, we know
        by (\ref{eq:i(n)Limit}) it does 
        tend to $1$ when $n$ tends to infinity,
        under the assumptions of Proposition \ref{prop:MinFunctionLargeN}.
        So the term $\iota(n)$ can be ignored for a large sample, possibly even minimal using the reasoning of Remark \ref{rmk:DependceOfBoundarys}. 
        Therefore for a large enough sample size, $\phi$ becomes small.
        \item[(B)] \label{item:Small_PhUi}
        A small value of $P(U_i)$ for some $i=1,\dots ,t$, makes $\phi$ large, \emph{i.e.,} close to $1$.
        \item[(C)] \label{item:SmallDistinctiveUniqueness}
     If a clustering quality function $Q^i$ has many points with values very near the global minimum then  $\phi$ is close to $1$. Indeed, 
        if there are many local minima of $Q^i(-,P_i)$, $g_i \in \mathcal{F}^i$ such that $|Q^i(g_i, P_i) - Q^i(f_i, P_i)|$ is small, for the given  global minimizer $f_i$ of $Q^i(-, P_i)$, then $S_{P_i}^{Q^i}(\gamma)$ is small (see Figure \ref{fig_uniquenessOfMinimizers} for an illustration), making $N^i(S_{P_i}^{Q^i}(\gamma), \delta)$ large, and in consequence, making $\phi$ large too.
    \end{enumerate}
\end{remark}

The above points add to the reasons for instability presented in Remark \ref{rmk:Reasons_which_produce_instability_1}.
The conditions given  in (A) and (B) can be seen as the global and local versions, respectively, of a similar phenomenon, since a small $P(U_i)$ means that the proportion of sampled points from $\mathcal{X}$ that fall into a $U_i$ is likely to be small.

The chosen clustering method and metric play a crucial role in the causes for instability of Remark \ref{rmk:Reasons_which_produce_instability_1} and in cause (C) in Remark \ref{rmk:Reasons_which_produce_instability_2}.
Among all the parameters selected for the classical Mapper algorithm given by a filter function and interval cover of $\mathbb{R}$,
the weight $P(U_i)$ in (B) depends only on the cover $\left\lbrace I_i\right\rbrace_{i=1}^t$ of $\mathbb{R}$ and the filter function $h \colon \mathcal{X} \longrightarrow \mathbb{R}$.
Hence, by choosing suitable $\left\lbrace I_i\right\rbrace_{i=1}^t$ and $h$, we would be able to control the value of $P(U_i)$, providing we have sufficient information of the 
distribution $P$.

Finally, notice that if (C)
applies, this may produce not only instability but also inaccuracy. This would arise in situations when the global minimum is not distinct enough, which leads to a possible error in finding the minimizer. This  is often a sign of a mismatch between the model and the data \cite{Shamir-Tishby10}.

%%%%%%%%%%%%%%%%%%%%%%%%%%%%%%%%%%%%%%%%%%%%%%%

\section{On the sharpness of bounds on instability}\label{sec:Sharpness_of_upper_bound}

In the previous section, we proved two theorems describing 
upper bounds on the instability of Mapper in terms of the behaviour of the Mapper parameters necessary to produce an output from some given data. In this Section, we discuss the 
efficiency of these estimates. 
To get a feel for the problem, let us first 
address the obvious question of the possible range 
of values for instability and its upper bounds. 
Let 
\begin{equation*}
    Bound_{D_\partial} = 2P(T_\gamma(f)) + 2P \left( D_\partial(f^n, f) > \gamma \right) + 2P(n_i = 0),
\end{equation*}
denote the bound from  Theorem \ref{thm:Bound_depending_mostly_on_gamma}.
As previously stated in Remark \ref{rmk:Reasons_which_produce_instability_1} we omit the summand $2P(n_i = 0)$ form the remainder of the discussion for reasonable Mapper parameters and sample size it will not substantially effect the bound.  
Let us also denote by
\begin{equation*}
    Bound_\phi = 2P(T_\gamma(f)) + 2 \phi,
\end{equation*}
 the bound from Theorem \ref{thm:Bound_depending_on_cover_sampleSize_etc}.
Fix all parameters except $\gamma>0$ and $\delta \in (0,1)$.
Since 
$P(T_\gamma(f)),$ $P \left( D_\partial(f_n, f) > \gamma \right), \phi \in [0,1]$, we have that
\begin{equation*}
    Bound_{D_\partial}, Bound_\phi \in [0, 4].
\end{equation*}
In contrast, the instability is by definition an expectation over the image of $D_{M}$ and
$D_{M}(f, g) \in [0,1]$ for any $f, g \in \mathcal{N}_n$ (see Definition \ref{def:D_Mapper}), so 
\begin{equation*}
    \text{\rm{InStab}}_\text{\rm{Mapper}}(\left\lbrace Q^i_{n_i}\right\rbrace _{i=1}^t,n, P) \in [0, 1].
\end{equation*}
This shows that the choice of specific values of the 
parameters $\gamma$ and $\delta$ is crucial if we want 
to be able to control the value of instability,
it particularly important to be able to obtain
\begin{equation*}
    \inf_{\gamma>0} Bound_{D_\partial}, 
    \hspace{8mm}
    \text{and}
    \hspace{8mm}
    \inf_{\gamma>0, \delta\in(0,1)} Bound_{\phi}.
\end{equation*}
In the remainder of the section,
we discuss choices of  parameters $\gamma$ and $\delta$
for which tight bounds are attained.

\begin{remark}\label{rmk:Effect_of_varying_gamma_delta} 
    We can make the following simple observation about varying $\gamma$ and $\delta$.
    \begin{enumerate}
        \item 
        As $\gamma$ increases, $P\left(T_\gamma(f)\right)$ increases and $P \left( D_\partial(f^n, f) > \gamma \right)$ decreases. 
        \item
        Analogously, as $\gamma$ increases,
        each $S_{P_i}^{Q^i}(\gamma)$ increases, forcing  
        $N^i(S_{P_i}^{Q^i}(\gamma), \delta)$ to increase, with the overall effect of making $\phi$ smaller.
        However, increasing $\gamma$ also increases $P\left(T_\gamma(f) \right)$.
        \item
        Similarly, as $\delta$ grows, each $N^i(S_{P_i}^{Q^i}(\gamma), \delta)$ decreases, which diminishes the value of $\phi$. 
        However, when  $\delta$ grows, the 
        value of $\left( 1 - \delta \right)$ gets smaller, which increases the value of $\phi$.
        \end{enumerate}
\end{remark}

From Remark \ref{rmk:Effect_of_varying_gamma_delta}, we see that in general there is no straightforward way to identify optimal values of $\gamma$ and $\delta$.
However, the following Corollary of Theorems \ref{thm:Bound_depending_mostly_on_gamma} and \ref{thm:Bound_depending_on_cover_sampleSize_etc} shows 
that to obtain useful boundaries we need to consider small 
values of $\gamma$. 

\begin{corollary}\label{cor:BoundedX_implies_bound_over_1}
    If $\mathcal{X}$ is bounded then there exists some $\Gamma>0$ such that for $\gamma \geq  \Gamma$, we have
    \begin{equation*}
        1 \leq Bound_{D_\partial}(\gamma) \leq Bound_\phi (\gamma, \delta),
    \end{equation*}
    for all $\delta\in(0,1).$
\end{corollary}

\begin{proof}
    If $\mathcal{X}$ is bounded, then there is some $\gamma>0$ such that
    $P \left(T_\gamma(f) \right) \geq \frac{1}{2}$, hence the corollary follows from Theorem \ref{thm:Bound_depending_mostly_on_gamma} and \ref{thm:Bound_depending_on_cover_sampleSize_etc}.
\end{proof}

Since $\text{\rm{InStab}}_\text{\rm{Mapper}}(\left\lbrace Q^i_{n_i}\right\rbrace _{i=1}^t, n, P) \leq 1$, an upper bound above $1$ gives no information.
A consequence of Corollary \ref{cor:BoundedX_implies_bound_over_1} is that large values of $\gamma$ produce such large bounds. 
In the next theorem we show that under reasonable conditions selecting suitable $\gamma>0$, $\phi >0$ with a large enough $n$, make $Bound_{D_\partial}$ arbitrarily close to 0 and therefore to the instability.
In particular theorem includes clustering instability.

\begin{definition}
    We call the pair $(P,Q)$ consisting of a probability measure on metric space $(\mathcal{X},D)$ and a clustering quality function $Q$ on point samples $X\in \mathcal{X}^n$, a \emph{proper pair}
    if all decision boundaries of the clustering function in the image of $C$ (the associated clustering function of $Q$, see (\ref{def:OptimalCluster})) are of zero mass with respect to $P$. 
\end{definition}

In most applications a proper pair would be expected.
For example this is the case on subsets of $\mathbb{R}^a$, if the probability measure is obtained form a continuous probability distribution and the boundaries of $f_i$ are possibly empty finite unions of Jordan arcs. 

\begin{remark}\label{rmk:ProperPair}
    In particular a proper pair implies that for $\gamma>0$ and optimal clustering function $f$, tube $T_{\gamma}(f)$ may be of arbitrarily small mass.
    The clustering boundary $\partial f$ is by definition a finite union of boundaries, hence nowhere dense.
    Therefore if a nonzero lower bound existed on $P(T_{\gamma}(f))$,
    then for any sequence $\gamma_j$ such that $\gamma_j\to 0$ as $j \to \infty$,
    we would have that $\partial f =\bigcap_{j\in \mathbb{N}}T_{\gamma_j}(f)$ is of nonzero mass.
\end{remark}

\begin{theorem}\label{thm:MapperStability}
    Given the assumptions of Remark \ref{rmk:Conditions}, Proposition \ref{prop:MinFunctionLargeN}
    and that each $U_i$ is a bounded, connected, not all of $\mathcal{X}$
    and that each $(P_i,Q^i)$ is a proper pair on $(\mathcal{X},D)$.
    Then for each $1 > \epsilon >0$, there is a $\gamma>0$ and $N\in\mathbb{N}$, such that for $n\geq N$ the instability of the Mapper algorithm satisfies
    \begin{equation}\label{eq:FullInequlity}
        0 \leq \text{\rm{InStab}}_\text{\rm{Mapper}}(\left\lbrace Q^i_{n_i}\right\rbrace _{i=1}^t,n, P) \leq Bound_{D_\partial}(\gamma) \leq \epsilon.
    \end{equation}
    If we remove the condition that each $U_i$ has nonzero mass and weaken  boundedness of $U_i$ to the assumption of bounded support of $P$ and on subset of $\mathbb{R}^a$ removing the compactness assumption,
     we obtain that
    \begin{equation*}
        \text{\rm{InStab}}_\text{\rm{Mapper}}(\left\lbrace Q^i_{n_i}\right\rbrace _{i=1}^t,n, P) \xrightarrow[n \rightarrow \infty ]{} 0.
    \end{equation*}
    In particular for clustering instability $\text{\rm{InStab}}_\text{\rm{Clustering}}( Q_{n}, P)$, when $t=1$ and $U_1=\mathcal{X}$, if there are always at least two clusters, we retain the same result.
\end{theorem}

    \begin{proof}
        Pick $1 > \epsilon >0$ and recall that 
        \begin{equation*}
            Bound_{D_\partial}(\gamma) = 2P(T_\gamma(f)) + 2P \left(  D_\partial(f^n, f) > \gamma \right) + 2P(n_i = 0).
        \end{equation*}
        Since each $U_i$ has nonzero mass it is clear that we choose $N'\in \mathcal{N}$ such that for all $n \geq N'$
        \begin{equation*}
            2P(n_i = 0) \leq \frac{\epsilon}{3}.
        \end{equation*}
        Following Remark \ref{rmk:ProperPair}, since $(P_i,Q^i)$ is a proper probability measure, $P_i(T_\gamma(f_i))$ becomes arbitrarily small as $\gamma$ goes to zero.
        By (\ref{eq:tube_as_union}), we have $T_\gamma(f) =  \bigcup_{i=1}^t T_\gamma(f_i)$,
        so we may choose $\gamma>0$ so that
        \begin{equation*}\label{eq:GammaTubeBound}
            2P(T_\gamma(f)) \leq \frac{\epsilon}{3}.
        \end{equation*}
        By Proposition \ref{prop:MinFunctionLargeN}, we have $P(D_{\partial}(f_i^n,f_i) \leq \gamma) \xrightarrow[n \rightarrow \infty ]{} 1$ and, in addition by (\ref{eq:D_boundary_as_max}), we also have
        $D_\partial (f, g) = \max_{i=1,\dots,t} D_\partial (f_i, g_i)$.
        Therefore, we may choose $N\in \mathbb{N}$, with $N\geq N'$ such that for all $n \geq N$
        \begin{equation*}
            2P \left(  D_\partial(f^n, f) > \gamma \right) \leq \frac{\epsilon}{3},
        \end{equation*}
        which proves (\ref{eq:FullInequlity}).
        By construction, $0 \leq \text{\rm{InStab}}_\text{\rm{Mapper}} \leq Bound_{D_\partial}(\gamma)$, so
        \begin{equation*}
        \text{\rm{InStab}}_\text{\rm{Mapper}}(\left\lbrace Q^i_{n_i}\right\rbrace _{i=1}^t,n, P) \xrightarrow[n \rightarrow \infty ]{} 0.
        \end{equation*}
        We may drop the condition on $U_i$ having nonzero mass, since if $P(U_i)=0$ then the clustering procedure must return the same result on any sample so $U_i$ dose not contribute to the instability.
        As pointed out in Remark \ref{rmk:pseudo-distance}, if the support of $P$ is bounded and $U_i$ is unbounded, then we may restrict $\mathcal{X}$ and $U_i$ to some bounded subset containing the support of $P$.
        If $U_i$ is a subset of $\mathbb{R}^a$ then we may take its closure, so this bounded subset may also be assumed to be closed, hence compact.
        These alterations of the cover do not change the value of the instability while allowing $Bound_{D_\partial}(\gamma)$ to be well defined.
        
        Dropping the restriction that no $U_i$ is all of $\mathcal{X}$ allows us to consider the clustering case where $t=1$ and $U_i=\mathcal{X}$.
        However as noted in (\ref{eq:ExtraSummand}), this means that $Bound_{D_\partial}(\gamma)$ becomes
        \begin{equation*}
            2P(T_\gamma(f)) + 2P \left(  D_\partial(f^n, f) > \gamma \right) + 2P(f_1^n=c^1_1),
        \end{equation*}
        where $P(f_i^n=c^i_1)$ is the probability that the optimal empirical clustering function has a single cluster.
        We have already shown that for sufficiently large $n$ the first two summands may be made arbitrarily small.
        As also noted in (\ref{eq:ExtraSummandBetter}), assuming there are at least two clusters,
        we may replace $2P(f_1^n=c^1_1)$ by $2P(n \leq 1)$, which is $0$ provide $n \geq 2$.
    \end{proof}

    Considering the Mapper output over the space of possible Mapper parameters, we would expect most choices of parameters to satisfy the conditioned of Theorem \ref{thm:MapperStability}.
    Setting aside conditions on the underlying probability distribution,
    most other conditions can be satisfied by choosing a reasonable Mapper setup, such as the classical Mapper algorithm and a sensible clustering procedure.
    The exception to this is the assumption that the quality functions $Q^i$ has a unique global minimizer.
    However as discussed below Defintion \ref{def:OptimalCluster}, this is most likely cause by a symmetry of $P_i$ in $U_i$, which we might interpret as a transition in the structure of the Mapper output at a particular choice of parameters.
    Therefore following Theorem \ref{thm:MapperStability} and as observed experimentally in Table \ref{fig:ChangingEpsilon} and Figure \ref{fig:ResolutionVsGain}, for a large enough sample size, we would expect the values of instability over the parameter space to form regions of low instability separated by ridges of instability.
    In this sense, as Theorem \ref{thm:MapperStability} justifies the existence of regions of stability, it could be considered a stability theorem for Mapper.
    
    \begin{remark}
       Equation (\ref{eq:FullInequlity}) will not hold if $Bound_{D_\partial}(\gamma)$ is replace with $Bound_{\phi}(\gamma,\delta)$.
       Recall that 
       \begin{equation*}
            Bound_{\phi}(\gamma,\delta) = 2 \left( P(T_\gamma(f)) + 1 - \iota(n) \left(1 - \delta \right)^t \prod_{i=1}^t P(U_i) \cdot P \left( n_i \geq N^i(S_{P_i}^{Q^i}(\gamma), \delta) \right) \right).    
       \end{equation*}
       As shown in the proof of Theorem \ref{thm:MapperStability}, the term $P(T_\gamma(f))$ can be made arbitrarily small for large $n$.
       By (\ref{eq:i(n)Limit}), we have $\iota(n) \xrightarrow[n \rightarrow \infty ]{} 1$.
       Also by construction $P \left( n_i \geq N^i(S_{P_i}^{Q^i}(\gamma), \delta) \right)$ may be arbitrarily close to $1$ if $n$ is large enough. 
       Therefore under the conditions of Theorem \ref{thm:MapperStability}, it follows that 
       \begin{equation*}
           \inf_{\gamma>0,\;\delta \in (0,1)}{Bound_{\phi}(\gamma,\delta)} \xrightarrow[n \rightarrow \infty ]{} 2\left( 1-\prod_{i=1}^t P(U_i) \right)
       \end{equation*}
       and $\prod_{i=1}^t P(U_i)$ is fixed by the choice of cover.
    \end{remark}
    
%%%%%%%%%%%%%%%%%%%%%%%%%%%%%%%%%%%%%
%%%%%%%%%%%%%%%%%%%%%%%%%%%%%%%%%%%%%
%%%%%%%%%%%%%%%%%%%%%%%%%%%%%%%%%%%%%

\section{Experimental tests for the instability of Mapper}\label{sec:Verifying}

In this section we present numerical experiments to investigate and demonstrate the causes of instability given in Remarks \ref{rmk:Reasons_which_produce_instability_1} and \ref{rmk:Reasons_which_produce_instability_2}.
In particular we focus on causes of instability unique to the Mapper algorithm.

{\centering
\begin{table}[ht]
	\begin{tabular}{cc}
		\begin{tabular}{ccc}
			\begin{subfigure}{0.15\textwidth}\centering\includegraphics[width=1\columnwidth]{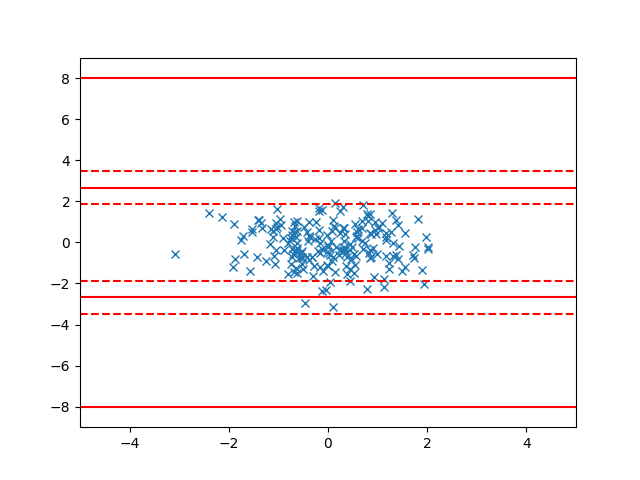}\caption*{200 points}
			\end{subfigure}&
			\begin{subfigure}{0.15\textwidth}\centering\includegraphics[width=1\columnwidth]{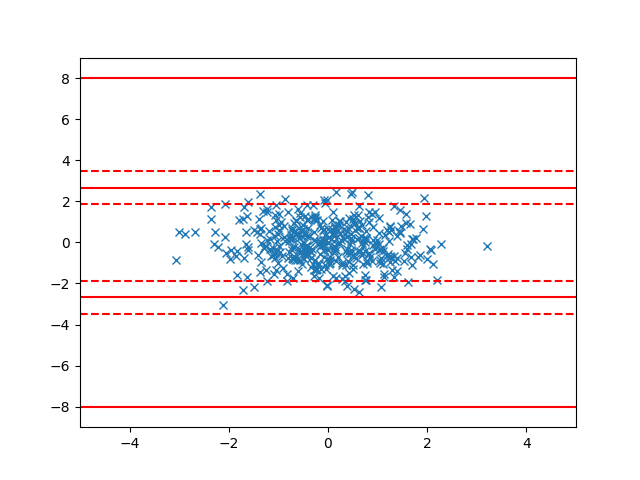}\caption*{400 points}
			\end{subfigure}&
			\begin{subfigure}{0.15\textwidth}\centering\includegraphics[width=1\columnwidth]{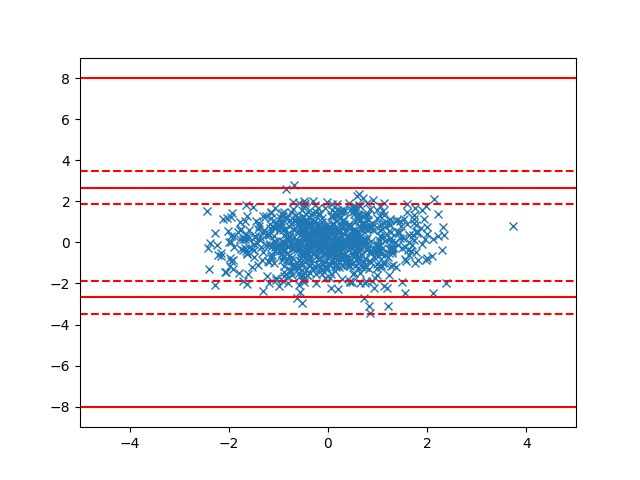}\caption*{800 points}
			\end{subfigure}\\
			\newline
			\begin{subfigure}{0.15\textwidth}\centering\includegraphics[width=1\columnwidth]{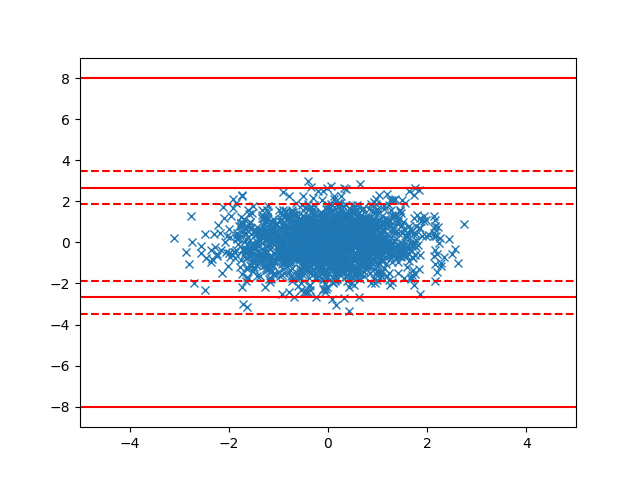}\caption*{1600 points}
			\end{subfigure}&
			\begin{subfigure}{0.15\textwidth}\centering\includegraphics[width=1\columnwidth]{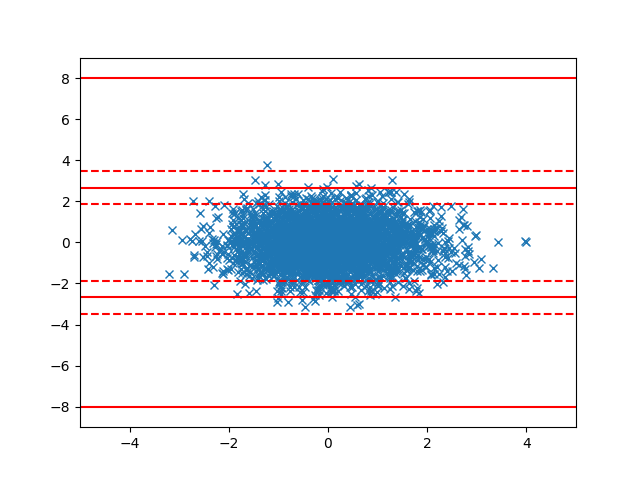}\caption*{3200 points}
			\end{subfigure}&
			\begin{subfigure}{0.15\textwidth}\centering\includegraphics[width=1\columnwidth]{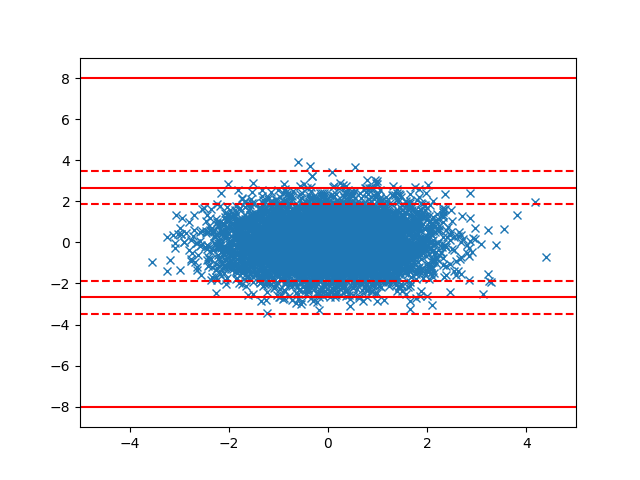}\caption*{6400 points}
			\end{subfigure}\\
			\newline
			\begin{subfigure}{0.15\textwidth}\centering\includegraphics[width=1\columnwidth]{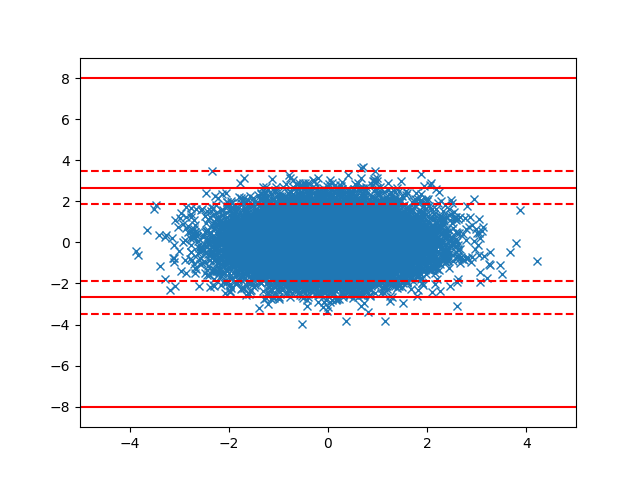}\caption*{12800 points}
			\end{subfigure}&
			\begin{subfigure}{0.15\textwidth}\centering\includegraphics[width=1\columnwidth]{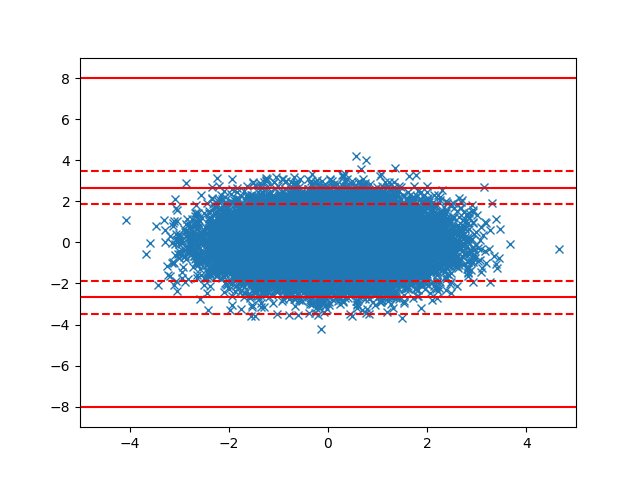}\caption*{25600 points}
			\end{subfigure}&
			\begin{subfigure}{0.15\textwidth}\centering\includegraphics[width=1\columnwidth]{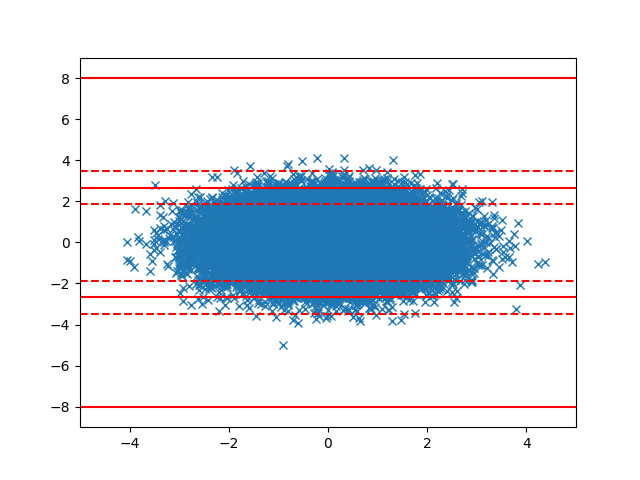}\caption*{51200 points}
			\end{subfigure}\\
		\end{tabular}
		\begin{subfigure}{0.5\textwidth}\centering\includegraphics[width=1\columnwidth]{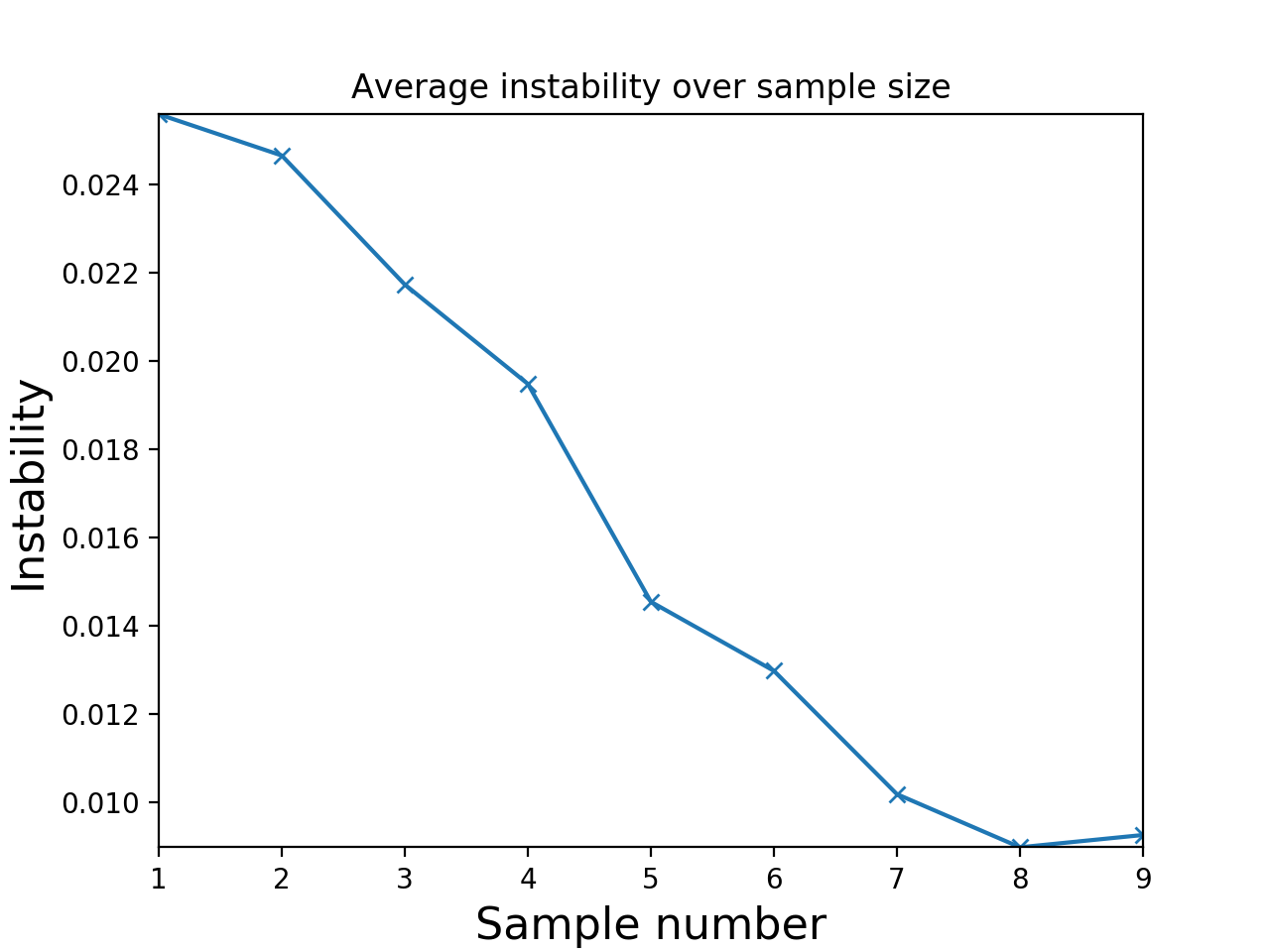}
		\end{subfigure}
	\end{tabular}
	\caption{On the left are nine samples from a bivariant Gaussian distribution centred at the origin. The first sample has $200$ points doubled each time up to $51200$ points. The dashed red lines denote the boundaries of the overlapping bins and the solid lines show the centre of the overlap. On the right we give a plot of the corresponding Mapper instability for each of the datasets on the left. The clustering procedure used was K-means with $K=2$ cluster on $15$ percent overlap between bins, the instabilities were averaged over $30$ different samples and each instability was computed using 40 sub-samples. See \S \ref{sec:ComputingInstability} for details of the procedure.}
	\label{fig:NumberOfPoints}
\end{table}
}

Table \ref{fig:NumberOfPoints} demonstrates a relationship between increasing numbers of points and lower values of instability as discussed in part (A) of Remark \ref{rmk:Reasons_which_produce_instability_2} and Theorem \ref{thm:MapperStability}.
While 
it is  intuitively clear that larger samples should lower 
the instability, experiments of this kind allow one 
to quantify the sample size necessary to ensure that is 
is not, by itself, a source of instability. 

{\centering
\begin{table}[ht]
	\begin{tabular}{cc}
		\begin{tabular}{ccc}
			\begin{subfigure}{0.15\textwidth}\centering\includegraphics[width=0.8\columnwidth]{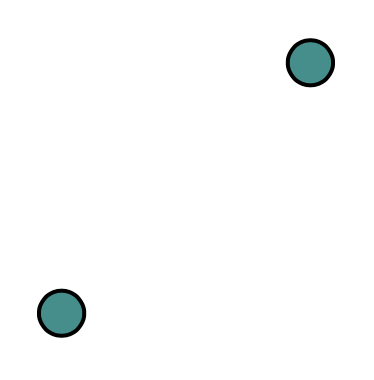}
			\end{subfigure}&
			\begin{subfigure}{0.15\textwidth}\centering\includegraphics[width=0.8\columnwidth]{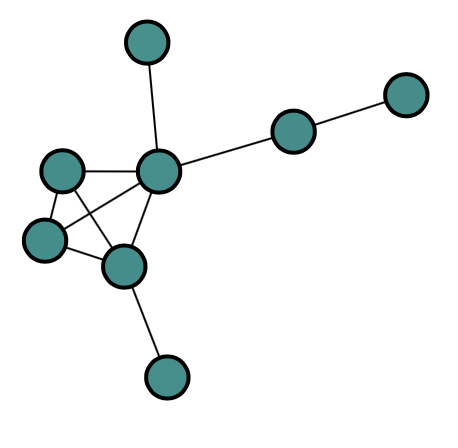}
			\end{subfigure}&
			\begin{subfigure}{0.15\textwidth}\centering\includegraphics[width=0.8\columnwidth]{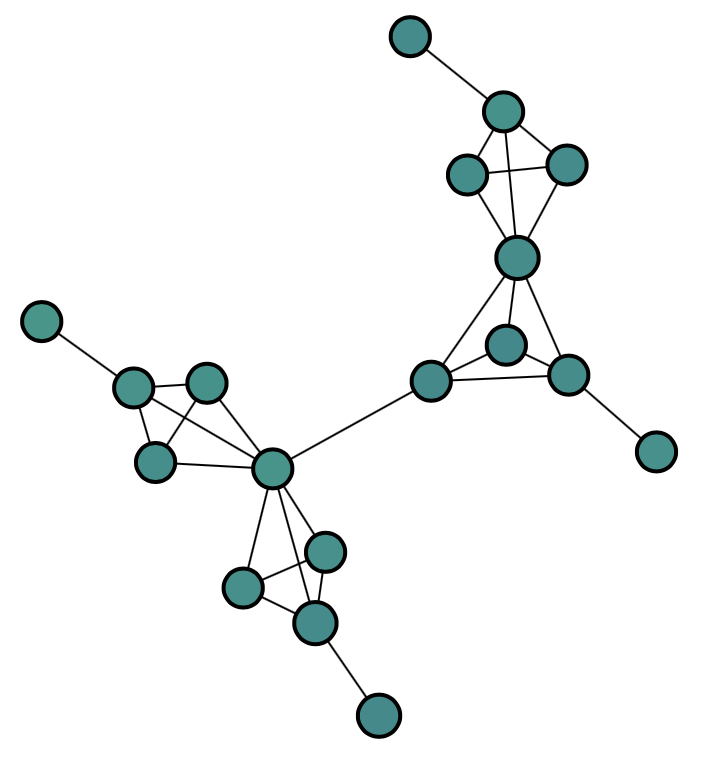}
			\end{subfigure}\\
			\newline
			\begin{subfigure}{0.15\textwidth}\centering\includegraphics[width=1\columnwidth]{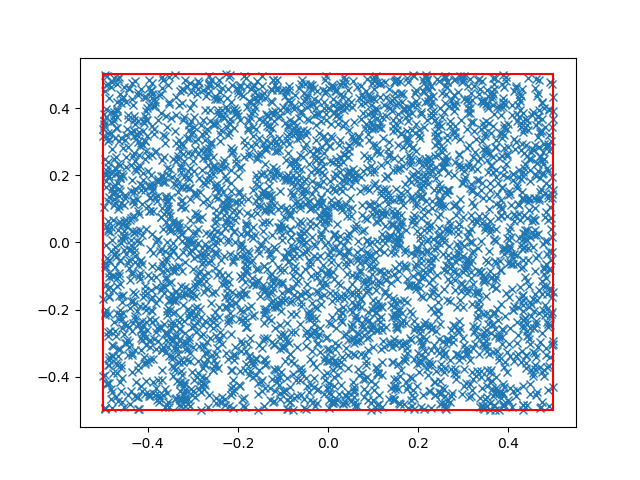}\caption*{1 cube}
			\end{subfigure}&
			\begin{subfigure}{0.15\textwidth}\centering\includegraphics[width=1\columnwidth]{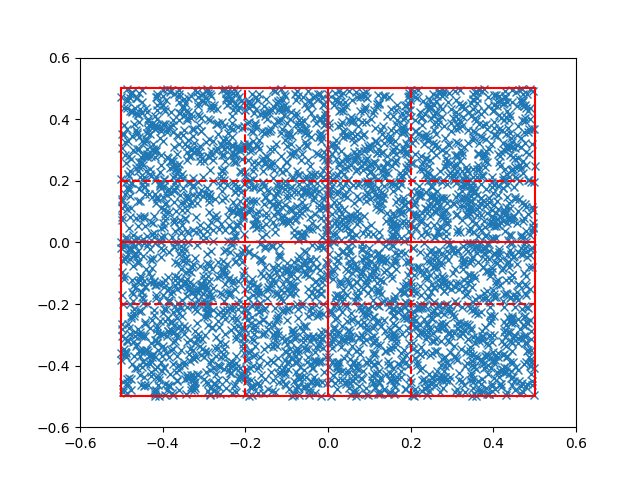}\caption*{4 cubes}
			\end{subfigure}&
			\begin{subfigure}{0.15\textwidth}\centering\includegraphics[width=1\columnwidth]{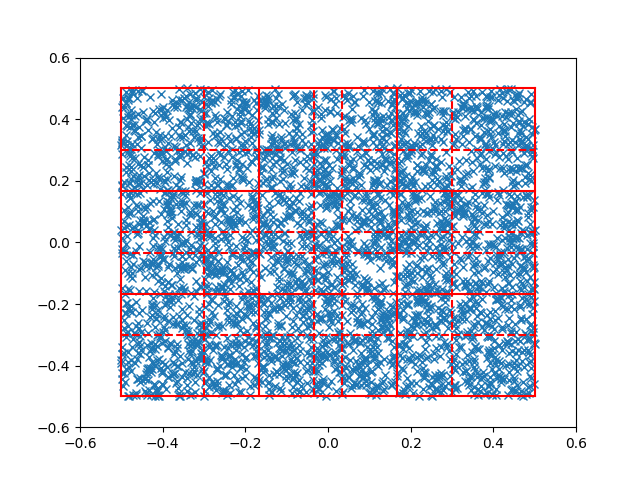}\caption*{9 cubes}
			\end{subfigure}\\
			\newline
			\begin{subfigure}{0.15\textwidth}\centering\includegraphics[width=0.8\columnwidth]{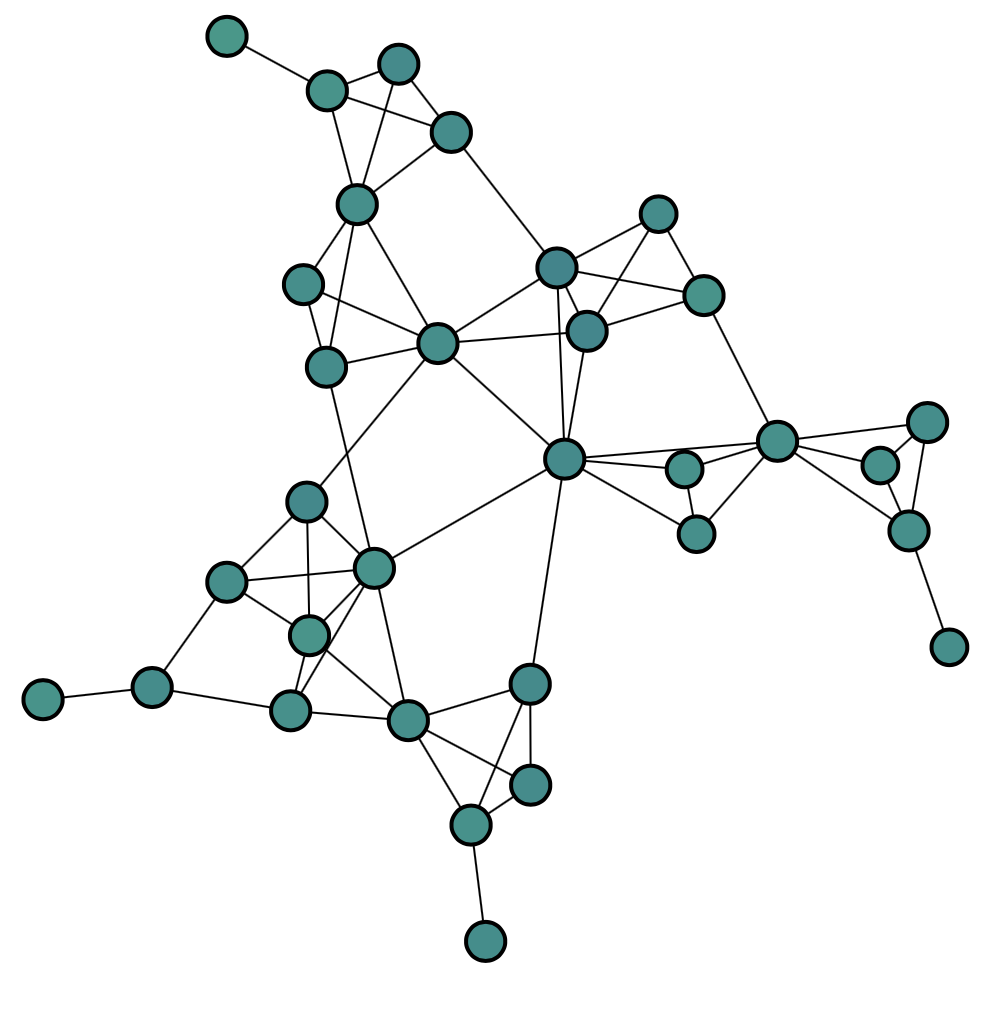}
			\end{subfigure}&
			\begin{subfigure}{0.15\textwidth}\centering\includegraphics[width=0.8\columnwidth]{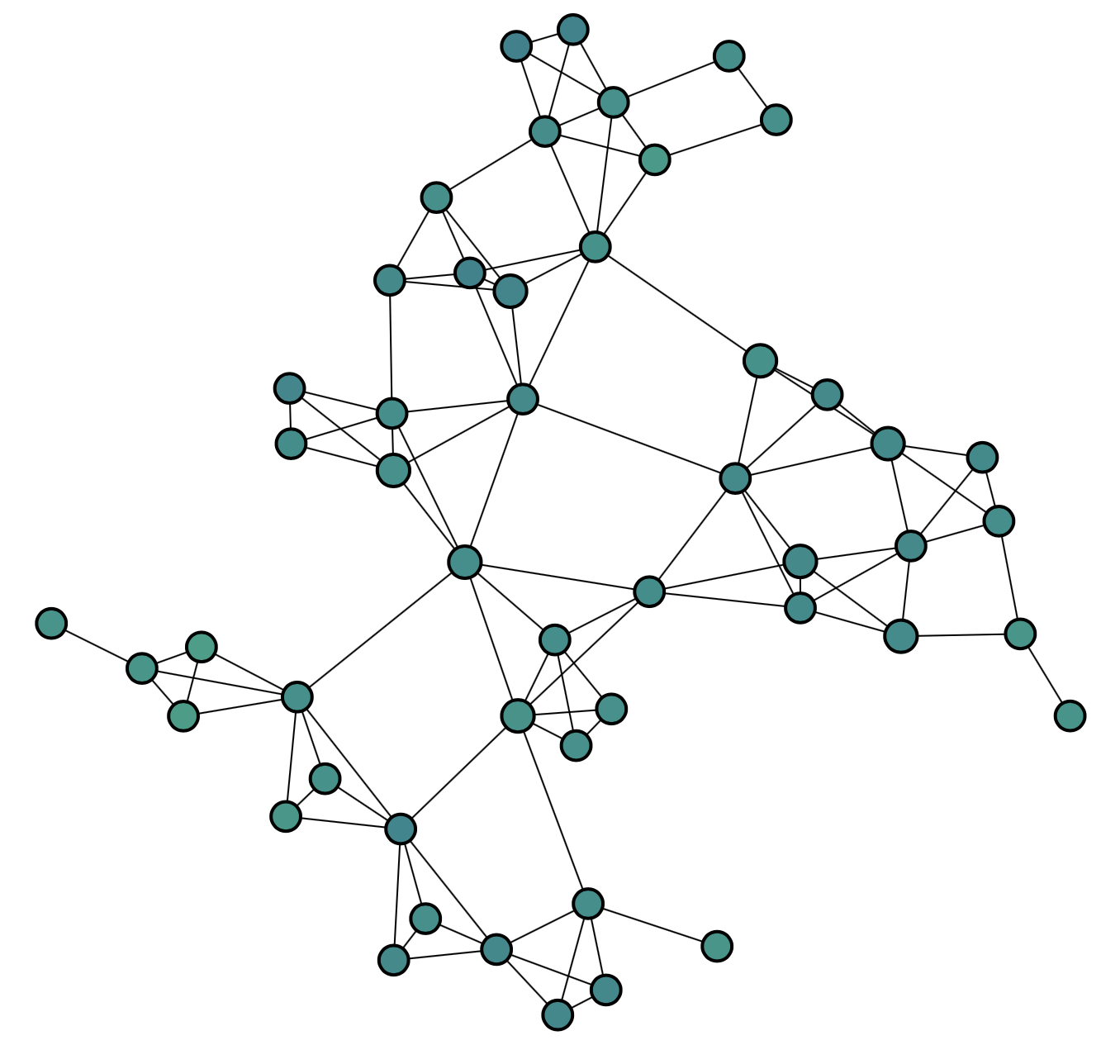}
			\end{subfigure}&
			\begin{subfigure}{0.15\textwidth}\centering\includegraphics[width=0.8\columnwidth]{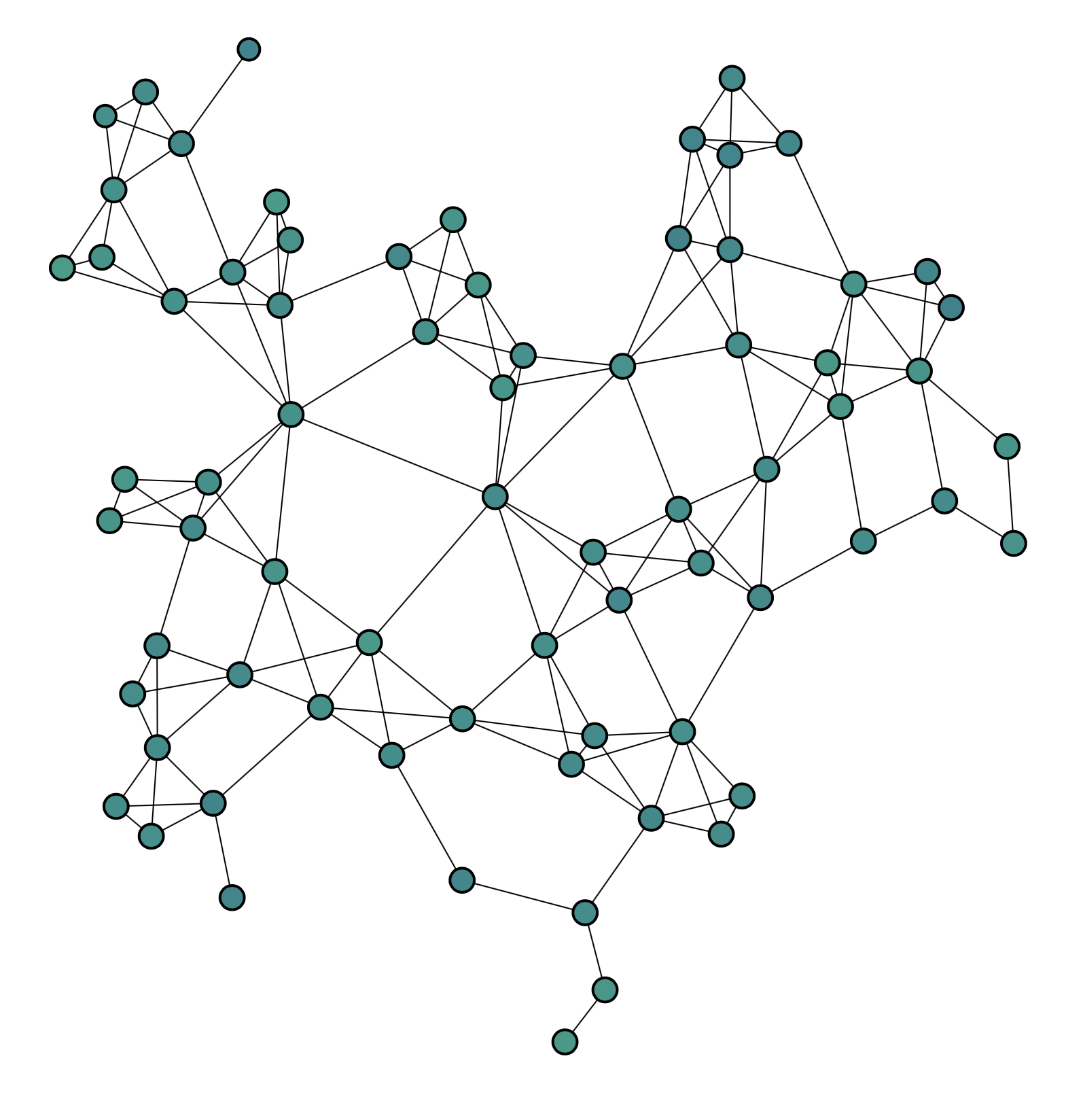}
			\end{subfigure}\\
			\newline
			\begin{subfigure}{0.15\textwidth}\centering\includegraphics[width=1\columnwidth]{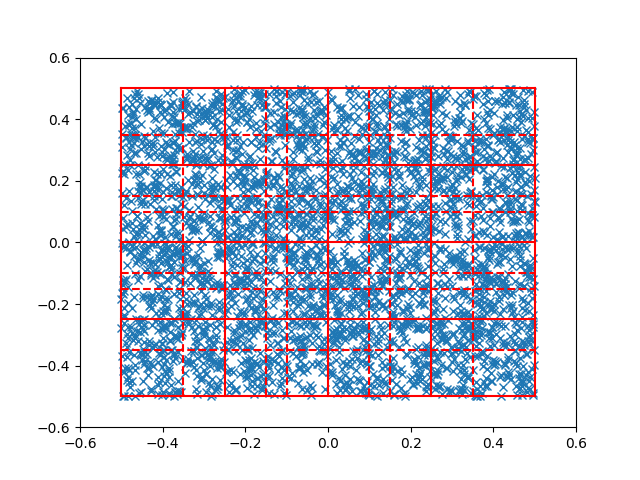}\caption*{16 cubes}
			\end{subfigure}&
			\begin{subfigure}{0.15\textwidth}\centering\includegraphics[width=1\columnwidth]{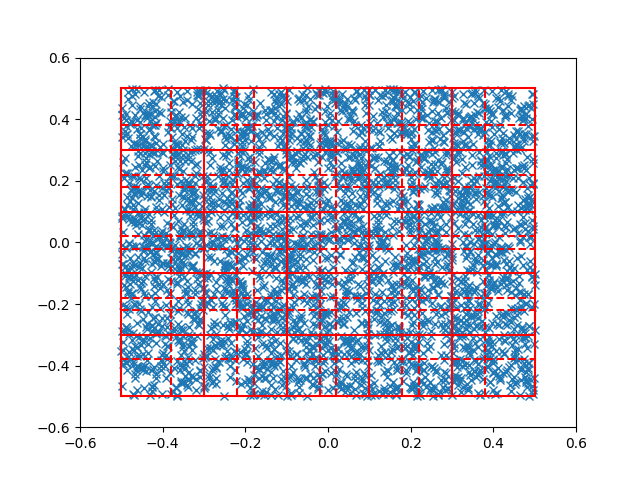}\caption*{25 cubes}
			\end{subfigure}&
			\begin{subfigure}{0.15\textwidth}\centering\includegraphics[width=1\columnwidth]{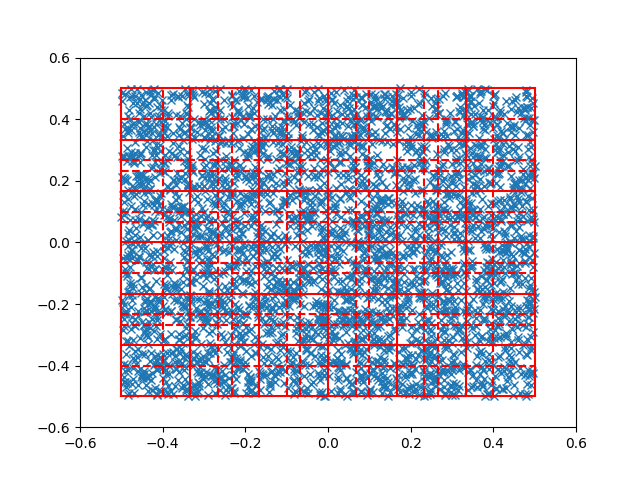}\caption*{36 cubes}
			\end{subfigure}\\
		\end{tabular}
		\begin{subfigure}{0.5\textwidth}\centering\includegraphics[width=1\columnwidth]{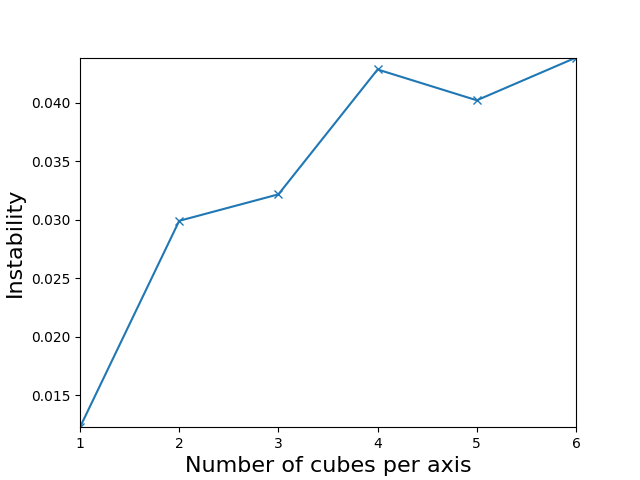}\caption*{Average instability over number of bins}
		\end{subfigure}
	\end{tabular}
	\caption{On the left 6 uniform samples of $3600$ points from a unit square centred at the origin, alongside Mapper graphs obtained by increasing the number of bins per axis from $1$ to $6$. The dashed red lines denote the boundaries of the overlapping bins and the solid lines the centre of the overlap. On the right is a plot of the corresponding Mapper instabilities for each dataset to the left. The clustering procedure used was K-means with $2$ clusters, there was $40$ percent overlap between bins, the instabilities were averaged over $30$ different samples and the instability for each sample was computed using 40 sub-samples. See \S \ref{sec:ComputingInstability} for details of the procedure.}
	\label{fig:NumberOfBins}
\end{table}
}

Table \ref{fig:NumberOfBins}  demonstrates the relationship between increasing numbers of bins and higher values of instability. In each case, we draw the same number of points
from a uniform distribution in a unit square. As the 
sample size is constant, by increasing the 
number of bins, the number of points in each bin decreases,  hence $P(U_i)$ decreases as explained by part (B) of Remark \ref{rmk:Reasons_which_produce_instability_2}. 

{\centering
\begin{table}[ht]
	\begin{tabular}{cc}
		\begin{tabular}{ccc}
			\begin{subfigure}{0.15\textwidth}\centering\includegraphics[width=1\columnwidth]{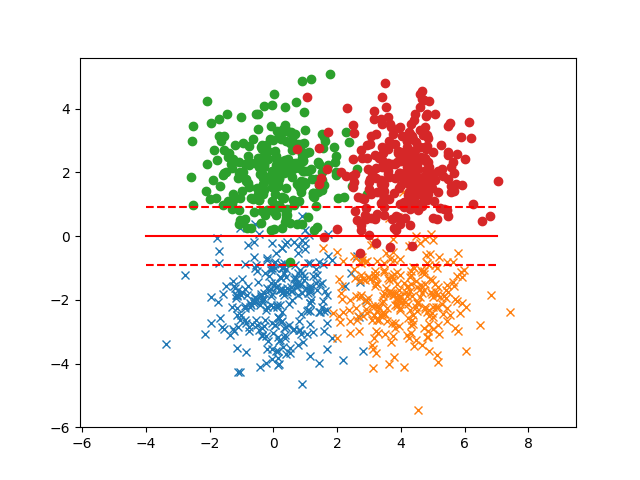}\caption*{Shift 0}
			\end{subfigure}&
			\begin{subfigure}{0.15\textwidth}\centering\includegraphics[width=1\columnwidth]{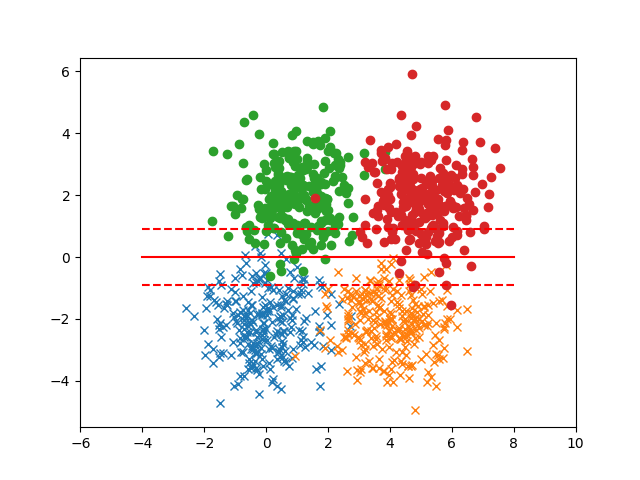}\caption*{Shift 1}
			\end{subfigure}&
			\begin{subfigure}{0.15\textwidth}\centering\includegraphics[width=1\columnwidth]{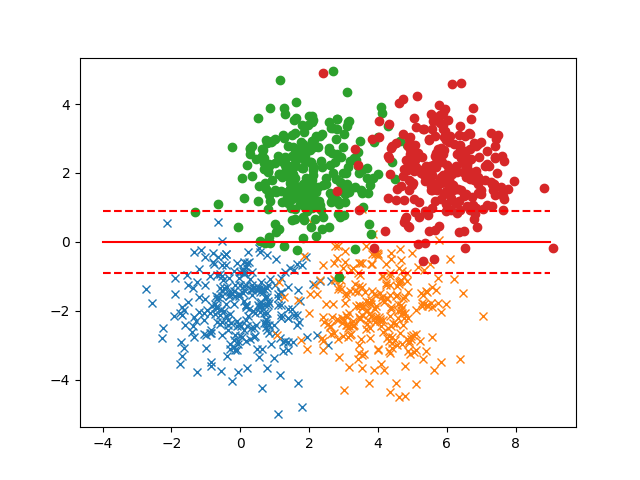}\caption*{Shift 2}
			\end{subfigure}\\
			\newline
			\begin{subfigure}{0.15\textwidth}\centering\includegraphics[width=1\columnwidth]{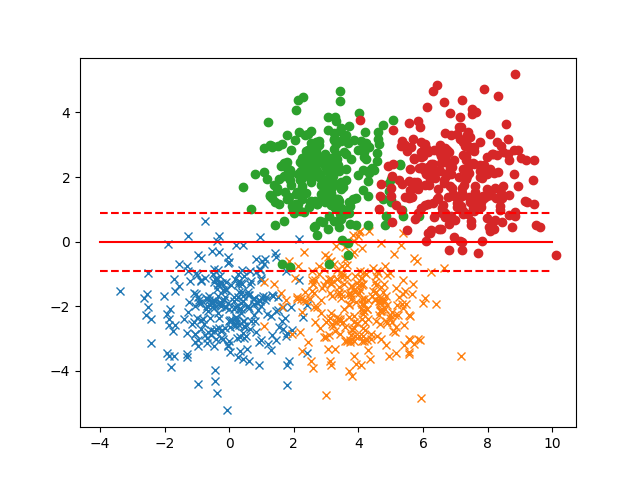}\caption*{Shift 3}
			\end{subfigure}&
			\begin{subfigure}{0.15\textwidth}\centering\includegraphics[width=1\columnwidth]{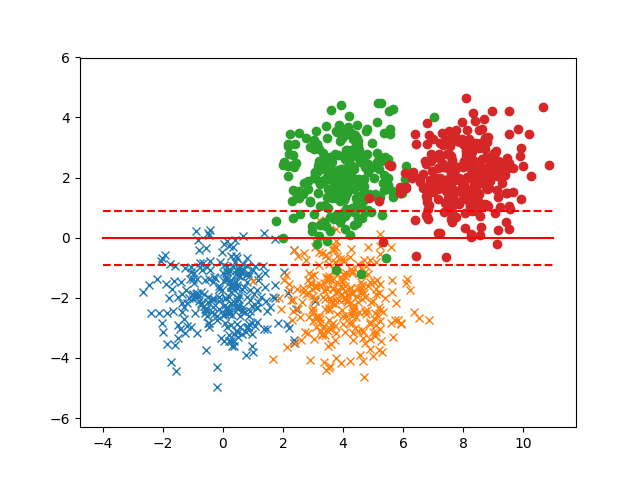}\caption*{Shift 4}
			\end{subfigure}&
			\begin{subfigure}{0.15\textwidth}\centering\includegraphics[width=1\columnwidth]{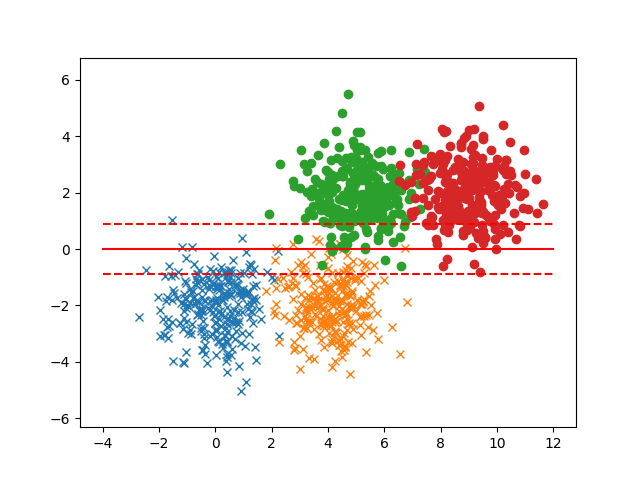}\caption*{Shift 5}
			\end{subfigure}\\
			\newline
			\begin{subfigure}{0.15\textwidth}\centering\includegraphics[width=1\columnwidth]{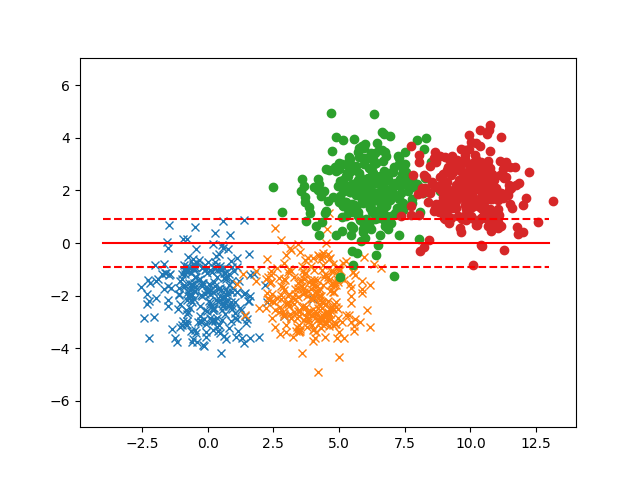}\caption*{Shift 6}
			\end{subfigure}&
			\begin{subfigure}{0.15\textwidth}\centering\includegraphics[width=1\columnwidth]{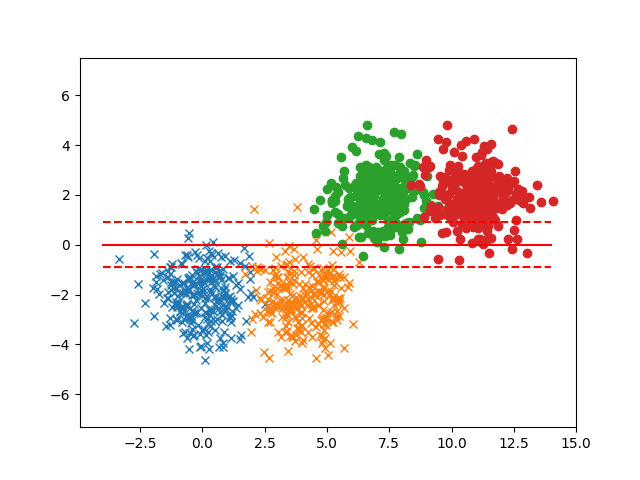}\caption*{Shift 7}
			\end{subfigure}&
			\begin{subfigure}{0.15\textwidth}\centering\includegraphics[width=1\columnwidth]{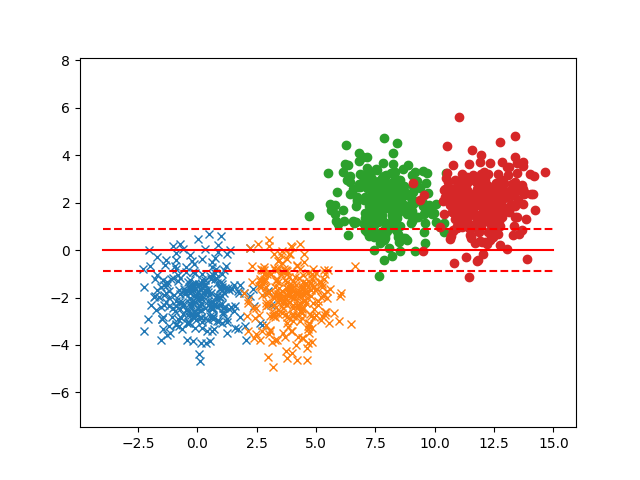}\caption*{Shift 8}
			\end{subfigure}\\
		\end{tabular}
		\begin{subfigure}{0.5\textwidth}\centering\includegraphics[width=1\columnwidth]{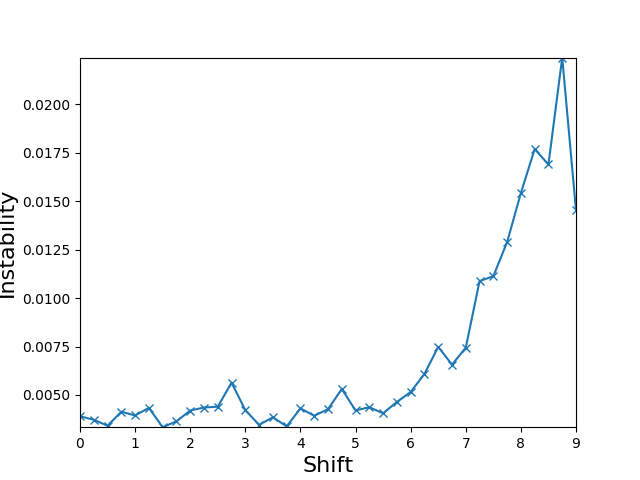}\caption*{Average instability over shift}
		\end{subfigure}
	\end{tabular}
	\caption{On the left, nine examples of samples of 4 bi-variant Gaussian distributions. Each dataset has $256$ points. In each case, the upper pair of Gaussians shifts furtherer to the right. The dashed red lines denote the boundaries of the overlapping bins and the solid lines the centre of the overlap. On the right is a plot of the Mapper instabilities of samples  on the left, with a shift between $0$ and $9$. The clustering procedure used was K-means with $2$ clusters, there was $15$ percent overlap between bins, the instabilities were averaged over $30$ different samples and the instabilities were computed using $32$ different sub-samples. See \S \ref{sec:ComputingInstability} for details of the procedure.}
	\label{fig:ShiftOfBoundaries}
\end{table}
}

Table \ref{fig:ShiftOfBoundaries} demonstrates the relationship between increasing distance between decision boundaries and higher values of instability. Since the overlap between the bins is low, the decision boundary in the upper and lower bins occur roughly between the pair of Gaussians in each bin. Hence, the decision boundaries in different bins move apart as the pairs of Gaussians move away from each over.
Part (d) of Remark \ref{rmk:Reasons_which_produce_instability_1} explains this result.

Observe also that the instability values in Table \ref{fig:ShiftOfBoundaries} are lower than those appearing in Table \ref{fig:NumberOfBins}.
This is a consequence of part (a) of \ref{rmk:Reasons_which_produce_instability_1},
clustering decision boundaries in Table \ref{fig:NumberOfBins} all pass through dense regions of points, while in Table \ref{fig:ShiftOfBoundaries}, the density of points around the decision boundaries is relatively small.

The spikes and ridges in instability that occur around changes in the structure of the Mapper graph in Table \ref{fig:ChangingEpsilon} and Figure \ref{fig:ResolutionVsGain} of section \ref{sec:InitialResults} are inaction to Theorem \ref{thm:MapperStability}, explained by part (e) of Remark \ref{rmk:Reasons_which_produce_instability_1} and part (b) of Remark \ref{rmk:Reasons_which_produce_instability_2}.
This is because at the boundary values of epsilon between structural changes in the graph, the clustering function in some bins changes dramatically with the choices of sample. 

High instability in the top left hand corner of the contour plot in Figure \ref{fig:ResolutionVsGain} and to a lesser extent most of the left hand side of the plot, appears to correspond to Mapper graphs with a fragmented outer circle.
This feature can be explained by (c) of Remark \ref{rmk:Reasons_which_produce_instability_1},
since the low percentage overlap between the bins is causing fragments of the outer circle to partially join together in an inconsistent fashion over varying subsamples.

%%%%%%%%%%%%%%%%%%%%%%%%%%%%%%%%%%%

%%%%%%%%%%%%%%%%%%%%%%%%%%%%%%%%%%%

%%%%%%%%%%%%%%%%%%%%%%%%%%%%%%%%%%%

\section{Conclusions}\label{sec:Conclusion}

In this paper we have demonstrated that changes in  the choice of particular parameters to create Mapper outputs can 
lead to very unstable results. To help alleviate this 
shortcoming, we have created a framework that can be used 
to select regions in the parameter space which are likely 
to create reliable Mapper outputs. 
We have introduced  Mapper instability to provide a numerical
 measure of reliability of a particular Mapper output, especially when considered over a range of parameters.
In particular our construction makes very few assumption on the specifics of the chosen Mapper construction, which makes it applicable
to any Mapper-type algorithm.

We provide theoretical results to describe and explain 
the behaviour of the Mapper instability and in our discussion 
we make very few assumptions about the specifics 
of the structure of the data or the particular cover used 
to create Mapper outputs and show that in most circumstances the instability converges to zero as the sample size is increased. 
We construct explicit bounds  which lead to practical criteria for Mapper instability. We provide a  number of experimental results to further support the practical use our findings. 

An important outcome of our discussion is that we are 
now able to verify when a change in the Mapper output 
is indeed supported by the structure of the data. Specifically, while more complicated Mapper outputs often
suffer from a greater instability, we show that 
when the increase in instability is accompanied by 
low instability, the resulting structure is indeed 
present in the data. 

%%%%%%%%%%%%%%%%%%%%%%%%%%%%%%%%%%%%%
%%%%%%%%%%%%%%%%%%%%%%%%%%%%%%%%%%%%%
%%%%%%%%%%%%%%%%%%%%%%%%%%%%%%%%%%%%%

\section{Appendix}

In this Appendix we justify the assumption that $\mathcal{I}(Q_n)$ of \ref{eq:InstabilityVariable} is a random variable.
In other words,  
$\mathcal{I}(Q_n)$ needs to be a measurable function 
with respect to the Borel $\sigma$-algebra on $\mathbb{R}$ and the product probability measure on $U^n\times U^n$. 
The measurability of $\mathcal{I}(Q_n)$ can be guaranteed provided the empirical quality function satisfies the condition of the following lemma. 

\begin{lemma}\label{lem:MeasureCondition}
	Let $i \colon \mathcal{F}_n^U \times \mathcal{F}_n^U  \to \mathcal{F}_{2n} \times \mathcal{F}_{2n}$ be the inclusion map given by the Voronoi cells (\ref{eq:Voronoi}).
	Then for each pair $(f,g)$ of clustering functions on $2n$ points, if the pre-image of
	\begin{equation}\label{eq:PreImage}
	    U^n\times U^n \to \mathcal{F}_{2n}\times \mathcal{F}_{2n}, \;\;\;
	    (X,X') \mapsto i(C_n(X),C_n(X'))
	\end{equation}
	at $(f,g)$ across $\mathcal{F}_{2n}\times \mathcal{F}_{2n}$ is measurable,
	then $\mathcal{I}(Q_n)$ is a random variable.
\end{lemma}

\begin{proof}
Given that there are only finitely many clustering functions on $2n$ points, the map 
\[
D_m: \mathcal{F}_{2n}\times  \mathcal{F}_{2n} \longrightarrow 
\mathbb{R}
\]
determined by the matching metric is measurable. In consequence, by formula (\ref{eq:InstabilityVariable}), the 
map $\mathcal{I}(Q_n)$ is a random variable when the assumption of the Lemma holds. 
\end{proof}

The condition of the Lemma \ref{lem:MeasureCondition} is easily verified for common quality functions.
For example in the case of nearest neighbour clusterings, given $\epsilon >0$ and clusterings $f,g$ on $2n$ points, we 
can describe the preimage in (\ref{eq:PreImage})
by a set of simple conditions. 
More precisely, the preimage  is  given by the set  of points $((X_1,\dots,X_n),(X'_1,\dots,X'_n)) \in U^n \times U^n$
that satisfy the following. First, we define an $\epsilon$-path in a metric
space to be a sequence of points $(X_1, \dots, X_k)$ such that $D(X_i, X_{i+1}) \leq \epsilon$ for $i=1, \dots, k-1$. 
\begin{enumerate}
    \item For every two points $X_{i_\alpha}$ and $X_{i_\beta}$ chosen 
    from $(X_1, \dots, X_n)$, we have that 
    $f(X_{i_\alpha})=f(X_{i_\beta})$ if and only there is an $\epsilon$-path consisting of points from the list $(X_1, \dots, X_n)$ connecting 
    $X_{i_\alpha}$ and $X_{i_\beta}$.  
    
    \item The function $g$ satisfies an analogous condition on the 
    sequence of points $(X'_1, \dots, X'_n)$. 
    \item For every $i=1, \dots, n$, let $j$ be the smallest index 
    so that the element $X_j$ from the list $(X_1, \dots, X_n)$ minimises 
    the distance $D(X_i', X_k)$, for $k=1, \dots, n$. Then if $f(X_j) = C$ then also $f(X'_i)=C$.
    \item
    An analogous condition holds for the clustering $g$. 
\end{enumerate}

\begin{lemma}
    For each $(f,g)\in \mathcal{F}_{2n}\times \mathcal{F}_{2n}$, 
    the subsets of $(\mathbb{R}^a)^n\times (\mathbb{R}^a)^n$ described above are measurable.
\end{lemma}

\begin{proof}
    Given $(f,g)\in \mathcal{F}_{2n} \times \mathcal{F}_{2n}$,
    consider in turn the restrictions imposed by each of the conditions (1), (2), (3) and (4) given above the lemma.
    
    For (1), since all points sharing a label are connected by $\epsilon$-paths and any two such points are connected by a path, we may consider adding these points inductively in the following way.
    When $n=1$ there is a single point which can take any value in $\mathbb{R}^a$.
    In particular, $\mathbb{R}^a$ is a measurable set.
    Now assume inductively that for some $k=1,\dots,n$, the possible values of the points $X_1,\dots,X_k$ under condition (1) form a measurable set $S_k\subseteq (\mathbb{R}^a)^k$.
    The corresponding set $S_{k+1}$ on points $X_1,\dots ,X_k,X_{k+1}$ is a subspace of $S_k \times \mathbb{R}^a$, under the condition that the final point $X_{k+1}$ is at most a distance of $\epsilon$ from any of the points of $X_1,\dots ,X_k$ with the same label and at least a distance greater than $\epsilon$ form any with a different label.
    More precisely $X_{k+1}$ satisfies that, for each $j=1,\dots,k$,
    \begin{equation*}
        D(X_j,X_{k+1}) \leq \epsilon \text{ if } f(X_j)=f(X_{k+1})
        \text{ and }
        D(X_j,X_{k+1}) \leq \epsilon \text{ if } f(X_j) \neq f(X_{k+1}).
    \end{equation*}
    Note that the possible values of $X_{k+1}$ are nonempty.
    If $X_{k+1}$ shares a label with one of $X_1,\dots,X_k$, then it may for example take the same value and if not the union of the epsilon neighbourhoods of points $X_1,\dots,X_k$ cannot cover all of $\mathbb{R}^a$.
    So the possible values of $X_{k+1}$ are the nonempty intersection of a closed set determined by the first set of strict bounds and an open set determined by the second set of non-strict bounds. 
    Since $S_k$ is measurable, the above inequalities on $X_{k+1}$ extend it to a measurable set $S_{k+1}$. 
    Hence the possible values of $X_1,\dots,X_n$ under condition (1) lie in a measurable set $A=S_n$.
    Analogously we see that the set $B$ of the possible values of $X'_1,\dots ,X'_n$ under condition (2) is measurable.
    
    For each $i=1,\dots ,n$, consider the subsets
    \begin{equation*}
        X_\alpha^i = \{ X_{\alpha_1},\dots X_{\alpha_k} \} \subseteq \{ X_1,\dots,X_n\},
    \end{equation*}
    such that $f(X_{\alpha_j})=f(X'_i)$ for each $j= 1,\dots,k$.
    The Voronoi cells of $X_1,\dots,X_n$ are defined in (\ref{eq:Voronoi}).
    For each $X_{\alpha_j}$ its corresponding cell is obtained by a finite set of inequalities.
    Each inequality is strict if it arises from a pair of points $X_{\alpha_j}^i$ and 
    $X_p$ such that $p<\alpha_j$ and non-strict if $p>\alpha_j$.
    Condition (3) is equivalent to requiring $X'_i$ is contained in the Voronoi cell of one of the elements of $X_{\alpha}^i$.
    We may split the conditions on the Voronoi cells of $X_{\alpha}$ onto those with a strict inequality and those with an non-strict inequality.
    Using a similar inductive augment used when considering condition (1) in the previous part of the proof,
    we may now describe the possible values of $X'_1,\dots,X'_n$ under condition (3)
    as the intersection of an open and closed set, built from the strict and non-strict inequalities respectively to obtain a measurable set $C$.
    Similarly (4) gives us a measurable subset $D$ of $(\mathbb{R}^a)^n$. 
    
    Putting this all together, the subset of points in $(\mathbb{R}^a)^n \times (\mathbb{R}^a)^n$ we wish to describe, is the intersection of the sets $A \times (\mathbb{R}^a)^n$, $(\mathbb{R}^a)^n \times B$, $C \times (\mathbb{R}^a)^n$ and $(\mathbb{R}^a)^n \times D$.
    Since each of $A$, $B$, $C$ and $D$ are measurable sets, the intersection is a measurable set.
\end{proof}

%%%%%%%%%%%%%%%%%%%%%%%%%%%%%%%%%%%%%%%%%%%%%%%%%%%%%%%%%%%%%%%%%%%%%%%%%%%%%%%%%%%%%%%%%

\bibliographystyle{plain}
{\footnotesize
\bibliography{bib}{}
}

\end{document}

%% file: mapperIllustration_wText.pdf_tex
%% Creator: Inkscape inkscape 0.91, www.inkscape.org
%% Accompanies image file 'mapperIllustration_wText.pdf' (pdf, eps, ps)
\begingroup%
  \makeatletter%
  \providecommand\color[2][]{%
    \errmessage{(Inkscape) Color is used for the text in Inkscape, but the package 'color.sty' is not loaded}%
    \renewcommand\color[2][]{}%
  }%
  \providecommand\transparent[1]{%
    \errmessage{(Inkscape) Transparency is used (non-zero) for the text in Inkscape, but the package 'transparent.sty' is not loaded}%
    \renewcommand\transparent[1]{}%
  }%
  \providecommand\rotatebox[2]{#2}%
  \ifx\svgwidth\undefined%
    \setlength{\unitlength}{1417.32283465bp}%
    \ifx\svgscale\undefined%
      \relax%
    \else%
      \setlength{\unitlength}{\unitlength * \real{\svgscale}}%
    \fi%
  \else%
    \setlength{\unitlength}{\svgwidth}%
  \fi%
  \global\let\svgwidth\undefined%
  \global\let\svgscale\undefined%
  \makeatother%
  \begin{picture}(1,0.44)%
    \put(0,0){\includegraphics[width=\unitlength,page=1]{mapperIllustration_wText.pdf}}%
    \put(0.60252447,0.21062138){\color[rgb]{0,0,0}\makebox(0,0)[lb]{\smash{$h$}}}%
    \put(0,0){\includegraphics[width=\unitlength,page=2]{mapperIllustration_wText.pdf}}%
    \put(0.71992438,0.41289951){\color[rgb]{0,0,0}\makebox(0,0)[lb]{\smash{$\mathbb{R}$}}}%
    \put(0.7499251,0.32166817){\color[rgb]{0,0,0.81176471}\makebox(0,0)[lb]{\smash{$I_1$}}}%
    \put(0.74992489,0.1521529){\color[rgb]{0.85098039,0.74901961,0.22352941}\makebox(0,0)[lb]{\smash{$I_3$}}}%
    \put(0.68096262,0.23681961){\color[rgb]{0,0.6745098,0}\makebox(0,0)[lb]{\smash{$I_2$}}}%
    \put(0.68096275,0.06748622){\color[rgb]{1,0,0}\makebox(0,0)[lb]{\smash{$I_4$}}}%
    \put(0.00715715,0.06748626){\color[rgb]{1,0,0}\makebox(0,0)[lb]{\smash{$U_4$}}}%
    \put(0.00715698,0.1521529){\color[rgb]{0.92156863,0.6,0}\makebox(0,0)[lb]{\smash{$U_3$}}}%
    \put(0.00715715,0.23681958){\color[rgb]{0,0.6745098,0}\makebox(0,0)[lb]{\smash{$U_2$}}}%
    \put(0.00715698,0.32148627){\color[rgb]{0,0,0.81176471}\makebox(0,0)[lb]{\smash{$U_1$}}}%
  \end{picture}%
\endgroup%

%% file: ResolutionGain.pdf_tex
%% Creator: Inkscape inkscape 0.91, www.inkscape.org
%% Accompanies image file 'ResolutionGain.pdf' (pdf, eps, ps)
\begingroup%
  \makeatletter%
  \providecommand\color[2][]{%
    \errmessage{(Inkscape) Color is used for the text in Inkscape, but the package 'color.sty' is not loaded}%
    \renewcommand\color[2][]{}%
  }%
  \providecommand\transparent[1]{%
    \errmessage{(Inkscape) Transparency is used (non-zero) for the text in Inkscape, but the package 'transparent.sty' is not loaded}%
    \renewcommand\transparent[1]{}%
  }%
  \providecommand\rotatebox[2]{#2}%
  \ifx\svgwidth\undefined%
    \setlength{\unitlength}{595.27559055bp}%
    \ifx\svgscale\undefined%
      \relax%
    \else%
      \setlength{\unitlength}{\unitlength * \real{\svgscale}}%
    \fi%
  \else%
    \setlength{\unitlength}{\svgwidth}%
  \fi%
  \global\let\svgwidth\undefined%
  \global\let\svgscale\undefined%
  \makeatother%
  \begin{picture}(1,1)%
    \put(0,0){\includegraphics[width=\unitlength,page=1]{ResolutionGain.pdf}}%
    \put(0.0541502,0.29878553){\color[rgb]{0,0,0}\makebox(0,0)[lt]{\begin{minipage}{0.19621169\unitlength}\raggedright 0.00778\end{minipage}}}%
    \put(0,0){\includegraphics[width=\unitlength,page=2]{ResolutionGain.pdf}}%
    \put(0.0706477,0.65060522){\color[rgb]{0,0,0}\makebox(0,0)[lt]{\begin{minipage}{0.14824547\unitlength}\raggedright 0.0226\end{minipage}}}%
    \put(0,0){\includegraphics[width=\unitlength,page=3]{ResolutionGain.pdf}}%
    \put(0.04941519,0.1173565){\color[rgb]{0,0,0}\makebox(0,0)[lt]{\begin{minipage}{0.12448821\unitlength}\raggedright 0.00706\end{minipage}}}%
    \put(0,0){\includegraphics[width=\unitlength,page=4]{ResolutionGain.pdf}}%
    \put(0.04656429,0.66425374){\color[rgb]{0,0,0}\makebox(0,0)[lt]{\begin{minipage}{0.09788004\unitlength}\raggedright \end{minipage}}}%
    \put(0.06937124,0.47279359){\color[rgb]{0,0,0}\makebox(0,0)[lt]{\begin{minipage}{0.09597945\unitlength}\raggedright 0.0171\end{minipage}}}%
    \put(0,0){\includegraphics[width=\unitlength,page=5]{ResolutionGain.pdf}}%
    \put(0.28337295,0.07071628){\color[rgb]{0,0,0}\makebox(0,0)[lt]{\begin{minipage}{0.08881534\unitlength}\raggedright 0.00289\end{minipage}}}%
    \put(0,0){\includegraphics[width=\unitlength,page=6]{ResolutionGain.pdf}}%
    \put(0.20146178,0.12303023){\color[rgb]{0,0,0}\makebox(0,0)[lt]{\begin{minipage}{0.08932743\unitlength}\raggedright 0.00883\end{minipage}}}%
    \put(0,0){\includegraphics[width=\unitlength,page=7]{ResolutionGain.pdf}}%
    \put(0.34649214,0.12774853){\color[rgb]{0,0,0}\makebox(0,0)[lt]{\begin{minipage}{0.07923847\unitlength}\raggedright 0.00578\end{minipage}}}%
    \put(0,0){\includegraphics[width=\unitlength,page=8]{ResolutionGain.pdf}}%
    \put(0.45899085,0.24707572){\color[rgb]{0,0,0}\makebox(0,0)[lt]{\begin{minipage}{0.10453207\unitlength}\raggedright \end{minipage}}}%
    \put(0.46944409,0.07836647){\color[rgb]{0,0,0}\makebox(0,0)[lt]{\begin{minipage}{0.10928356\unitlength}\raggedright 0.0068\end{minipage}}}%
    \put(0,0){\includegraphics[width=\unitlength,page=9]{ResolutionGain.pdf}}%
    \put(0.40862538,0.68448952){\color[rgb]{0,0,0}\makebox(0,0)[lt]{\begin{minipage}{0.16915193\unitlength}\raggedright 0.0149\end{minipage}}}%
    \put(0,0){\includegraphics[width=\unitlength,page=10]{ResolutionGain.pdf}}%
    \put(0.38008656,0.0307711){\color[rgb]{0,0,0}\makebox(0,0)[lt]{\begin{minipage}{0.08935749\unitlength}\raggedright 0.00228\end{minipage}}}%
    \put(0,0){\includegraphics[width=\unitlength,page=11]{ResolutionGain.pdf}}%
    \put(0.50647695,0.83467748){\color[rgb]{0,0,0}\makebox(0,0)[lt]{\begin{minipage}{0.1169144\unitlength}\raggedright 0.024\end{minipage}}}%
    \put(0,0){\includegraphics[width=\unitlength,page=12]{ResolutionGain.pdf}}%
    \put(0.56087102,0.02759558){\color[rgb]{0,0,0}\makebox(0,0)[lt]{\begin{minipage}{0.10148226\unitlength}\raggedright 0.00333\end{minipage}}}%
    \put(0,0){\includegraphics[width=\unitlength,page=13]{ResolutionGain.pdf}}%
    \put(0.62607913,0.12017931){\color[rgb]{0,0,0}\makebox(0,0)[lt]{\begin{minipage}{0.11308469\unitlength}\raggedright 0.00761\end{minipage}}}%
    \put(0,0){\includegraphics[width=\unitlength,page=14]{ResolutionGain.pdf}}%
    \put(0.80584742,0.4635328){\color[rgb]{0,0,0}\makebox(0,0)[lt]{\begin{minipage}{0.12068703\unitlength}\raggedright 0.0186\end{minipage}}}%
    \put(0,0){\includegraphics[width=\unitlength,page=15]{ResolutionGain.pdf}}%
    \put(0.55752304,0.65503057){\color[rgb]{0,0,0}\makebox(0,0)[lt]{\begin{minipage}{0.1273391\unitlength}\raggedright 0.0204\end{minipage}}}%
    \put(0,0){\includegraphics[width=\unitlength,page=16]{ResolutionGain.pdf}}%
    \put(0.21001441,0.81896989){\color[rgb]{0,0,0}\makebox(0,0)[lt]{\begin{minipage}{0.08647655\unitlength}\raggedright 0.0182\end{minipage}}}%
    \put(0,0){\includegraphics[width=\unitlength,page=17]{ResolutionGain.pdf}}%
    \put(0.81059885,0.60482814){\color[rgb]{0,0,0}\makebox(0,0)[lt]{\begin{minipage}{0.1273085\unitlength}\raggedright 0.0233\end{minipage}}}%
    \put(0,0){\includegraphics[width=\unitlength,page=18]{ResolutionGain.pdf}}%
    \put(0.84953274,0.78522051){\color[rgb]{0,0,0}\makebox(0,0)[lt]{\begin{minipage}{0.1273391\unitlength}\raggedright 0.0166\end{minipage}}}%
    \put(0,0){\includegraphics[width=\unitlength,page=19]{ResolutionGain.pdf}}%
    \put(0.64619834,0.8156299){\color[rgb]{0,0,0}\makebox(0,0)[lt]{\begin{minipage}{0.09788008\unitlength}\raggedright 0.0202\end{minipage}}}%
    \put(0,0){\includegraphics[width=\unitlength,page=20]{ResolutionGain.pdf}}%
    \put(0.72199149,0.02208012){\color[rgb]{0,0,0}\makebox(0,0)[lt]{\begin{minipage}{0.13283862\unitlength}\raggedright 0.00289\end{minipage}}}%
    \put(0,0){\includegraphics[width=\unitlength,page=21]{ResolutionGain.pdf}}%
    \put(0.85526243,0.0261005){\color[rgb]{0,0,0}\makebox(0,0)[lt]{\begin{minipage}{0.10833324\unitlength}\raggedright 0.0065\end{minipage}}}%
    \put(0,0){\includegraphics[width=\unitlength,page=22]{ResolutionGain.pdf}}%
    \put(0.90723272,0.38373458){\color[rgb]{0,0,0}\makebox(0,0)[lt]{\begin{minipage}{0.15875369\unitlength}\raggedright 0.0142\end{minipage}}}%
    \put(0,0){\includegraphics[width=\unitlength,page=23]{ResolutionGain.pdf}}%
    \put(0.80785937,0.18176827){\color[rgb]{0,0,0}\makebox(0,0)[lt]{\begin{minipage}{0.15573628\unitlength}\raggedright 0.008\end{minipage}}}%
    \put(0,0){\includegraphics[width=\unitlength,page=24]{ResolutionGain.pdf}}%
    \put(0.84198203,0.31022077){\color[rgb]{0,0,0}\makebox(0,0)[lt]{\begin{minipage}{0.10873779\unitlength}\raggedright 0.00756\end{minipage}}}%
    \put(0,0){\includegraphics[width=\unitlength,page=25]{ResolutionGain.pdf}}%
    \put(0.90456629,0.17701686){\color[rgb]{0,0,0}\makebox(0,0)[lt]{\begin{minipage}{0.10873779\unitlength}\raggedright 0.00733\end{minipage}}}%
    \put(0,0){\includegraphics[width=\unitlength,page=26]{ResolutionGain.pdf}}%
  \end{picture}%
\endgroup%

%% file: d_boundary_a_setting.pdf_tex
%% Creator: Inkscape inkscape 0.91, www.inkscape.org
%% Accompanies image file 'd_boundary_a_setting.pdf' (pdf, eps, ps)
\begingroup%
  \makeatletter%
  \providecommand\color[2][]{%
    \errmessage{(Inkscape) Color is used for the text in Inkscape, but the package 'color.sty' is not loaded}%
    \renewcommand\color[2][]{}%
  }%
  \providecommand\transparent[1]{%
    \errmessage{(Inkscape) Transparency is used (non-zero) for the text in Inkscape, but the package 'transparent.sty' is not loaded}%
    \renewcommand\transparent[1]{}%
  }%
  \providecommand\rotatebox[2]{#2}%
  \ifx\svgwidth\undefined%
    \setlength{\unitlength}{283.46456426bp}%
    \ifx\svgscale\undefined%
      \relax%
    \else%
      \setlength{\unitlength}{\unitlength * \real{\svgscale}}%
    \fi%
  \else%
    \setlength{\unitlength}{\svgwidth}%
  \fi%
  \global\let\svgwidth\undefined%
  \global\let\svgscale\undefined%
  \makeatother%
  \begin{picture}(1,1.17594849)%
    \put(0,0){\includegraphics[width=\unitlength,page=1]{d_boundary_a_setting.pdf}}%
    \put(0.5233355,0.36239975){\color[rgb]{0,0,1}\makebox(0,0)[lb]{\smash{$\partial(f)$}}}%
    \put(0.1970184,0.36239975){\color[rgb]{1,0,0}\makebox(0,0)[lb]{\smash{$\partial(g)$}}}%
    \put(0.26115671,0.0355307){\color[rgb]{0,0,0}\makebox(0,0)[lb]{\smash{$D_{\partial}(f,g)$}}}%
    \put(0.25331019,0.94230881){\color[rgb]{0,0,0}\makebox(0,0)[lb]{\smash{\textbf{1}}}}%
    \put(0.72890106,0.82039292){\color[rgb]{0,0,0}\makebox(0,0)[lb]{\smash{\textbf{2}}}}%
  \end{picture}%
\endgroup%

%% file: d_boundary_b_TubeAroundG.pdf_tex
%% Creator: Inkscape inkscape 0.91, www.inkscape.org
%% Accompanies image file 'd_boundary_b_TubeAroundG.pdf' (pdf, eps, ps)
\begingroup%
  \makeatletter%
  \providecommand\color[2][]{%
    \errmessage{(Inkscape) Color is used for the text in Inkscape, but the package 'color.sty' is not loaded}%
    \renewcommand\color[2][]{}%
  }%
  \providecommand\transparent[1]{%
    \errmessage{(Inkscape) Transparency is used (non-zero) for the text in Inkscape, but the package 'transparent.sty' is not loaded}%
    \renewcommand\transparent[1]{}%
  }%
  \providecommand\rotatebox[2]{#2}%
  \ifx\svgwidth\undefined%
    \setlength{\unitlength}{283.46456426bp}%
    \ifx\svgscale\undefined%
      \relax%
    \else%
      \setlength{\unitlength}{\unitlength * \real{\svgscale}}%
    \fi%
  \else%
    \setlength{\unitlength}{\svgwidth}%
  \fi%
  \global\let\svgwidth\undefined%
  \global\let\svgscale\undefined%
  \makeatother%
  \begin{picture}(1,1.17594849)%
    \put(0,0){\includegraphics[width=\unitlength,page=1]{d_boundary_b_TubeAroundG.pdf}}%
    \put(0.63021153,0.55286942){\color[rgb]{0,0,0}\makebox(0,0)[lb]{\smash{$T_\gamma(g)$}}}%
    \put(0.35146778,0.05246384){\color[rgb]{0,0,0}\makebox(0,0)[lb]{\smash{$\gamma$}}}%
    \put(0,0){\includegraphics[width=\unitlength,page=2]{d_boundary_b_TubeAroundG.pdf}}%
  \end{picture}%
\endgroup%

%% file: d_boundary_c_TubeAroundF.pdf_tex
%% Creator: Inkscape inkscape 0.91, www.inkscape.org
%% Accompanies image file 'd_boundary_c_TubeAroundF.pdf' (pdf, eps, ps)
\begingroup%
  \makeatletter%
  \providecommand\color[2][]{%
    \errmessage{(Inkscape) Color is used for the text in Inkscape, but the package 'color.sty' is not loaded}%
    \renewcommand\color[2][]{}%
  }%
  \providecommand\transparent[1]{%
    \errmessage{(Inkscape) Transparency is used (non-zero) for the text in Inkscape, but the package 'transparent.sty' is not loaded}%
    \renewcommand\transparent[1]{}%
  }%
  \providecommand\rotatebox[2]{#2}%
  \ifx\svgwidth\undefined%
    \setlength{\unitlength}{283.46456426bp}%
    \ifx\svgscale\undefined%
      \relax%
    \else%
      \setlength{\unitlength}{\unitlength * \real{\svgscale}}%
    \fi%
  \else%
    \setlength{\unitlength}{\svgwidth}%
  \fi%
  \global\let\svgwidth\undefined%
  \global\let\svgscale\undefined%
  \makeatother%
  \begin{picture}(1,1.17594849)%
    \put(0,0){\includegraphics[width=\unitlength,page=1]{d_boundary_c_TubeAroundF.pdf}}%
    \put(0.70381859,0.78044358){\color[rgb]{0,0,0}\makebox(0,0)[lb]{\smash{$T_\gamma(f)$}}}%
    \put(0,0){\includegraphics[width=\unitlength,page=2]{d_boundary_c_TubeAroundF.pdf}}%
    \put(0.385336,0.05246719){\color[rgb]{0,0,0}\makebox(0,0)[lb]{\smash{$\gamma$}}}%
  \end{picture}%
\endgroup%

%% file: FarDecsionBoundariesSquare.pdf_tex
%% Creator: Inkscape inkscape 0.91, www.inkscape.org
%% Accompanies image file 'FarDecsionBoundariesSquare.pdf' (pdf, eps, ps)
\begingroup%
  \makeatletter%
  \providecommand\color[2][]{%
    \errmessage{(Inkscape) Color is used for the text in Inkscape, but the package 'color.sty' is not loaded}%
    \renewcommand\color[2][]{}%
  }%
  \providecommand\transparent[1]{%
    \errmessage{(Inkscape) Transparency is used (non-zero) for the text in Inkscape, but the package 'transparent.sty' is not loaded}%
    \renewcommand\transparent[1]{}%
  }%
  \providecommand\rotatebox[2]{#2}%
  \ifx\svgwidth\undefined%
    \setlength{\unitlength}{1034.64566929bp}%
    \ifx\svgscale\undefined%
      \relax%
    \else%
      \setlength{\unitlength}{\unitlength * \real{\svgscale}}%
    \fi%
  \else%
    \setlength{\unitlength}{\svgwidth}%
  \fi%
  \global\let\svgwidth\undefined%
  \global\let\svgscale\undefined%
  \makeatother%
  \begin{picture}(1,0.36986301)%
    \put(0,0){\includegraphics[width=\unitlength,page=1]{FarDecsionBoundariesSquare.pdf}}%
    \put(0.72372032,-0.69992207){\color[rgb]{0,0,0}\makebox(0,0)[lt]{\begin{minipage}{0.3159121\unitlength}\raggedright \end{minipage}}}%
    \put(0.06611938,0.28040699){\color[rgb]{0,0,0}\makebox(0,0)[lt]{\begin{minipage}{0.17718756\unitlength}\raggedright $f_1\neq g_1$\end{minipage}}}%
    \put(0.93359193,-0.73747804){\color[rgb]{0,0,0}\makebox(0,0)[lt]{\begin{minipage}{0.05522931\unitlength}\raggedright \end{minipage}}}%
    \put(0.56790254,0.21197475){\color[rgb]{0,0,0}\makebox(0,0)[lt]{\begin{minipage}{0.17718756\unitlength}\raggedright $f'_1\neq g'_1$\end{minipage}}}%
    \put(0.34449956,0.12103844){\color[rgb]{0,0,0}\makebox(0,0)[lt]{\begin{minipage}{0.17718756\unitlength}\raggedright $f_2\neq g_2$\end{minipage}}}%
    \put(0.85325917,0.20238616){\color[rgb]{0,0,0}\makebox(0,0)[lt]{\begin{minipage}{0.17718756\unitlength}\raggedright $f'_2\neq g'_2$\end{minipage}}}%
    \put(0.08279252,0.16402297){\color[rgb]{0,0,0}\makebox(0,0)[lt]{\begin{minipage}{0.05779855\unitlength}\raggedright \end{minipage}}}%
    \put(0.09510111,0.15162018){\color[rgb]{0,0,0}\makebox(0,0)[lt]{\begin{minipage}{0.15465019\unitlength}\raggedright $U_1$\end{minipage}}}%
    \put(0.37476823,0.24067689){\color[rgb]{0,0,0}\makebox(0,0)[lt]{\begin{minipage}{0.1546502\unitlength}\raggedright $U_2$\end{minipage}}}%
    \put(0.6013075,0.11886269){\color[rgb]{0,0,0}\makebox(0,0)[lt]{\begin{minipage}{0.15465019\unitlength}\raggedright $U_1$\end{minipage}}}%
    \put(0.87936542,0.28908701){\color[rgb]{0,0,0}\makebox(0,0)[lt]{\begin{minipage}{0.1546502\unitlength}\raggedright $U_2$\end{minipage}}}%
    \put(0,0){\includegraphics[width=\unitlength,page=2]{FarDecsionBoundariesSquare.pdf}}%
  \end{picture}%
\endgroup%

%% file: uniquenessOfMinimizers.pdf_tex
%% Creator: Inkscape inkscape 0.91, www.inkscape.org
%% Accompanies image file 'uniquenessOfMinimizers.pdf' (pdf, eps, ps)
\begingroup%
  \makeatletter%
  \providecommand\color[2][]{%
    \errmessage{(Inkscape) Color is used for the text in Inkscape, but the package 'color.sty' is not loaded}%
    \renewcommand\color[2][]{}%
  }%
  \providecommand\transparent[1]{%
    \errmessage{(Inkscape) Transparency is used (non-zero) for the text in Inkscape, but the package 'transparent.sty' is not loaded}%
    \renewcommand\transparent[1]{}%
  }%
  \providecommand\rotatebox[2]{#2}%
  \ifx\svgwidth\undefined%
    \setlength{\unitlength}{2154.33070866bp}%
    \ifx\svgscale\undefined%
      \relax%
    \else%
      \setlength{\unitlength}{\unitlength * \real{\svgscale}}%
    \fi%
  \else%
    \setlength{\unitlength}{\svgwidth}%
  \fi%
  \global\let\svgwidth\undefined%
  \global\let\svgscale\undefined%
  \makeatother%
  \begin{picture}(1,0.20394737)%
    \put(0,0){\includegraphics[width=\unitlength,page=1]{uniquenessOfMinimizers.pdf}}%
    \put(0.18368096,0.02012296){\color[rgb]{0,0,0}\makebox(0,0)[lt]{\begin{minipage}{0.0978177\unitlength}\raggedright $\gamma$\end{minipage}}}%
    \put(0,0){\includegraphics[width=\unitlength,page=2]{uniquenessOfMinimizers.pdf}}%
    \put(0.26753422,0.02532179){\color[rgb]{0,0,0}\makebox(0,0)[lt]{\begin{minipage}{0.18020958\unitlength}\raggedright $\mathcal{F}^{i}$\end{minipage}}}%
    \put(-0.00049041,0.20160683){\color[rgb]{0,0,0}\makebox(0,0)[lt]{\begin{minipage}{0.14114131\unitlength}\raggedright $Q^i(-, P_i)$\end{minipage}}}%
    \put(0,0){\includegraphics[width=\unitlength,page=3]{uniquenessOfMinimizers.pdf}}%
    \put(0.49561076,0.02012296){\color[rgb]{0,0,0}\makebox(0,0)[lt]{\begin{minipage}{0.0978177\unitlength}\raggedright $\gamma$\end{minipage}}}%
    \put(0,0){\includegraphics[width=\unitlength,page=4]{uniquenessOfMinimizers.pdf}}%
    \put(0.57946402,0.02532179){\color[rgb]{0,0,0}\makebox(0,0)[lt]{\begin{minipage}{0.18020958\unitlength}\raggedright $\mathcal{F}^{i}$\end{minipage}}}%
    \put(0.31143939,0.20160683){\color[rgb]{0,0,0}\makebox(0,0)[lt]{\begin{minipage}{0.14114131\unitlength}\raggedright $Q^i(-, P_i)$\end{minipage}}}%
    \put(0,0){\includegraphics[width=\unitlength,page=5]{uniquenessOfMinimizers.pdf}}%
    \put(0.80754053,0.02012296){\color[rgb]{0,0,0}\makebox(0,0)[lt]{\begin{minipage}{0.0978177\unitlength}\raggedright $\gamma$\end{minipage}}}%
    \put(0,0){\includegraphics[width=\unitlength,page=6]{uniquenessOfMinimizers.pdf}}%
    \put(0.89139379,0.02532179){\color[rgb]{0,0,0}\makebox(0,0)[lt]{\begin{minipage}{0.18020958\unitlength}\raggedright $\mathcal{F}^{i}$\end{minipage}}}%
    \put(0.62336916,0.20160683){\color[rgb]{0,0,0}\makebox(0,0)[lt]{\begin{minipage}{0.14114131\unitlength}\raggedright $Q^i(-, P_i)$\end{minipage}}}%
    \put(0,0){\includegraphics[width=\unitlength,page=7]{uniquenessOfMinimizers.pdf}}%
    \put(0.29024966,0.13316855){\color[rgb]{0,0,0}\makebox(0,0)[lt]{\begin{minipage}{0.31657116\unitlength}\raggedright $S_{P_i}^{Q^i}(\gamma)$\end{minipage}}}%
    \put(0.60518964,0.10711411){\color[rgb]{0,0,0}\makebox(0,0)[lt]{\begin{minipage}{0.31657116\unitlength}\raggedright $S_{P_i}^{Q^i}(\gamma)$\end{minipage}}}%
    \put(0,0){\includegraphics[width=\unitlength,page=8]{uniquenessOfMinimizers.pdf}}%
    \put(0.90534797,0.07717752){\color[rgb]{0,0,0}\makebox(0,0)[lt]{\begin{minipage}{0.31657116\unitlength}\raggedright $S_{P_i}^{Q^i}(\gamma)$\end{minipage}}}%
    \put(0,0){\includegraphics[width=\unitlength,page=9]{uniquenessOfMinimizers.pdf}}%
  \end{picture}%
\endgroup%

%% file: A_numerical_measure_of_the_instability_of_Mapper.bbl
\begin{thebibliography}{10}

\bibitem{Alagappan2012}
M.~Alagappan.
\newblock From 5 to 13: Redefining the positions in basketball.
\newblock {\em MIT Sloan Sports Analytics Conference}, 2012.

\bibitem{BenDavid07}
S.~Ben-David.
\newblock A framework for statistical clustering with constant time
  approximation algorithms for k-median and k-means clustering.
\newblock {\em Machine Learning}, 66(2):243--257, Mar 2007.

\bibitem{Ben-David09}
S.~Ben-David and M.~Ackerman.
\newblock Measures of clustering quality: A working set of axioms for
  clustering.
\newblock In D.~Koller, D.~Schuurmans, Y.~Bengio, and L.~Bottou, editors, {\em
  Advances in Neural Information Processing Systems 21}, pages 121--128. Curran
  Associates, Inc., 2009.

\bibitem{BenDavidEtAl07}
S.~Ben-David, D.~P{\'a}l, and H.~U. Simon.
\newblock Stability of k-means clustering.
\newblock In N.~H. Bshouty and C.~Gentile, editors, {\em Learning Theory: 20th
  Annual Conference on Learning Theory, COLT 2007, San Diego, CA, USA; June
  13-15, 2007. Proceedings}, pages 20--34. Springer Berlin Heidelberg, 2007.

\bibitem{BenDavid_vonLuxburg08}
S.~Ben-David and U.~von Luxburg.
\newblock Relating clustering stability to properties of cluster boundaries.
\newblock In {\em COLT 2008}, pages 379--390, Madison, WI, USA, July 2008.
  Max-Planck-Gesellschaft, Omnipress.

\bibitem{BenDavidEtAl06}
S.~Ben-David, U.~von Luxburg, and D.~P{\'a}l.
\newblock A sober look at clustering stability.
\newblock In G.~Lugosi and H.~U. Simon, editors, {\em Learning Theory: 19th
  Annual Conference on Learning Theory, COLT 2006, Pittsburgh, PA, USA, June
  22-25, 2006. Proceedings}, pages 5--19. Springer Berlin Heidelberg, 2006.

\bibitem{BenHur02}
A.~Ben-Hur, A.~Elisseeff, and I~Guyon.
\newblock A stability based method for discovering structure in clustered data.
\newblock {\em Pacific Symposium on Biocomputing. Pacific Symposium on
  Biocomputing}, pages 6--17, 2002.

\bibitem{Bittner2000}
M.~Bittner~et al.
\newblock Molecular classification of cutaneous malignant melanoma by gene
  expression profiling.
\newblock {\em Nature}, 406:536 EP --, 2000.

\bibitem{Bowman2008}
G.~R. Bowman, X.~Huang, Y.~Yao, J.~Sun, G.~Carlsson, L.~J. Guibas, and V.~S.
  Pande.
\newblock Structural insight into rna hairpin folding intermediates.
\newblock {\em JACS Communications, pp}, pages 9676--9678, 2008.

\bibitem{Breiman1996}
L.~Breiman.
\newblock Bagging predictors.
\newblock {\em Machine Learning}, 24(2):123--140, Aug 1996.

\bibitem{breiman1998}
L.~Breiman.
\newblock Arcing classifier (with discussion and a rejoinder by the author).
\newblock {\em Ann. Statist.}, 26(3):801--849, 06 1998.

\bibitem{Camara2017}
P.~G. Camara.
\newblock Topological methods for genomics: Present and future direction.
\newblock {\em Current Opinion in Systems Biology}, 1:95--101, 2017.

\bibitem{Carlsson09}
G.~Carlsson.
\newblock Topology and data.
\newblock {\em Bull. Amer. Math. Soc. (N.S.)}, 46(2):255--308, 2009.

\bibitem{Carlsson2017}
G~Carlsson.
\newblock The shape of biomedical data.
\newblock {\em Current Opinion in Systems Biology}, 1, 2017.

\bibitem{Carlsson10}
G.~Carlsson and F.~M\'{e}moli.
\newblock Characterization, stability and convergence of hierarchical
  clustering methods.
\newblock {\em Journal of Machine Learning Research}, 11:1425--1470, 04 2010.

\bibitem{Carriere_17}
M.~Carri{\`e}re, B.~Michel, and S.~Oudot.
\newblock Statistical analysis and parameter selection for mapper.
\newblock {\em Journal of Machine Learning Research}, 19, 06 2017.

\bibitem{Carrire_Oudot16}
M.~Carri{\`e}re and S.~Oudot.
\newblock {Structure and Stability of the 1-Dimensional Mapper}.
\newblock In S{\'a}ndor Fekete and Anna Lubiw, editors, {\em 32nd International
  Symposium on Computational Geometry (SoCG 2016)}, volume~51 of {\em Leibniz
  International Proceedings in Informatics (LIPIcs)}, pages 25:1--25:16,
  Dagstuhl, Germany, 2016. Schloss Dagstuhl--Leibniz-Zentrum fuer Informatik.

\bibitem{Cecco2015}
L.~D. Cecco~et al.
\newblock Head and neck cancer subtypes with biological and clinical relevance:
  Meta-analysis of gene-expression data.
\newblock {\em Oncotarget}, 6:9627--9642, 2015.

\bibitem{Chan2013}
J.~M. Chan, G.~Carlsson, and R.~Rabadana.
\newblock Topology of viral evolution.
\newblock {\em Proceedings of the National Academy of Science}, 110, 2013.

\bibitem{Chang2013}
J.~Chang, M.~M. Nicolau, T.~R. Cox, D.~Wetterskog, J.~W. Martens, H.~E. Barker,
  and J.~T. Erler.
\newblock Loxll2 induces aberrant acinar morphogenesis via erbb2 signaling.
\newblock {\em Breast Cancer Research}, 15, 2013.

\bibitem{Dey16}
K.~T. Dey, F.~M{\'{e}}moli, and Wang Y.
\newblock Multiscale mapper: Topological summarization via codomain covers.
\newblock In {\em {SODA}}, pages 997--1013. {SIAM}, 2016.

\bibitem{Dey17}
K.~T. Dey, F.~M{\'{e}}moli, and Wang Y.
\newblock Topological analysis of nerves, reeb spaces, mappers, and multiscale
  mappers.
\newblock In {\em Symposium on Computational Geometry}, volume~77 of {\em
  LIPIcs}, pages 36:1--36:16. Schloss Dagstuhl - Leibniz-Zentrum fuer
  Informatik, 2017.

\bibitem{Dlotko19}
P.~{D{\l}otko}.
\newblock {Ball mapper: a shape summary for topological data analysis}.
\newblock {\em arXiv e-prints}, January 2019.

\bibitem{Duponchel2018a}
L.~Duponchel.
\newblock Exploring hyperspectral imaging data sets with topological data
  analysis.
\newblock {\em Analytica Chimica Acta}, 1000:123--131, 2018.

\bibitem{Duponchel2018b}
L.~Duponchel.
\newblock When remote sensing meets topological data analysis.
\newblock {\em Journal of Spectral Imaging}, 2018.

\bibitem{KeplerMapper}
J.~Hendrik and S.~Nathaniel.
\newblock Keplermapper.
\newblock http://doi.org/10.5281/zenodo.1054444, nov 2017.

\bibitem{Timothy2015}
T.~S.~C. Hinks~et al.
\newblock Innate and adaptive t cells in asthmatic patients: Relationship to
  severity and disease mechanisms.
\newblock {\em Journal of Allergy and Clinical Immunology}, 136(2):323--333,
  2015.

\bibitem{Hinks2016}
T.~S.~C. Hinks~et al.
\newblock Multidimensional endotyping in patients with severe asthma reveals
  inflammatory heterogeneity in matrix metalloproteinases and chitinase 3-like
  protein 1.
\newblock {\em J.Allergy Clin Immunol}, 138(1), 2016.

\bibitem{Jeitziner17}
R.~Jeitziner, M.~Carri\'{e}re, J.~Rougemont, S.~Oudot, K.~Hess, and C.~Brisken.
\newblock {Two-Tier Mapper: a user-independent clustering method for global
  gene expression analysis based on topology}.
\newblock {\em arXiv e-prints}, December 2017.

\bibitem{Kamruzzaman2017}
M.~Kamruzzaman, A.~Kalyanaraman, B.~Krishnamoorthy, and P.~Schnable.
\newblock Toward a scalable exploratory framework for complex high-dimensional
  phenomics data.
\newblock {\em arXiv e-prints}, 2017.

\bibitem{Kleinberg03}
J.~M. Kleinberg.
\newblock An impossibility theorem for clustering.
\newblock In S.~Becker, S.~Thrun, and K.~Obermayer, editors, {\em Advances in
  Neural Information Processing Systems 15}, pages 463--470. MIT Press, 2003.

\bibitem{LeeEtAl17}
Y.~Lee, S.~D. Barthel, P.~D{\l}otko, S.~M. Moosavi, K.~Hess, and B.~Smit.
\newblock Quantifying similarity of pore-geometry in nanoporous materials.
\newblock {\em Nature Communications}, 8:15396, May 2017.

\bibitem{Levine2001}
E.~Levine and E.~Domany.
\newblock Resampling method for unsupervised estimation of cluster validity.
\newblock {\em Neural Computation}, 13(11):2573--2593, 2001.

\bibitem{Li2015}
L.~Li, W.~Cheng, B.~S. Glicksberg, O.~Gottesman, R.~Tamler, R.~Chen, E.~P.
  Bottinger, and J.~T. Dudley.
\newblock Identification of type 2 diabetes sub-groups through topological
  analysis of patient similarity.
\newblock {\em Science Translational Medicine}, 7(311), 2015.

\bibitem{Pek2013}
P.~Y. Lum~et al.
\newblock Extracting insights from the shape of complex data using topology.
\newblock {\em Scientific Reports}, 3(1236), 2013.

\bibitem{Meila05}
M.~Meil\v{a}.
\newblock Comparing clusterings: An axiomatic view.
\newblock In {\em Proceedings of the 22Nd International Conference on Machine
  Learning}, ICML '05, pages 577--584, New York, NY, USA, 2005. ACM.

\bibitem{Monica2011}
N.~Monica, A.~J. Levine, and G.~Carlsson.
\newblock Topology based data analysis identifies a subgroup of breast cancers
  with a unique mutational profile and excellent survival.
\newblock {\em Proceedings of the National Academy of Sciences of the United
  States of America}, 108(17):7265--7270, 2011.

\bibitem{PythonMapper}
D.~M{\"{u}}llner and A.~Babu.
\newblock Python mapper: An open-source toolchain for data exploration,
  analysis and visualization.
\newblock http://danifold.net/mapper, 2013.

\bibitem{TDAmapper}
D.~M{\"{u}}llner, P.~Pearson, and G.~Singh.
\newblock {TDA} mapper.
\newblock https://github.com/paultpearson/TDAmapper, May 2010.

\bibitem{Nielson2015}
J.~L. Nielson~et al.
\newblock Topological data analysis for discovery in preclinical spinal cord
  injury and traumatic brain injury.
\newblock {\em Nature Communications}, 6:8581+, 2015.

\bibitem{OstrovskyEtAl06}
R.~Ostrovsky, Y.~Rabani, L.~Schulman, and C.~Swamy.
\newblock The effectiveness of lloyd-type methods for the k-means problem.
\newblock In {\em 2006 47th Annual {IEEE} Symposium on Foundations of Computer
  Science (FOCS'06)}. {IEEE}, 2006.

\bibitem{Pirashvili2018}
M.~Pirashvili, L.~Steinberg, F.~Belch\'{\i}~Guillamon, M.~Niranjan, J.~G. Frey,
  and J.~Brodzki.
\newblock Improved understanding of aqueous solubility modeling through
  topological data analysis.
\newblock {\em Journal of Cheminformatics}, 10(1):54, Nov 2018.

\bibitem{Riihimaki2019}
H.~Riihimaki, W.~Chacholski, J.~Theorell, J.~Hillert, and R.~Ramanujam.
\newblock A topological data analysis based classification method for multiple
  measurement.
\newblock {\em arXiv e-prints}, 03 2019.

\bibitem{Rizvi2017}
A.~H. Rizvi, P.~G. Camara, E.~K. Kandror, T.~J. Roberts, I.~Schieren,
  T.~Maniatis, and R.~Rabadan.
\newblock Single-cell topological rna-seq analysis reveals insights into
  cellular differentiation and development.
\newblock {\em Nature Biotechnology}, 35, 2017.

\bibitem{Nicolau2014}
D.~Romano~et al.
\newblock Topological methods reveal high and low functioning neuro-phenotypes
  within fragile x syndrome.
\newblock {\em Human Brain Mapping}, 35:4904--4915, 2014.

\bibitem{Rucco2015}
M.~Rucco, E.~Merelli, D.~Herman, D.~Ramanan, T.~Petrossian, L.~Falsetti,
  C.~Nitti, and A.~Salvi.
\newblock Using topological data analysis for diagnosis pulmonary embolism.
\newblock {\em Journal of Theoretical and Applied Computer Science}, 9:41--55,
  2015.

\bibitem{Sarikonda2014}
G.~Sarikonda~et al.
\newblock Cd8 t-cell reactivity to islet antigens is unique to type 1 while cd4
  t-cell reactivity exists in both type 1 and type 2 diabetes.
\newblock {\em Journal of Autoimmunity}, 50(Supplement C):77--82, 2014.

\bibitem{Savir2017}
A.~Savir, G.~Toth, and L.~Duponchel.
\newblock Topological data analysis (tda) applied to reveal pedoge- netic
  principles of european topsoil system.
\newblock {\em Science of the Total Environment}, 586(2):1091--1100, 2017.

\bibitem{Schofield2019}
J.~P.~R. Schofield~et al.
\newblock Stratification of asthma phenotypes by airway proteomic signatures.
\newblock {\em Journal of Allergy and Clinical Immunology}, 2019.

\bibitem{Shamir-Tishby10}
O.~Shamir and N.~Tishby.
\newblock Stability and model selection in k-means clustering.
\newblock {\em Machine Learning}, 80(2):213--243, Sep 2010.

\bibitem{Singh-Memoli-Carlsson07_foundationalMapperPaper}
G.~Singh, F.~M{\'e}moli, and G.~Carlsson.
\newblock Topological methods for the analysis of high dimensional data sets
  and 3{D} object recognition.
\newblock In {\em SPBG}. The Eurographics Association, 2007.

\bibitem{Strazzeri18}
F.~Strazzeri and R.~J. S{\'{a}}nchez{-}Garc{\'{\i}}a.
\newblock Morse theory and an impossibility theorem for graph clustering.
\newblock {\em CoRR}, abs/1806.06142, 2018.

\bibitem{Schneider2016}
B.~Y. Torres, J.~H.~O. Oliveira, A.~T. Tate, R.~Poonam, K.~Cumnock, and D.~S.
  Schneider.
\newblock Tracking resilience to infections by mapping disease space.
\newblock {\em PLOS Biology}, 14(6):e1002436, 2016.

\bibitem{vonLuxburg10_ClustStability_overview}
U.~von Luxburg.
\newblock Clustering stability: An overview.
\newblock {\em Found. Trends Mach. Learn.}, 2(3):235--274, March 2010.

\bibitem{vonLuxburgEtAl08}
U.~von Luxburg, S.~Bubeck, S.~Jegelka, and M.~Kaufmann.
\newblock Consistent minimization of clustering objective functions.
\newblock In {\em Advances in Neural Information Processing Systems 20}, pages
  961--968, Red Hook, NY, USA, September 2008. Max-Planck-Gesellschaft, Curran.

\end{thebibliography}
